\documentclass{article}
\usepackage{amsthm,amscd}
\usepackage{amsmath,amssymb,bm,graphicx,color}
\usepackage{thmtools}
\usepackage{mycommands}
\usepackage{authblk}
\usepackage{hyperref}
\date{}

\begin{document}

\title{Exact, graded, immersed Lagrangians and Floer theory}

\author{Garrett Alston and Erkao Bao}

\maketitle

\begin{abstract}
We develop Lagrangian Floer Theory for exact, graded, immersed Lagrangians with clean self-intersection using Seidel's setup \cite{seidel-fcpclt}. 
A positivity assumption on the index of the self intersection points is imposed to rule out certain (but not all) disc bubbles. 
This allows the Lagrangians to be included in the exact Fukaya category. 
We also study quasi-isomorphism of Lagrangians under certain exact deformations which are not Hamiltonian.
\end{abstract}

\tableofcontents{}

\section{Introduction}
\label{sec:introduction}
Exact, embedded Lagrangians are involved in many interesting classical questions in symplectic geometry.
They also constitute the class of objects in the  exact Fukaya category of a symplectic manifold.
In general, exact Lagrangians are rather rigid objects, and existence of even one has strong implications.
On the other hand, immersed exact Lagrangians satisfy an h-principle and always exist in abundance.
The goal of this paper is two-fold: one, to show that exact, graded, immersed Lagrangians with clean self-intersection, and which satisfy a certain positivity condition, can be admitted as objects of the exact Fukaya category; and two, to examine some of the invariance properties of these objects in the Fukaya category.

Immersed Floer theory was first developed by Akaho \cite{MR2155230}, and Akaho and Joyce \cite{MR2785840}.
In the context of curves on Riemann surfaces it also appears in work of Seidel \cite{MR2819674} and Abouzaid \cite{MR2383898}.
Immersed Floer theory also plays a central role in Sheridan's work on homological mirror symmetry in \cite{2011arXiv1111.0632S} and \cite{MR2863919}.

The results of this paper fall somewhere between \cite{MR2155230} and \cite{MR2785840}:
Our theory is more general than \cite{MR2155230}, but less general than \cite{MR2785840}. 
In \cite{MR2155230}, a topological condition is imposed that rules out the existence of all holomorphic discs; in this paper, we impose a positivity condition that only rules out some discs.
As a consequence, the Floer differential can have a non-topological part.
The $A_\infty$ algebras however have $\m{0}=0$.
What is gained over \cite{MR2785840} is that we work over $\zz_2$ instead of the Novikov ring, and we employ the less abstract and more computable perturbation scheme of Seidel \cite{seidel-fcpclt}.
One could work over $\zz$ if one wanted to deal with orientation issues.
The positivity condition rules out bad degenerations of holomorphic curves that would need a more advanced perturbation scheme to deal with.
Our theory allows immersions with clean, double-point self-intersection; and also transverse self-intersections without the double point restriction.
More generally, the theory works for arbitrary immersions with clean self-intersection as long as there exists a totally geodesic metric on the image of the immersion.
Also, inspired by \cite{MR2785840} and \cite{MR2863919}, we prove invariance under certain exact deformations which are not Hamiltonian.

We would like to also mention our debt to the ideas and work of Fukaya, Oh, Ohta and Ono \cite{fooo}.

\subsection{Setup and results}
We start with $(\mnfld,\sympl,\sigma,\jstd,\hvf)$ where $\mnfld$ is a compact manifold with boundary, $\sympl=d\oneform$ is an exact symplectic form, $\jstd$ is an $\omega$-compatible almost complex structure, and $\hvf$ is a nowhere-vanishing section of $\topext_\cc (T^*\mnfld,\jstd)$.
Actually, $\hvf$ only needs to be well-defined up to a $\pm$ sign.
Following Section (7a) of \cite{seidel-fcpclt}, the following convexity conditions are imposed: the Liouville vector field $X_{\oneform}$ defined by $\sympl(X_{\oneform},\cdot)=\oneform$ points strictly outwards along $\bdy M$, and any $\jstd$-holomorphic curve touching $\bdy M$ is completely contained in $\bdy M$.
In particular, $(\bdy\mnfld,\sigma)$ is a contact manifold.
If $\jstd$ is compatible with the contact structure near the boundary then it automatically satisfies the convexity condition.
In this case, one can allow $\mnfld$ to be noncompact by attaching a conical end to the  boundary and taking $\jstd$ to be cylindrical on the end.
Up to homotopy equivalence, the Floer cohomology we will construct does not depend on $\jstd$.
Also, $\hvf$ is used for grading puposes and so the constructions only depend on the homotopy equivalence class of $\hvf^{\otimes 2}$.

Let $\lag$ be a closed connected manifold and $\imm:\lag\to\mnfld$ a Lagrangian immersion, and let $\imlag=\imm(\lag)\subset\mnfld$.
For simplicity, we primarily focus on the Floer theory of the single Lagrangian immersion $\imm$, but note that with some obvious modifications the theory can be applied to any number of Lagrangian immersions.
Assume that $\imm$ is exact in the sense that there exists $\lagprim:\lag\to\rr$ such that $d\lagprim = \imm^*\sigma$.
Also, assume that $\lag$ is graded in the sense that there exists $\laggrad:\lag\to\rr$ such that $e^{2\pi i \laggrad}=\phase\circ D\imm$.
Here, $\phase$ is the (squared) phase function, defined by $\hvf$ in the following way: for $X_1,\ldots,X_n$ a basis of a Lagrangian plane in $T\mnfld$,
\begin{displaymath}
  \phase(X_1\w\cdots\w X_n)=\frac{\hvf(X_1\w\cdots\w X_n)^2}{|\hvf(X_1\w\cdots\w X_n)|^2}\in S^1.
\end{displaymath}

Let
\begin{equation}
  \label{eq:1}
  \brjmps=\set{(p,q)\in\lag\times\lag}{\imm(p)=\imm(q),\ p\neq q},\quad \imbrjmps=\imm(\brjmps)\subset \imlag.
\end{equation}
(Abusing notation, $\imm$ also denotes the obvious map $\brjmps\to\mnfld$.)
We call $\brjmps$ the set of \textit{branch jump types} and $\imbrjmps$ the \textit{branch points} or \textit{self-intersection points} of $\imlag$.
For the time being we assume that $\imm$ has only transverse double points; later on we will allow $\imm$ to have clean self-intersection.
It is interesting to note that 
\begin{equation}
  \label{eq:41}
  \lag\amalg\brjmps = \lag\times_{\imlag}\lag =\set{(x,y)\in\lag\times\lag}{\imm(x)=\imm(y)}.
\end{equation}
In the clean intersection case, this serves to define a topology on $\brjmps$.
An important concept that is a consequence of exactness is the notion of energy:
\begin{definition}
  \label{dfn:11}
The \textit{energy} $E(p,q)$ of $(p,q)\in\brjmps$ is $E(p,q)=-\lagprim(p)+\lagprim(q)$.
\end{definition}

In addition to graded and exact, the main condition we impose on the Lagrangian is a \textit{positivity condition}.
This condition is needed to control disc bubbling.
\begin{condition}
  \label{ass:1}
  If $(p,q)\in \brjmps$ and $E(p,q)>0$ then $\ind (p,q)\ge 3$.
  Here, $\ind (p,q)$ is a Maslov type index and is defined in Definition \ref{dfn:8}.
\end{condition}

From Definition \ref{dfn:8} and Definition \ref{dfn:11} it can be seen that $E(p,q)=-E(q,p)$ and $\ind (q,p)=\dim L-\ind (p,q)$.
In Section \ref{sec:poss-gener} we discuss the possibility of weakening the positivity condition.

Now we explain our results.
First, we need to choose some additional data. 
For a time dependent Hamiltonian function $\h=\hdep{t}{0\leq t\leq 1}$ with $H_t:M\to \rr$, denote the associated time dependent Hamiltonian vector field by $X_{\h}$ and the Hamiltonian flow by $\hamdiff{\h}{t}$.
(Our convention on Hamiltonian vector fields is $dH_t(\cdot)=\sympl(X_{H_t},\cdot)$.)
\begin{definition}
  \label{dfn:1}
  \textit{Floer data} $\h,\J$ for $\imm:\lag\to\imlag$ consists of
  \begin{itemize}
  \item a time dependent Hamiltonian $\h=\hdep{t}{0\leq t\leq 1}$ such that $H_t=0$ near the boundary of $M$, and
  \item a time dependent $\omega$-compatible almost complex structure $\J=\Jdep{t}{0\leq t\leq 1}$ such that $J_t=J_M$ near the boundary of $M$.
  \end{itemize}
\end{definition}

Let $\chords{\h}$ denote the set of Hamiltonian chords $\gamma:[0,1]\to\mnfld$ that start and stop on $\imlag$.
Assume that
\begin{itemize}
\item $\left[(\hamdiff{\h}{1})^{-1}(\imlag)\cup\hamdiff{\h}{1}(\imlag)\right]\cap \imbrjmps=\emptyset$, and
\item $\imlag$ is transverse to $\hamdiff{\h}{1}(\imlag)$.
\end{itemize}
The first point implies that no element of $\chords{\h}$ stops or starts on $\imbrjmps$ and the second implies that $\chords{\h}$ is a finite set.
When we want to emphasize these conditions we will refer to the Floer data as being \textit{admissible}, but since we generally impose admissibility we will usually not mention it.

Associated to this data, let $\cf(\imm)=\cf(\imm;\h,\J)=\zz_2\cdot\chords{\h}$ denote the $\zz_2$-vector space formally generated by the Hamiltonian chords.
Intuitively, if $H_t=H$ is time-independent and $C^2$-small, then $\cf(\imm)$ will have two generators for each self-intersection point and one generator for each critical point of the function $H\circ\imm$ (see Lemma \ref{lemma:11}).
Hence it is useful to imagine that $\cf(\imm)=\mathrm{CM}(\lag)\oplus \mathrm{CM}(\brjmps)$, where $\mathrm{CM}$ denotes the Morse complex.
Following \cite{seidel-fcpclt}, we define $A_\infty$ operations $\m{k}:\cf(\imm)^{\otimes k}\to\cf(\imm)$ for all $k\geq 1$ by counting inhomogeneous holomorphic curves with boundary marked points converging to Hamiltonian chords; this requires choosing other perturbation data in addition to the Floer data $\h,\J$.
$\delta:=\m{1}$ counts holomorphic strips and is the Floer differential.
The first main results are:
\begin{theorem}[See Sections \ref{sec:cfiota-textslh-texts} and \ref{sec:invariance-hfimm}]
  \label{thm:3}
  The Floer differential is well-defined and satisfies $\delta^2=0$. 
  Hence the Floer cohomology $\hf(\imm):=\Ker\delta/\Ima\delta$ is well-defined.
  Moreover, the cohomology is independent of the Floer data $\h,\J$.
\end{theorem}
\begin{theorem}[See Theorem \ref{thm:1} and Section \ref{sec:units-a_infty-categ}]
  \label{thm:4}
  The operations $\m{k}$ satisfy the $A_\infty$ relations.
  More generally, $\imm$ can be admitted as an object of the exact Fukaya category (as defined in \cite{seidel-fcpclt}), and the quasi-isomorphism class of $\imm$ is independent of the choices needed to include it as an object in the Fukaya category.
  In particular, the homotopy type of the $A_\infty$ algebra $(\cf(\imm),\sett{\m{k}})$ is an invariant of $\imm$.
\end{theorem}

Unlike the embedded case, it is not necessarily true that $\hf(\imm)\neq 0$; see Section \ref{sec:examples} for an example.
Also, it is not necessarily true that an exact deformation of an immersion $\imm$ is a Hamiltonian deformation.
A one-parameter family of immersions $\dep{\imm}{t}{t}$ with $\imm_0=\imm$ is an \textit{exact deformation} if the one form $\imm_t^*\sympl(\del\imm_t/\del t,\cdot)$ is exact; it is a \textit{Hamiltonian deformation} if $\imm_t=\hamdiff{}{t}\circ\imm$ for some family of Hamiltonian diffeomorphisms $\hamdiff{}{t}$. 
If $\dep{\imm}{t}{t}$ is a Hamiltonian deformation then it is exact,
but the converse is not true if $\imm$ is only an immersion.
Condition \ref{ass:1} is preserved under Hamiltonian deformation but not under general exact deformation (although if $E(p,q)\neq 0$ for all $(p,q)$, it is preserved for small enough exact deformation).
If $\imm_0$ is exact, then $\dep{\imm}{t}{t}$ is an exact deformation if and only if each $\imm_t$ is exact.
\begin{theorem}[See Theorem \ref{thm:12}]
  \label{thm:7}
  Let $\dep{\imm}{t}{t}$ be a (connected) family of exact graded Lagrangian immersions such that each $\imm_t$ satisfies Condition \ref{ass:1}.
  Moreover assume that each $\imm_t$ has only transverse double points.
  Then any two immersions in the family are quasi-isomorphic in the Fukaya category.
\end{theorem}

This theorem gives a partial answer to Question 13.18 in \cite{MR2785840}, which we now explain.
Exact deformations have an interesting interpretation in terms of Legendrians.
Observe that $\mnfld\times\rr$ is a contact manifold with contact form $\alpha=dz-\oneform$, where $z$ is the $\rr$-coordinate.
If $\leg\subset\mnfld\times\rr$ is an immersed Legendrian then the projection to $\mnfld$ is an immersed exact Lagrangian; conversely, an immersed exact Lagrangian lifts to an immersed Legendrian via the immersion $\imm\times\lagprim:\lag\to\mnfld\times\rr$.
The self-intersection points of $\imlag$ correspond to Reeb chords of $\leg$.
Moreover, $\leg$ is embedded if and only if $E(p,q)\neq 0$ for all $(p,q)\in\brjmps$.
Question 13.18 in \cite{MR2785840} asks if immersed Floer cohomology is invariant under an exact deformation which lifts to an isotopy of \textit{embedded} Legendrians (in general, exact deformation only lifts to isotopy of \textit{immersed} Legendrians).
Theorem \ref{thm:7} gives a partial affirmative answer to this: if $\imm_0$ satisfies the positivity condition and lifts to an embedded Legendrian, and $\dep{\imm}{t}{t}$ lifts to an isotopy of embedded Legendrians, then each $\imm_t$ must also satisfy the positivity condition; hence if the immersed Lagrangians have only transverse double points then the theorem says that the Floer cohomology is invariant.
(Note that Theorem \ref{thm:7} may still apply even if the isotopy of Legendrians is not an embedded isotopy.)

Now we discuss the clean self-intersection case.
See Definition \ref{dfn:23} for a precise definition; essentially the adjective clean means that the self-intersection (in $\imlag$) is a submanifold of $\mnfld$, with tangent space equal to the intersection of the tangent spaces of the branches.
Also, singular points are still only double points.
We define $\brjmps$ and $\imbrjmps$ in the same way as before; these are now smooth manifolds instead of just finite sets of points.
In particular, we can still ask if the positivity condition is satisfied.
In case it is, we can define the Floer theory in the same way as before.
\begin{theorem}[See Section \ref{sec:clean-self-inters-1}]
  \label{thm:8}
  Let $\imm$ be a graded, exact, clean immersion satisfying Condition \ref{ass:1}.
  Then $\imm$ can be included as an object of the exact Fukaya category, and the quasi-isomorphism class of $\imm$ is independent of the choices needed to include it.
\end{theorem}

In the previous theorem, if $\imlag=\imbrjmps$ then admissible Floer data does not exist because the endpoints of Hamiltonian chords necessarily intersect $\imbrjmps$.
To deal with this, we need to alter the definition of $\chords{\h}$ slightly: we should also include lifts to $\lag$ of the start and end points of the chords, see Definition \ref{dfn:22}.
Also, if $\imlag=\imbrjmps$, then $\lag\to\imlag$ is a double cover and $\imlag$ is exact.
In this situation, we do not need to require that $\imm$ has only double points, so consider $\imm:\lag\to\imlag$ any covering space, say $r$-fold.
Then $\imm$ defines a local system on $\imlag$ with fiber $(\zz_2)^r$ (see prior to Theorem \ref{thm:6}), call it $\locsys_\imm$, and $(\imlag,\locsys_\imm)$ can be thought of as an object of the embedded, exact Fukaya category with local systems.
We suspect the following theorem is essentially already known to many people.

\begin{theorem}[See Theorem \ref{thm:6}]
  \label{thm:10}
  $\imm$ and $(\imlag,\locsys_\imm)$ are quasi-isomorphic objects in the Fukaya category.
  The Floer-cohomology is isomorphic to the singular cohomology (with $\zz_2$-coefficients) of $\lag\amalg\brjmps=\lag\times_{\imlag}\lag$.
\end{theorem}

It is interesting to note that if $\lag$ is an embedded exact Lagrangian then $\lag\times_{\imlag}\lag=\lag$ and the Floer cohomology is isomorphic to the singular cohomology of $\lag\times_{\imlag}\lag$.
More generally, if $L_1$ and $L_2$ are embedded, exact Lagrangians, then their fiber product is $L_1\times_M L_2=L_1\cap L_2$ and there is a spectral sequence computing the Floer cohomology of the pair $(L_1,L_2)$ which starts at the singular cohomology of $L_1\cap L_2$ (assuming they intersect cleanly).
A similar spectral sequence should exist in the immersed case. (We do not investigate it in this paper.)

The immersions which are covering spaces $\lag\to\imlag$ have an obvious class of exact deformations.
Take a generic function $f$ on $\lag$, and deform $\imm$ in the direction of $-J\nabla f$.
These deformations may or may not satisfy the positivity condition; if they do then the result is a quasi-isomorphic object of the Fukaya category.
\begin{theorem}[See Theorem \ref{thm:13}]
  \label{thm:11}
  Let $\dep{\imm}{t}{|t|\leq \epsilon}$ be the deformation described above.
  Assume each $\imm_t$ satisfies Condition \ref{ass:1}, and each $\imm_t$ for $t\neq 0$ has transverse double points.
  Then any immersion in the family is quasi-isomorphic to $\imm$.
\end{theorem}

In particular, choosing different permissible functions $f$ could lead to quasi-isomorphic objects with differing numbers of self-intersection points.
These would be examples of exact deformations qualitatively different than those appearing in Theorem \ref{thm:7}.

In Section \ref{sec:more-than-double} we discuss the case of immersions with general clean self-intersection.
Our theory works for these more general immersions as long as they satisfy the positivity condition and the image of the immersion is totally geodesic with respect to some metric on $\mnfld$.
An example of such an immersion is one in which all self-intersections are transverse but not necessarily double points (i.e., they can be triple points, etc.).
The totally geodesic restriction is likely  a technical restriction that could be removed with a more elaborate setup.

\subsection{Method of proof}
The general framework we follow is the one devised by Seidel in \cite{seidel-fcpclt} for exact embedded Lagrangians.
The \textit{definition} of the Floer cohomology and the $A_\infty$ structure exactly parallels that given in \cite{seidel-fcpclt}.
However, because our Lagrangians are immersed and not embedded, holomorphic discs with corners at the self-intersection points might exist.
These in principle could prevent the $A_\infty$ relations from holding, so the extra work we need to do is to show that these discs do not not cause problems.

To deal with this, we define moduli spaces of curves with two types of boundary marked points, which we call Type I and Type II marked points.
Type I marked points are the kind appearing in \cite{seidel-fcpclt} and must converge to Hamiltonian chords.
Type II marked points converge to self-intersection points.

We also define two types of moduli spaces: holomorphic polygons, and holomorphic discs.
Holomorphic polygons have Type I and possibly Type II marked points and satisfy an inhomogeneous Cauchy-Riemann equation.
The holomorphic polygons with only Type I marked points are used to define the $A_\infty$ structure.
Holomorphic discs have only Type II marked points and satisfy a homogeneous Cauchy-Riemann equation; they appear attached to polygons in the compactification of the moduli space of polygons and serve only an auxiliary purpose; in particular, they do not enter into the definition of the $A_\infty$ structure.

The preliminary extra details needed beyond \cite{seidel-fcpclt} are a description of the analytic and Fredholm setup of the moduli spaces of polygons (including ones with Type II points), a description of the Gromov compactness of the two types of moduli spaces, and a statement on regularity of the moduli spaces of holomorphic polygons.
These mostly follow from standard facts already well represented in the literature.
With these in hand, the main type of statement to prove which is not needed in \cite{seidel-fcpclt} is the following: a zero- or one-dimensional moduli space of holomorphic polygons with no Type II points can be compactified without introducing any Type II points.
The positivity condition, Gromov compactness, and regularity are the essential ingredients needed to prove this statement.

\subsection{Examples}
\label{sec:examples}
The first example is an immersed sphere in $\cc^n$ with one self-intersection point, which is described in Example 13.12 in \cite{MR2785840}.
The sphere can be described as follows:
Start with the curve $C=\set{s+it\in\cc}{t^2=s^2-s^4}$; $C$ is an immersed circle in $\cc$ which looks like an $\infty$ symbol, and has one self-intersection point at the origin.
Let $\imlag=\set{(\lambda x_1,\ldots,\lambda x_n)}{\lambda\in C,\ x_1,\ldots,x_n\in\rr,\ x_1^2+\cdots+x_n^2=1}$.
Then $\imlag$ is an immersed sphere inside $\cc^n$ with a single self-intersection point at the origin.
The indices of the two elements of $\brjmps$ are $-1$ and $n+1$.
A simple calculation shows that the index $n+1$ branch jump has positive energy.
(This is easy for the $n=1$ case; for the higher dimensional case just consider one of the holomorphic discs bounded by $C\times \sett{0}^{n-1}\subset \imlag$.)
Thus the Lagrangian $\imlag$ satisfies the positivity condition when $n\geq 2$.
Since every compact Lagrangian in $\cc^n$ is displaceable, the Floer cohomology is $0$.

The next examples are Lagrangian spheres in (smoothings of) $A_N$ surfaces, described in \cite{2013arXiv1311.2327A}.
The $A_N$ surface is the hypersurface $X=\sett{xy+f_N(z)=0}\subset\cc^3$, where $f_N(z)$ is a degree $N$ polynomial with no repeated roots.
As symplectic manifolds, $A_1\cong \cc^2$ and $A_2\cong T^*S^2$.
The map $X\to\cc$, $(x,y,z)\mapsto z$ is a Lefschetz fibration with critical values at the roots of $f_N$.
Given an embedded path in the base $\cc$ that starts and stops at the same critical value and doesn't intersect any other critical values, the union of all the vanishing cycles over all the points in the base path is an immersed sphere with a single self-intersection point.
The indices of the branch jumps are $3$ and $-1$, and they satisfy the positivity condition.
We refer readers to \cite{2013arXiv1311.2327A} for details; also, the case $N=2$ is essentially the affine version of an example studied by Auroux in \cite{MR2386535}.
In \cite{2013arXiv1311.2327A}, a pearly version of the Floer theory described in this paper is used to calculate the Floer cohomology.
The result is that the cohomology has rank $4$ if the interior of the region bounded by the base path contains at least one critical value, and is $0$ otherwise.
All of the Lagrangians in the $N=1$ case fall into the latter category; they essentially coincide with the Lagrangian in $\cc^2$ described in the previous class of examples.

\subsection{Relation to previous works}
\label{sec:relat-prev-works}
In \cite{MR2155230}, Akaho defines Floer cohomology and the product for immersed Lagrangians that satisfy the condition $\pi_2(\mnfld,\imlag)=0$.
This a priori rules out all disc bubbles. 
Such Lagrangians automatically satisfy Condition \ref{ass:1}.

In \cite{MR2785840}, Akaho and Joyce define Floer theory for arbitrary (oriented, spin) immersed Lagrangians with only transverse double points.
Their theory is similar to the theory of Fukaya, Oh, Ohta, and Ono for embedded Lagrangians \cite{fooo}, and requires using the Novikov ring $\Lambda_{0,nov}$, and the virtual perturbation technique of Kuranishi spaces to deal with transversality issues.

Our theory does not require using Kuranishi spaces, and this makes it simpler and more computable (of course, the expense is loss in generality).
In addition to the virtual perturbation technique, another difficulty of the general theory is that it involves a combination of geometric and algebraic constructions to arrive at the final $A_\infty$ algebra; for example, the homological perturbation lemma is used several times along the way.
Also, at some point one has to deal with infinite dimensional chain complexes.
The methods of Seidel \cite{seidel-fcpclt} that we employ are more direct and purely geometric, and the chain complexes are finite dimensional from the beginning.

A local Floer cohomology for two cleanly intersecting embedded Lagrangians was studied by Pozniak \cite{Pozniak}.
It seems likely that a Morse-Bott version of \cite{MR2785840} can be defined to deal with Lagrangians that are cleanly immersed by mimicking the Morse-Bott version of Floer theory for two cleanly intersecting embedded Lagrangians in \cite{fooo}.

In the language of \cite{fooo} and \cite{MR2785840}, Lagrangians satisfying Condition \ref{ass:1} are automatically unobstructed; this is stated as Proposition 13.10 in \cite{MR2785840}.
This proposition has a weaker hypothesis than the positivity condition: it is simply required that there are no self-intersection points of index 2.
In the next section we will discuss the prospect of bringing our results in line with this.

We learned a great deal about immersed Floer theory from the work of Sheridan in \cite{MR2863919} and \cite{2011arXiv1111.0632S}.
The example that Sheridan deals with does not exactly fit into any of the frameworks discussed so far; rather, Sheridan deals with bubbling and compactness issues by working in a covering of the symplectic manifold in which the immersed Lagrangian lifts to an embedded Lagrangian.
Theorem \ref{thm:11} is inspired by a similar result of Sheridan for his example.

\subsection{Possible generalizations}
\label{sec:poss-gener}
We expect that Condition \ref{ass:1} can be relaxed to the following: if $\ind (p,q)=2$ then $E(p,q)\leq 0$.
The reason is that the virtual dimension of the moduli space of holomorphic discs with one (Type II) marked point of index less than 2 is negative.
Hence, generically, one expects no such discs to exist, and hence they shouldn't cause a problem with disc bubbling.
The index 2 discs are more problematic because (in the terminology of \cite{MR2785840}) they contribute to $\mathfrak{m}_0(1)$ and can be an obstruction to Floer cohomology being well-defined.
Thus they need to still be explicitly ruled out with a positivity condition.
In this paper we never address the regularity of holomorphic discs (we only require the moduli spaces of inhomogeneous holomorphic polygons to be regular) and therefore need Condition \ref{ass:1}. 
However, we expect that the results of Ekholm, Etnyre, and Sullivan \cite{MR2197142}, \cite{MR2299457} on Legendrian contact homology can be adapted to our setting to give the necessary regularity arguments.
We plan to investigate this in future work.

From a computational point of view, it would be beneficial to have a Morse-Bott version of the theory that does not require perturbing the Lagrangian by a Hamiltonian.
In this theory, one would take a Morse function on $L\amalg\brjmps$ and define the Floer cochain complex to be the Morse complex of $L$ plus the Morse complex of $\brjmps$.
Then the Floer differential, and more generally the $A_\infty$ structure, would be defined by counting ``pearly curves'', which are combinations of holomorphic discs and Morse trajectories.
In the case of embedded monotone Lagrangians, such a theory for Floer cohomology is developed by Biran and Cornea in \cite{MR2555932}.
Sheridan also develops such a theory for his immersed Lagrangians in \cite{2011arXiv1111.0632S} and \cite{MR2863919}.
In \cite{2013arXiv1311.2327A}, the first author described how such a theory works for Floer cohomology in the context discussed in this paper, and we plan to address the full $A_\infty$ structure in future work.
We remark that in \cite{2013arXiv1311.2327A}, a stronger positivity condition was imposed; namely that if $E(p,q)>0$ then $\ind(p,q)\geq (\dim \lag+3)/2$.
This stronger condition is needed in the pearly case to rule out degenerations resulting from nodal holomorphic curves which contain a constant disc component that maps to a self-intersection point; we do not know if it can be relaxed to the condition used in this paper.
(Such problematic degenerations are not a problem in this paper because they are ruled out by the use of Hamiltonian perturbations; see Remark \ref{rmk:7}.)

\subsection{Outline of paper}
Here is an outline of the paper.
First, a few general points.
Before Section \ref{sec:a_infty-structure-1} we focus solely on Floer cohomology and restrict attention to the case that $\imm$ has transverse double points.
In particular, in Section \ref{sec:floer-cohomology} we explain why the Floer differential is well-defined, and the Floer cohomology is invariant under the choice of Floer data.
The proofs are standard, but we still explain them carefully in order to show why the positivity condition is needed.
In Section \ref{sec:a_infty-structure-1} we consider the full $A_\infty$ structure, and in Section \ref{sec:clean-self-inters-1} we consider the clean immersion case.
Fewer details are given in these sections because the main points are similar to those discussed in detail in Section \ref{sec:floer-cohomology}.

Regarding moduli spaces of curves, we will consider two types of curves which we call polygons and discs.
In the literature these two terms are often synonymous, but for us they will have different meanings:
Polygons will be used to define the $A_\infty$ structure, and discs will be used to describe bubbling.
Before Section \ref{sec:a_infty-structure-1} the only polygons we will consider are strips because we will only discuss Floer cohomology and not the full $A_\infty$ structure.
The term curve is a catch-all term that can mean a strip, or other polygon, or a disc.

We begin in Section \ref{sec:marked-curves} by discussing the notion of marked strips and marked discs, and the Maslov index.
In Section \ref{sec:moduli-spac-analyt} we describe how the space of strips and discs with $W^{1,p}$ regularity can be given the structure of a Banach manifold, and we define the moduli spaces we will use.
In Sections \ref{sec:grom-comp-moduli} and \ref{sec:regularity} we discuss transversality and compactness of the moduli spaces.
Marked polygons (other than strips) will be introduced in Section \ref{sec:a_infty-structure-1} when we talk about the full $A_\infty$ structure.

In Section \ref{sec:more-than-double} we explain how we can relax the condition that self-intersection points are double points.
An appendix on asymptotic analysis is included for completion.

\section{Marked curves: topological definitions}
\label{sec:marked-curves}
In general, we will consider two different types of curves, which we will call polygons and discs.
Holomorphic polygons will be used to define the $A_\infty$ structure for a Lagrangian, and holomorphic discs will be used to describe disc bubbling.
Strips are special types of polygons (they can be thought of as 2-gons), and they are the curves used to define Floer cohomology.
We will be concerned only with Floer cohomology through Section \ref{sec:floer-cohomology}, so we postpone the discussion of general marked polygons to Section \ref{sec:marked-polygons}.
This section discusses the definitions and topological properties of strips and discs.
Holomorphic strips and discs will be considered in Section \ref{sec:moduli-spac-analyt}.

Let $\clrs$ denote a closed Riemann surface with boundary.
In particular, $\clrs$ has a complex structure $\jrs$, which is suppressed from the notation.
We will be interested in two types of marked boundary points which we call Type I and Type II points.
Let $\ptsi=(\zi_0,\ldots,\zi_k)\subset\bdy\clrs$
denote an ordered list of Type I marked points, and
$\ptsii=(\zii_1,\ldots,\zii_m)\subset\bdy\clrs$
denote an ordered list of Type II marked points.
Assume that all marked points are distinct.
Moreover, assume that each marked point is labeled as incoming $(-)$ or outgoing $(+)$; this determines decompositions $\ptsi=\ptsi^-\amalg\ptsi^+$ and $\ptsii=\ptsii^-\amalg\ptsii^+$.
Given $\clrs$ and $\ptsi$, let $\rs$ be the (non-compact) Riemann surface
$\rs=\clrs\setminus \ptsi$.
When we write this we really mean that $\clrs,\ptsi$ is part of the data of $\rs$.

One of the features of holomorphic curves with boundary on an immersed Lagrangian is that the curve can jump branches at a self-intersection point.
To deal with this type of behavior (and for other reasons), boundary lifts to $\lag$ need to be included as additional data of the curve.
The following definition will play an important role in describing the behavior of the lifts at the branch points.
\begin{definition}
  \label{dfn:5}
  Let $\gamma:(-\epsilon,\epsilon)\to\imlag$ be a path such that $\gamma(0)\in\imbrjmps$.
  Let $\tilde\gamma:(-\epsilon,0)\cup(0,\epsilon)\to \lag$ be a lift of $\gamma$ defined away from $0$, and let
  \begin{displaymath}
    p=\lim_{s\to 0^-}\tilde\gamma(s),\quad q=\lim_{s\to 0^+}\tilde\gamma(s).
  \end{displaymath}
  If $p\neq q$ then $(p,q)\in \brjmps$ and we say that the path $\tilde\gamma$ has a \textit{branch jump of type} $(p,q)$.
  If $p=q$ then $\tilde\gamma$ extends continuously over $0$ and $\gamma$ does not have a branch jump at $0$.
  Note that $\tilde\gamma$ is unique if it exists and $\gamma$ is not constant on $(-\epsilon,0)$ or $(0,\epsilon)$.
\end{definition}

\subsection{Marked strips}
\label{sec:mark-strips-polyg}
Let $\clrs=\disc=\sett{|z|\leq 1}$, $\ptsi^-=\sett{-1}$, $\ptsi^+=\sett{1}$ and $S=\clrs\setminus \ptsi$.
Identify $S$ with $\rr\times[0,1]$ and $\ptsi$ with $\sett{\pm\infty}$; we will call this Riemann surface the \textit{strip} and denote it by $\strip$.
Explicitly, use the standard identification
\begin{displaymath}
  z=s+it\in Z=\rr\times[0,1]\mapsto \frac{ie^{\pi z}+1}{ie^{\pi z}-1}\in S,\quad \pm\infty\mapsto\pm1.
\end{displaymath}

Fix a Hamiltonian $\h=\hdep{t}{t}$.
Recall that $\chords{\h}$ is the set of time-1 chords on $\imlag$.

\begin{definition}
  \label{dfn:2}
  A \textit{$C^0$-marked strip} $\markedstrip{u}=(u,\ptsii,\brjmptypes,\ell)$ \textit{connecting $\gamma_{\pm}\in\chords{\h}$} consists of
  \begin{itemize}
  \item a continuous map $u:(\strip,\bdy\strip)\to(\mnfld,\imlag)$ such that $\lim_{s\to\pm\infty}u(s,t)=\gamma_{\pm}(t)$ uniformly in $t$,
  \item a list of Type II marked points $\Delta\subset\bdy\strip$, all of which are considered outgoing, and are ordered starting from left to right along the bottom of the strip and continuing right to left along the top,
  \item a map $\brjmptypes:\sett{1,\ldots,|\Delta|}\to\brjmps$, and
  \item a continuous map $\lift:\bdy\strip\setminus\ptsii\to\lag$ such that
    $\imm\circ\lift=u|\bdy\strip\setminus\ptsii$.
  \end{itemize}
  Moreover, $\lift$ has a branch jump at each $\zii_i\in\ptsii$ of type $\brjmptypes(i)\in\brjmps$ (in the sense of Definition \ref{dfn:5}; moving counterclockwise around $\bdy\strip$ corresponds to increasing $s$ in the notation of the definition).
  Denote by $\strips{}{\strip,\ptsii,\brjmptypes}$ the space of all $C^0$-marked strips with Type II points $\ptsii$ ($\h$ is suppressed from the notation), and $\strips{}{\strip,\brjmptypes}=\coprod_{\ptsii}\strips{}{\strip,\ptsii,\brjmptypes}\times\sett{\ptsii}$.

\end{definition}

\subsection{Marked discs}
\label{sec:marked-discs}
Let $\clrs=\disc=\sett{|z|\leq 1}$.
\begin{definition}
  \label{dfn:12}
  A \textit{$C^0$-marked disc} $\markeddisc{u}=(u,\ptsii,\brjmptypes,\ell)$ consists of
  \begin{itemize}
  \item a continuous map $u:(\disc,\bdy\disc)\to(\mnfld,\imlag)$,
  \item a list of Type II marked points $\ptsii\subset\bdy\disc$ which are ordered in counterclockwise order, along with a decomposition $\ptsii=\ptsii^+\amalg\ptsii^-$ into outgoing and incoming points,
  \item a map $\brjmptypes:\sett{1,\ldots,|\Delta|}\to\brjmps$, and
  \item a continuous map $\lift:\bdy\disc\setminus\ptsii\to\lag$ such that $\imm\circ\lift=u|(\bdy\disc\setminus\ptsii)$.
  \end{itemize}
  Moreover, moving in the counterclockwise direction, $\lift$ has a branch jump at each $\zii_i\in\ptsii^+$ of type $\brjmptypes(i)\in\brjmps$; and, moving in the \textit{clockwise} direction, $\lift$ has a branch jump at each $\zii_i\in\ptsii^-$ of type $\brjmptypes(i)\in\brjmps$.
  See Definition \ref{dfn:5} for the definition of a branch jump type.

  Denote by $\discs{}{\disc,\ptsii,\brjmptypes}$
  the space of all $C^0$-marked discs, and $\discs{}{\disc,\brjmptypes}=\coprod_{\ptsii}\discs{}{\disc,\ptsii,\brjmptypes}\times\sett{\ptsii}$.
\end{definition}

\subsection{Maslov index}
\label{sec:maslov-index}
Let $\laggrass\to\mnfld$ be the fiber bundle of Lagrangian planes in $T\mnfld$.
Given $\Lambda_0, \Lambda_1\in\laggrass$ over the same point in $\mnfld$, with $\Lambda_0\cap\Lambda_1=\sett{0}$, and $\theta_0,\theta_1\in\rr$ such that $\phase(\Lambda_j)=e^{2\pi i \theta_j}$ for $j=0,1$, define the index of the pair of graded planes to be
\begin{equation}
  \label{eq:7}
  \ind ((\Lambda_0,\theta_0),(\Lambda_1,\theta_1))=n+\theta_1-\theta_0-2\cdot\Angle(\Lambda_0,\Lambda_1).
\end{equation}
Here  $n=\text{dim}M/2$ and 
 $\Angle(\Lambda_0,\Lambda_1)=\alpha_1+\cdots+\alpha_n$, where $\alpha_i\in (0,\frac{1}{2})$ are defined by requiring that there exists a unitary (with respect to $J_M$ and $\sympl$) basis $u_1,\ldots,u_n$ of $\Lambda_0$ such that
\begin{displaymath}
  \Lambda_1=\Span_\rr\sett{e^{2\pi i\alpha_1}u_1,\ldots,e^{2\pi i\alpha_n}u_n}.
\end{displaymath}
From Equation \eqref{eq:7} we can see that $$\ind ((\Lambda_0,\theta_0),(\Lambda_1,\theta_1))=n-\ind ((\Lambda_1,\theta_1),(\Lambda_0,\theta_0)).$$

Alternatively, the index can be defined in the following way.
Let $\lambda:[0,1]\to\laggrass$ be a continuous path of Lagrangians all lying over the same point in $\mnfld$.
Furthermore, assume $\lambda(0)=\Lambda_0$, $\lambda(1)=\Lambda_1$, and assume there exists a continuous family of real numbers $\theta_t$ from $\theta_0$ to $\theta_1$ such that $\phase(\lambda(t))=e^{2\pi i\theta_t}$.
Then $\ind ((\Lambda_0,\theta_0),(\Lambda_1,\theta_1))=n/2+\mu(\lambda,\Lambda_0).$
The $\mu$ on the right-hand side denotes the Maslov index of a path of Lagrangians with respect to a fixed Lagrangian, as defined in \cite{rs-mip}.
Equivalently, if $\lambda_{pos}$ denotes the path of Lagrangians which is the positive definite rotation from $\Lambda_1$ to $\Lambda_0$, then $\ind((\Lambda_0,\theta_0),(\Lambda_1,\theta_1))=\mu(\lambda\#\lambda_{pos})$.
Here the right-hand side denotes the Maslov index of a loop.

There is yet another definition of the index which connects it to the index theory of Cauchy-Riemann operators.
To describe this, let $p\in\mnfld$ be the point with $\Lambda_0,\Lambda_1\subset T_p\mnfld$.
Consider the unit disc $\disc$ with trivial symplectic vector bundle $E=\disc\times T_p\mnfld$.
Fix a marked point $z_0\in\bdy\disc$ (which should be thought of as an incoming marked point) and fix a parametrization $(0,1)\cong \bdy\disc\setminus\sett{z_0}$ (with the arc given the counterclockwise orientation).
Define a Lagrangian sub-bundle $F$ of $E$ over this arc by $F_t=\lambda(t)$ (if necessary, first homotope $\lambda$ slightly so that it is constant near $t=0$ and $t=1$).
Let $\dbar_{(E,F)}$ denote the standard Cauchy-Riemann operator of this bundle pair; the domain is $W^{1,2}$ sections of $E$ with $F$ boundary conditions, and the target is $L^2$ sections of $\Lambda^{0,1}_\disc\otimes E$.
Then $\ind ((\Lambda_0,\theta_0),(\Lambda_1,\theta_1))$ is equal to the Fredholm index of the operator $\dbar_{(E,F)}$.
See equation (11.20) and Lemma 11.11 in \cite{seidel-fcpclt} for a proof.

We now describe how to define an index for Hamiltonian chords.
Let $\h=\hdep{t}{t}$ be an admissible Hamiltonian, $\chords{\h}$ be the set of time-1 Hamiltonian chords, and $\hamdiff{\h}{t}$ be the flow of the Hamiltonian vector field $X_{\h}$.
For $\gamma\in\chords{\h}$, let $\Lambda_t\in\laggrass$, $\theta_t\in\rr$ be continuous in $t\in[0,1]$ and such that
\begin{displaymath}
  (\Lambda_0,\theta_0)=(T_{\gamma(0)}\imlag,\laggrad(\imm^{-1}(\gamma(0)))),\quad D(\hamdiff{\h}{t})\cdot\Lambda_0=\Lambda_{t},\quad \phase(\Lambda_t)=e^{2\pi i \theta_t}.
\end{displaymath}
In other words, $(\Lambda_0,\theta_0)$ is the graded plane determined by the Lagrangian at $\chord(0)$, and the linearized flow $D\hamdiff{\h}{t}$ is used to move it to a graded plane $(\Lambda_1,\theta_1)$ contained in $T_{\chord(1)}\mnfld$.
\begin{definition}
  \label{dfn:7}
  The index of $\gamma\in\chords{\h}$ is
  \begin{displaymath}
    \ind\gamma = \ind((\Lambda_1,\theta_1),(T_{\gamma(1)}\imlag,\laggrad(\imm^{-1}(\gamma(1))))).
  \end{displaymath}
\end{definition}

This coincides with the definition given in Section (12b) of \cite{seidel-fcpclt}.
To motivate it, consider the Lagrangians $L_0=\hamdiff{\h}{1}(\imlag)$ and $L_1=\imlag$.
Hamiltonian chords on $\imlag$ correspond to points of $L_0\cap L_1$ via $\gamma\leftrightarrow \gamma(1)$.
The Lagrangian $L_0$ inherits a grading from the grading of $\imlag$ and the isotopy $t\mapsto \hamdiff{\h}{t}(\imlag)$, $0\leq t\leq 1$.
In Definition \ref{dfn:7}, $(\Lambda_1,\theta_1)$ is the graded Lagrangian plane determined by the tangent space of $L_0$ at $\gamma(1)$.

Following \cite{MR2785840}, the grading of $\lag$ also allows us to assign an index to elements of $\brjmps$.
\begin{definition}
  \label{dfn:8}
  Let $(p,q)\in\brjmps$.
  Then the index of $(p,q)$ is
  \begin{displaymath}
    \ind (p,q) = \ind ((\imm_* T_p\lag,\laggrad(p)),(\imm_*T_q\lag,\laggrad(q))).
  \end{displaymath}
\end{definition}

Now we define the Maslov index for marked curves.
Let $\markedstrip{u}=\markedstripstd$ be a marked curve with domain $\rs=\clrs\setminus\ptsi$ and Type II points $\ptsii$.
Let $S'=\rs\cup \ptsi\times[0,1]$ be the compact surface with corners obtained from $\clrs$ by replacing each point in $\ptsi$ with a unit interval; topologically $S'$ is a closed disc.
For instance, if $\rs=\strip=\rr\times[0,1]$, then $S'=[-\infty,\infty]\times[0,1]$ which can be thought of as a compact rectangle.
Also, $\sett{z}\times[0,1]\subset\bdy S'$ is assumed to be  parameterized in such a way that motion from $0$ to $1$ in $\sett{z}\times[0,1]$ corresponds to counterclockwise motion along $\bdy S'$ if $z$ is an outgoing point, and clockwise motion if $z$ is an incoming point.
(Note that with the strip $-\infty$ is an incoming point and $+\infty$ is an outgoing point.)
Since $\markedstrip{u}$ converges to chords at points of $\ptsi$, it can be thought of as a continuous map into $\mnfld$ with domain $S'$.
Over $S'$, $u^*T\mnfld$ defines a symplectic vector bundle.
On the boundary of $S'$ and away from $\ptsii$ and $\ptsi\times[0,1]$, $\imm_*\ell^*T\lag$ defines a Lagrangian sub-bundle.
Extend the sub-bundle over $\ptsi\times[0,1]\amalg\ptsii\subset \bdy S'$ as follows (in the following, a positive definite rotation is the rotation $t\mapsto e^{tJ}$ for appropriate $\omega$-compatible $J$):
\begin{itemize}
\item For an outgoing Type I point $z^+\in\ptsi$ that converges to $\chord\in\chords{\h}$, we want the Lagrangian subbundle over $\sett{z^+}\times[0,1]\subset\bdy S'$ to be the path of Lagrangians that starts with $T_{\chord(0)}\imlag$, then moves along the linearized flow $D\hamdiff{\h}{t}$ until a plane in $T_{\chord(1)}\mnfld$ is reached, and then is followed by a positive definite rotation from this plane to $T_{\chord(1)}\imlag$. 
Forward progress along this path should correspond to moving in the counterclockwise direction along $\bdy S'$.
More precisely, we do the following.
Homotope $u^*TM$ slightly so that it is constant in a small neighborhood of $\sett{z^+}\times\sett{1}\in \sett{z^+}\times[0,1]\subset S'$.
This gives a a trivialization in a neighborhood of this point.
Let $\epsilon>0$ be small.
Then, for $t\in [0,1-\epsilon]$, the fiber of the subbundle over $\sett{z^+}\times\sett{t}\in\sett{z^+}\times[0,1]$ is $D\hamdiff{\h}{t}(T_{\chord(0)}\imlag)$.
For $1-\epsilon\leq t\leq 1$, move in the positive definite direction until $T_{\chord(1)}\imlag$ is reached.
The result is that the Lagrangian sub-bundle has been extended over $\sett{z^+}\times[0,1]\subset \bdy S'$.
We emphasize that since $z^+$ is an outgoing point, moving in the forward direction of the interval $\sett{z^+}\times[0,1]$ corresponds to moving in the counterclockwise direction around $\bdy S'$.
\item For an incoming Type I point $z^-\in\ptsi$, the same procedure is applied.
The difference in the resulting bundle is that forward movement along $\sett{z^-}\times[0,1]$ corresponds to clockwise movement around $\bdy S'$.
(Note that this is automatically achieved by the way we identified $\sett{z^-}\times[0,1]$ with a portion of the boundary of $S'$.)
\item For an outgoing Type II marked point of type $(p,q)\in\brjmps$ implant a counterclockwise path moving in the positive definite direction from $\imm_*T_p\lag$ to $\imm_*T_q\lag$.
\item For an incoming Type II marked point of type $(p,q)\in\brjmps$ implant a counterclockwise path moving in the negative definite direction from $\imm_*T_q\lag$ to $\imm_*T_p\lag$ (equivalently, a clockwise path moving in the positive definite direction from $\imm_*T_p\lag$ to $\imm_*T_q\lag$).
\end{itemize}

\begin{definition}
  \label{dfn:9}
  The Maslov index $\mslv(\markedstrip{u})$ of $\markedstrip{u}$ is defined to the be the Maslov index of the bundle pair defined above.
\end{definition}

\begin{lemma}
  \label{lemma:5}
  Let $\markedstrip{u}=\markedstripstd\in\strips{}{\strip,\ptsii,\brjmptypes}$ be a marked strip connecting $\gamma_\pm\in\chords{\h}$.
  Then the Maslov index $\mslv(\markedstrip{u})$ of $\markedstrip{u}$ is
  \begin{displaymath}
    \mslv(\markedstrip{u})=\ind \gamma_--\ind \gamma_+ - \sum_{1\leq i \leq |\ptsii|}\ind \brjmptypes(i).
  \end{displaymath}
\end{lemma}

\begin{lemma}
  \label{lemma:6}
  Let $\markeddisc{u}=\markeddiscstd\in\discs{}{\disc,\ptsii,\brjmptypes}$ be a marked disc with $\ptsii=(\zii_1,\ldots,\zii_k)$.
  Then the Maslov index $\mslv(\markeddisc{u})$ of $\markeddisc{u}$ is
  \begin{displaymath}
    \mslv(\markeddisc{u})=\sum_{\zii_i\in\ptsii^-}\ind \brjmptypes(i)-\sum_{\zii_i\in\ptsii^+}\ind \brjmptypes(i).
  \end{displaymath}
\end{lemma}

It is instructive to consider the indices in the following situation.
Let $H_t=H$ be a $C^2$-small time independent admissible Hamiltonian.
Furthermore, assume that $H_\lag:=H\circ\imm$ is a Morse function on $\lag$, and the critical points of $H_\lag$ are disjoint from $\imm^{-1}(\imbrjmps)$.
Then there exists $\epsilon_0>0$ such that for every $0<\epsilon<\epsilon_0$ there exists a canonical bijection $\chords{\epsilon H}\cong \crit(H_L)\amalg \brjmps$.
Note that $\hamdiff{\epsilon H}{1}=\hamdiff{H}{\epsilon}$.
To see the bijection, first note that the set of chords can be broken into two disjoint sets, $\chords{\epsilon H}=\Gamma_L\amalg \Gamma_R$,  based on the distance between $\imm^{-1}(\chord(0))$ and $\imm^{-1}(\chord(1))$ in $\lag$.
If the distance is small, $\chord$ is in $\Gamma_L$, and if the distance is large, $\chord$ is in $\Gamma_R$.
The chords in $\Gamma_R$ jump branches of $\imlag$ and correspond to points of $\brjmps$.
If $x\in\imbrjmps$, and $\imm^{-1}(x)=\sett{p,q}$, then there will be a chord that jumps from the $p$ to the $q$ branch and a chord that jumps from the $q$ to the $p$ branch.
The chord that jumps from the $p$ to the $q$ branch is characterized by the property that $\imm^{-1}(\chord(0))$ is near $p$ and $\imm^{-1}(\chord(1))$ is near $q$; a similar statement holds for the other chord with $p$ and $q$ reversed.
Thus these two chords correspond to $(p,q)$ and $(q,p)$ in $\brjmps$, respectively.
The chords in $\Gamma_L$ correspond to the critical points of $H_L$.
This can be deduced from the fact that $\hamvf{H}(p)$ is tangent to $\imlag$ if and only if $p\in\crit H_L$.

Now we examine the indices of these chords.
For $p\in\crit H_L$, let $\chord_p\in\chord_L$ denote the corresponding chord, and for $(p,q)\in\brjmps$ let $\chord_{(p,q)}$ denote the corresponding chord.
\begin{lemma}
  \label{lemma:11}
  $\ind\chord_p=\dim\lag-\ind_pH_L=\ind_p(-H_L)$, where $\ind_pH_L$ denotes the Morse index (number of negative eigenvalues of the Hessian of $H_L$ at $p$), and $\ind\chord_{(p,q)}=\ind (p,q)$, where the index on the right-hand side is from Definition \ref{dfn:8}.
\end{lemma}
\begin{proof}
  First consider the case of $\chord_p$.
  In this case there is no essential difference between an embedded or immersed Lagrangian, so we start by examining the simple example of $\lag=\rr\subset\cc$ and $H(x,y)=\pm x^2/2$, with $\hvf=dz$ and $\sympl=dx\wedge dy$.
  Equip $\lag$ with any constant grading $\laggrad\in\zz$.
  Then $\hamdiff{H}{t}(x,y)=(x,y\mp tx)$, so $\hamdiff{\epsilon H}{1}(\lag)$ is a line of slope $\mp\epsilon$, and this Lagrangian inherits a grading by continuing $\laggrad$ along the Hamiltonian isotopy.
  The only chord is the constant chord $\chord(t)=(0,0)$, and this corresponds to the critical point $(0,0)$ of $H_L$.
  By the discussion after Definition \ref{dfn:7}, $\ind\chord$ is equal to the index of the pair of graded planes $(T_{(0,0)}\hamdiff{\epsilon H}{1}(\lag),T_{(0,0)}\lag)$.
  The result is $\ind \chord = 0$ if $H=-x^2$ and $\ind \chord = 1 $ if $H=+x^2$.
  Thus the lemma holds in the simple case $\rr\subset\cc$, $H=\pm x^2$.
  For more general $H$ (depending on $x$ and $y$) the calculation is essentially the same as long as $H_\lag$ has a single critical point at $x=0$ and $\epsilon$ is chosen small enough ($\chord$ will no longer just be a constant chord though).
  The general result can then be deduced from this simple example by using the Weinstein neighborhood theorem and the Morse lemma to reduce the problem to a direct sum of one-dimensional problems.

  Now consider the case of $\chord_{(p,q)}$ with $(p,q)\in\brjmps$.
  In this case, $\ind\chord_{(p,q)}$ is calculated by moving the graded plane $T_{\chord_{(p,q)}(0)}\imlag$ along the linearized flow $D\hamdiff{\epsilon H}{t}$ until a graded plane $\Lambda_1$ contained in $T_{\chord_{(p,q)}(1)}\mnfld$ is reached, and then the index of the chord is the index of the pair of graded planes $(\Lambda_1,T_{\chord_{(p,q)}(1)}\imlag)$.
  If $\epsilon>0$ is small enough, this pair of graded planes can be homotoped to the pair of graded planes $(\imm_* T_p\lag,\imm_* T_q\lag)$ in such a way that the planes stay transverse throughout the homotopy.
  Thus $\ind\chord_{(p,q)}=\ind (p,q)$.
\end{proof}

\section{Analytic setup, moduli spaces and Fredholm theory}
\label{sec:moduli-spac-analyt}
In this section we define Banach manifolds of marked strips with $W^{1,p}$-regularity.
We also define moduli spaces of (inhomogeneous) holomorphic strips, and also holomorphic discs, and describe the Fredholm theory of the strips.
The analytic setup and Fredholm theory of the discs is not needed.
The constructions and propositions are essentially standard and can be seen as a combination of \cite{seidel-fcpclt} and \cite{MR2785840}.
One place where we give a few details is Proposition \ref{prop:3}.
This proposition deals with the regularity of holomorphic strips.
It is an important point to address because Type II points are not present in the embedded Lagrangian case of \cite{seidel-fcpclt}, and also our setup has a Hamiltonian term and a domain dependent complex structure, which is different than \cite{MR2785840}.

Fix a metric $\fixedmetric$ on $\mnfld$ that agrees with $\sympl(\cdot,\jstd\cdot)$ outside of a compact subset and is such that $\imlag$ is totally geodesic.
This can be done by the following lemma.
\begin{lemma}
  \label{lemma:14}
  Let $\imlag$ be an immersed Lagrangian such that all singular points are transverse double points.
  Then there exists a metric $g$ on $\mnfld$ such that $\imlag$ is totally geodesic.
\end{lemma}
\begin{proof}
  First, suppose $L$ is a smooth submanifold of $M$.
  Let $E\subset TM|L$ be a vector bundle over $L$ such that $E\oplus TL = TM|L$.
  Let $g$ be any metric on $M$ such that $E$ is the orthogonal complement of $TL$ in $TM|L$.
  $L$ is totally geodesic if and only if the second fundamental form vanishes, and this is equivalent to $g(Z,\nabla_X X)=0$ for all $Z\in E_x$, $X\in T_xL$, $x\in L$.
  Let $\tilde X,\tilde Z$ be vector fields in $TM$ defined on a neighborhood of $x\in L$ such that $[\tilde Z,\tilde X](x)=0$ and $\tilde Z(x)=Z$, $\tilde X(x)=X$.
  Then $g(Z,\nabla_X X)= -\frac{1}{2}(\tilde Z (g (\tilde X,\tilde X)))(x)$.
  Thus we have the following characterization of totally geodesic with given normal bundle: Given $L$ and $E$, and a metric $g$ on $M$ such that $TL$ and $E$ are orthogonal, $L$ is totally geodesic if and only if $\tilde Z (g(\tilde X,\tilde X))(x)=0$ for all $x\in L$, $X\in T_xL$, and $Z\in E_x$.

  Now suppose embedded $L$ and $E$ are given, and $g$ is a metric so that $TL$ and $E$ are orthogonal and $L$ is totally geodesic.
  Let $\phi$ be a smooth non-negative function on $M$ so that $(Z\phi)|L=0$ for any vector $Z$ in $E$.
  Then for any $Z\in E_x, X\in T_xL$  we have $\tilde Z ((\phi g)(\tilde X,\tilde X))(x)=(Z \phi)(x)\cdot g(X,X)(x)+\phi(x)\cdot (\tilde Z (g (\tilde X,\tilde X)))(x)=0$ on $L$.
  Thus, at points where $\phi g$ is a metric, $L$ is totally geodesic with respect to $\phi g$ and still has normal bundle $E$.
  Similarly, if $g_1$ and $g_2$ are metrics such that $E$ is orthogonal to $TL$ for both and $L$ is totally geodesic with respect to both, then the same is true of $g_1+g_2$.

  Using these facts we can construct metrics on $M$ such that embedded $L$ is totally geodesic in the following way.
  First, fix any $E$ so that $E\oplus TL=TM|L$.
  Given a small open set $U$ in $M$, it is clear that there exists some metric on $U$ such that $E|L\cap U$ is orthogonal to $TL|L\cap U$, and $L\cap U$ is totally geodesic in $U$.
  Then use a partition of unity $\sett{\phi_i}$ on $\mnfld$ which satsifies $(Z\phi_i)(x)=0$ for $Z\in E_x,x\in L$ to piece together these local metrics to get a global metric on $M$.
  $L$ will be totally geodesic with respect to this global metric, and also $E$ and $TL$ will be orthogonal.

  Now suppose $\imm:L\to \imlag\subset M$ is a Lagrangian immersion with transverse double points.
  In a small neighborhood $U$ of a double point, we can find a diffeomorphism from $U$ to a small open ball in $\rr^{2n}$ such that $\imlag$ maps to the union of two $n$-planes.
  Take the standard metric on $\rr^{2n}$, and pull it back to $U$ in $M$.
  We have thus constructed a metric on $M$ near the double points of $\imlag$ such that $\imlag$ is totally geodesic.
  Choose subbundle $E\subset \imm^*TM$ so that $E$ is the orthogonal complement of $TL$ near the double points of $\imlag$ with respect to the constructed metric.
  Since $\imlag$ is embedded away from the double points, the previous arguments can be used to extend the constructed metric to all of $M$ in such a way that $\imlag$ is totally geodesic and $E$ is the normal bundle.
\end{proof}

Let $\strip_+=[0,\infty)\times[0,1]$ and $\strip_-=(-\infty,0]\times[0,1]$.
Suppose given $\rs=\clrs\setminus\ptsi$ and Type II marked points $\ptsii$.
A choice of a \textit{strip-like end} for a marked point $z\in\ptsi^\pm\amalg\ptsii^\pm$ of Type I or Type II is a proper holomorphic embedding
\begin{equation}
  \label{eq:3}
  \stripend:\strip_\pm\to\clrs\setminus\sett{z}
\end{equation}
that satisfies $\lim_{s\to\pm\infty}\stripend(s,t)=z$ and is such that the image of $\stripend$ is disjoint from $\ptsi\amalg\ptsii$.
\subsection{Analytic setup of strips}
\label{sec:analyt-setup-mark}
Fix an admissible Hamiltonian $\h$.
First consider the case of fixed Type II marked points $\ptsii=\ptsii^+=(\zii_1,\ldots,\zii_m)\subset\bdy\strip$ ($\ptsii=\emptyset$ is allowed).
Fix a choice of strip like ends, $\stripend_i$ for each Type II marked point $\zii_i\in\ptsii$, such that the images of the strip-like ends do not overlap.

The strip-like ends determine a metric $g$ on $\strip$, up to equivalence, as follows.
First, on $\stripend_i([1,\infty)\times[0,1])$, with $s+it$ the standard coordinates on $[1,\infty)\times[0,1]\subset\stripplus$,  $g=ds^2+dt^2$.
Second, on the complement of the images of the strip-like ends, with $s+it$ the standard coordinates on $\strip$, $g=ds^2+dt^2$.
Finally, on the remaining part, the metric can be chosen in any way that smoothly matches up with the previous choices.
Different choices of strip-like ends will result in equivalent metrics.
Call such a metric a compatible metric.
The metric defines a measure (volume form) on $\strip\setminus\ptsii$.

Let $\gamma_\pm\in\chords{\h}$ and suppose $\markedstrip{u}=\markedstripstd$ is a marked strip satisfying
\begin{align}
  \label{eq:4}
  &\text{$u$ is constant in a neighborhood of each Type II marked point,}\nonumber\\
  &\text{$u(s,t)=\gamma_{\pm}(t)$ for $s$ near $\pm\infty$, and}\\
  &\text{$u$ is smooth.}\nonumber
\end{align}

\begin{definition}
  \label{dfn:3}
  With $\markedstrip{u}$ as above, let $\mapmodel{\markedstrip{u}^*T\mnfld}$  denote the set of all sections $\xi$ of $u^*T\mnfld$ over $\strip\setminus\ptsii$ such that $\xi$ is in $W^{1,p}$ and has boundary values in $\iota_*\lift^*T\lag$.
  The norm is defined by using the metric on $\strip\setminus\ptsii$ determined by the strip like ends and the fixed metric $\fixedmetric$ on $\mnfld$.
  We assume that $p>2$.
\end{definition}

Let $\mapchart{\markedstrip{u},\rho}$ be the subset of $\mapmodel{\markedstrip{u}^*T\mnfld}$ consisting of all sections $\xi$ such that $\|\xi\|_{W^{1,p}}<\rho$ for all $z\in\strip\setminus\ptsii$.
Define
\begin{equation}
  \label{eq:2}
  \Phi_{\markedstrip{u},\rho}:\mapchart{\markedstrip{u},\rho}\to\strips{}{\strip,\ptsii,\brjmptypes},\quad \xi\mapsto\markedstrip{u_\xi}=(u_\xi,\ptsii,\brjmptypes,\lift_\xi),
\end{equation}
where $u_\xi=\exp_u\xi$ and $\lift_\xi=\exp_\lift(\xi|\bdy\strip\setminus\ptsii)$.
For the last equation we use the fact that $\imlag$ is totally geodesic for $\fixedmetric$ and $\xi|\bdy\strip\setminus\ptsii$ can be lifted to a vector field on $T\lag$ using $\lift$ and $\iota$.
For small enough $\rho$, $\Phi_{\markedstrip{u},\rho}$ is injective.

\begin{definition}
  \label{dfn:4}
  Let 
  \begin{displaymath}
    \stripsfixedIIstd = \bigcup_{\markedstrip{u},\rho>0} \Phi_{\markedstrip{u},\rho}(\mapchart{\markedstrip{u},\rho}),
  \end{displaymath}
  where the union is over all $\markedstrip{u}$ that satisfy \eqref{eq:4}.
  The maps $\Phi_{\markedstrip{u},\rho}$ with $\rho>0$ small give $\stripsfixedIIstd$ the structure of a smooth Banach manifold.
\end{definition}

Now consider the case where the Type II marked points $\ptsii$ are allowed to vary.
Let
\begin{eqnarray}
  \label{eq:42}
  \config{m}{\strip}& = &\set{\ptsii}{|\ptsii|=m,\text{ $\ptsii$ is ordered as in Definition \ref{dfn:2}}}\\ 
  &\subset&(\bdy\strip)^m\setminus \text{(fat diagonal)}\nonumber
\end{eqnarray}
denote the configuration space of Type II marked points.

\begin{definition}
  \label{dfn:6}
  For $\brjmptypes:\sett{1,\ldots,m}\to\brjmps$, let
  \begin{displaymath}
    \stripsstd=\coprod_{\ptsii\in\config{m}{\strip}}\stripsfixedIIstd\times\sett{\ptsii}.
  \end{displaymath}
$\stripsstd$ is given the structure of a $C^0$-Banach manifold with coordinate charts
\begin{equation}
  \label{eq:6}
  \Psi_{\markedstrip{u},\rho}:\mapchart{\markedstrip{u},\rho}\times\mathcal C\to\stripsstd,\quad (\xi,\ptsii)\mapsto (\Phi_{\markedstrip{u},\rho}(\xi)\circ\upsilon_{\ptsii},\ptsii),
\end{equation}
where $\mathcal C$ is an open neighborhood of some $\Delta_0\in\config{m}{\strip}$, $\Phi_{\markedstrip{u},\rho}$ is the map \eqref{eq:2} for $\Delta_0$, and $\upsilon_{\ptsii}$ is a diffeomorphism from $\strip\setminus \ptsii\to\strip\setminus\ptsii_0$ which is a biholomorphism in a neighborhood of $\Delta$.
The manifold structure is only $C^0$ because of the appearance of $\upsilon_{\ptsii}$ (which appears because the underlying domains of the maps are not all the same).
\end{definition}

\subsection{Moduli spaces of strips and discs}

We now define moduli spaces of (inhomogeneous) holomorphic strips and holomorphic discs.
\begin{definition}
  \label{dfn:10}
  Fix admissible Floer data $\h,\J$ as in Definition \ref{dfn:1}.
  Given $\chord_\pm\in\chords{\h}$ and $\brjmptypes:\sett{1,\ldots,m}\to\brjmps$, where if $m=0$ we require $\chord_-\neq\chord_+$ (see Remark \ref{rmk:7}), let
  \begin{displaymath}
    \conntrajpar{}{\chord_-,\chord_+;\brjmptypes}=\conntrajpar{}{\chord_-,\chord_+;\brjmptypes;\h,\J}
  \end{displaymath}
  denote the set of all $\markedstrip{u}=\markedstripstd\in\strips{}{\strip,\brjmptypes}$ such that 
  \begin{itemize}
  \item $u$ restricted to $\strip\setminus\ptsii$ is smooth,
  \item $\lim_{s\to\pm\infty}u(s,t)=\chord_\pm(t)$ uniformly in $t$,
  \item the energy of $u$ is finite:
    \begin{displaymath}
      E(u)=\int_{\strip\setminus\ptsii}\biggl|\pd{u}{s}\biggr|^2dsdt<\infty,
    \end{displaymath}
  \item and $u$ restricted to $\strip\setminus\ptsii$ satisfies Floer's equation
    \begin{equation}
      \label{eq:8}
      \pd{u}{s}+J_t(u)\left(\pd{u}{t}-\hamvf{H_t}(u)\right)=0.
    \end{equation}
  \end{itemize}
  Let
  \begin{displaymath}
    \conntraj{}{\chord_-,\chord_+;\brjmptypes}=\conntraj{}{\chord_-,\chord_+;\brjmptypes;\h,\J}=\conntrajpar{}{\chord_-,\chord_+;\brjmptypes}/\rr,
  \end{displaymath} where $\rr$ acts on $\conntrajpar{}{\chord_-,\chord_+;\brjmptypes}$ by translation. 
  Note that besides translating $u$ and $\ell$, $\mathbb R$ also translates $\Delta$ in the obvious way.
  If $\ptsii=\emptyset$ (so $\brjmptypes=\emptyset:\emptyset\to\brjmps$) define
  \begin{gather*}
    \conntrajpar{}{\chord_-,\chord_+}=\conntrajpar{}{\chord_-,\chord_+;\emptyset},\\
    \conntraj{}{\chord_-,\chord_+}=\conntraj{}{\chord_-,\chord_+;\emptyset}.
  \end{gather*}
  $\conntraj{}{\chord_-,\chord_+;\brjmptypes}$ is called a \textit{moduli space of strips}.
\end{definition}

\begin{remark}
  \label{rmk:7}
  If $\chord_-=\chord_+=:\chord$ then the \textit{constant strip}\footnote{Calling this solution constant makes more sense if $u$ is viewed as a map from $\rr$ into the space of paths in $\mnfld$ with boundary on $\imlag$.} $u(s,t)=\chord(t)$ is a solution of Floer's equation.
  If $\ptsii=\emptyset$ then this curve is the only solution, as can be seen by equations \eqref{eq:9} and \eqref{eq:5}.
  In this case it is unstable and should not be considered, so we define $\conntraj{}{\chord,\chord}=\emptyset$.
  If $\ptsii\neq \emptyset$ then in principle constant strips are allowed; however, under the assumption that $\h$ is admissible, $\conntraj{}{\chord,\chord;\brjmptypes}$ necessarily cannot contain constant strips because $\imbrjmps$ is disjoint from the endpoints of all Hamiltonian chords.
  This will be important when we study transversality in Section \ref{sec:regularity}.
\end{remark}

\begin{proposition}
  \label{prop:3}
  If the Floer data $\h,\J$ is admissible, then
  \begin{displaymath}
    \conntrajpar{}{\chord_-,\chord_+;\brjmptypes}\subset\strips{1,p}{\strip,\brjmptypes}.
  \end{displaymath}
\end{proposition}
\begin{proof}
  This proposition is mostly a standard fact.
  Inhomogeneous holomorphic curves are smooth away from the marked points and hence have regularity $W^{1,p}_{loc}$ away from the marked points.
  It is well-known that the strips actually have exponential convergence near $\pm\infty$.
  The strips also have exponential convergence near the Type II marked points.
  This fact is slightly less standard because the complex structure and Hamiltonian are not translation invariant on the strip-like ends for the Type II points.
  One could pass to the graph construction (as in Chapter 8 of \cite{MR2954391}) to make the domain dependence of $\J$ and the Hamiltonian term disappear and then try to apply \cite{MR1849689}; however, one would then have to deal with non-transverse Lagrangian boundary conditions (because the boundary condition becomes $\rr\times\lag$ in the graph).
  Instead, we give an alternative proof of exponential convergence in Theorem \ref{thm: asymptotic behavior}, which in turn is based on $W^{1,p}$ estimates from \cite{MR1890078}.
  The proposition then follows because exponential convergence near the marked points is enough to put the strips into $\strips{1,p}{\strip,\brjmptypes}$.
\end{proof}

We also need to consider moduli spaces of discs that are holomorphic with respect to a fixed (that is, domain independent) almost complex structure $J$ on $\mnfld$.
Let $\ptsii$ be a list of Type II marked points, and $\brjmptypes:\sett{1,\ldots,|\ptsii|}\to\brjmps$ a choice of branch jumps.
Since $\imm:\lag\to\mnfld$ is exact, a non-constant holomorphic disc must necessarily have at least one branch point.
Therefore we assume $|\ptsii|\geq 1$.

\begin{definition}
  \label{dfn:13}
  Given $J$, $\ptsii$, and $\brjmptypes$ as above let $\holdiscspar{}{\ptsii,\brjmptypes;J}$
  denote the set of all $\markeddisc{u}=\markeddiscstd\in \discs{}{\ptsii,\brjmptypes}$ such that
  \begin{itemize}
  \item $u$ restricted to $\disc\setminus\ptsii$ is smooth,
  \item $\int_{\disc \setminus \ptsii} u^*\sympl<\infty$, and
  \item $u$ satisfies the homogeneous Cauchy-Riemann equation
    \begin{displaymath}
      \pd{u}{s}+J(u)\pd{u}{t}=0,
    \end{displaymath}
    where $s+it$ are holomorphic coordinates on $\disc\setminus\ptsii$.
  \end{itemize}
  Similarly, given only $J$ and $\brjmptypes$, let
  \begin{displaymath}
    \holdiscspar{}{\brjmptypes;J} = \coprod_{\ptsii}\holdiscspar{}{\ptsii,\brjmptypes;J}\times\sett{\ptsii}\subset \discs{}{\brjmptypes}.
  \end{displaymath}
  Here the union is over all $\ptsii$ such that $|\ptsii|$ is constant (the constant is determined by the domain of $\brjmptypes$).
  Let
  \begin{displaymath}
    \holdiscs{}{\brjmptypes;J}= \holdiscspar{}{\brjmptypes;J}/\Aut(\disc),
  \end{displaymath}
  where $\Aut(\disc)=\PSL(2,\rr)$ acts in the obvious way.
  $\holdiscs{}{\brjmptypes;J}$ is called a \textit{moduli space of discs}.
\end{definition}

\begin{remark}
  \label{rmk:5}
  For certain $\brjmptypes$, $\holdiscs{}{\brjmptypes;J}$ consists of (equivalence classes of) marked discs $\markeddisc{u}=\markeddiscstd$ such that $u$ is constant.
  However, $\ell$ can have different values on different components of $\bdy \disc\setminus\ptsii$.
  The map $u$ is constant if and only if $\int u^*\sympl=0$, which by exactness is actually a condition on $\brjmptypes$.
  We only consider these moduli spaces if $|\ptsii|\geq 3$ (because otherwise the curves are not stable).
\end{remark}

\subsection{Fredholm and index theory of holomorphic strips}
\label{sec:fredh-index-theory}
This section explains how to exhibit $\conntrajpar{}{\chord_-,\chord_+;\brjmptypes}$ as the zero set of a section of a Banach bundle over $\strips{1,p}{\strip,\brjmptypes}$, and the resulting Fredholm theory.
Fix Floer data $\h,\J$, chords $\chord_\pm\in\chords{\h}$, and $\brjmptypes$.
This determines $\strips{1,p}{\strip,\brjmptypes}$ as in Section \ref{sec:analyt-setup-mark}.

Let $\banbundle{0,p}{\strip,\brjmptypes}\to\strips{1,p}{\strip,\brjmptypes}$ be the Banach bundle with fiber
\begin{displaymath}
  \banbundle{0,p}{\strip,\brjmptypes}_{(\markedstrip{u},\ptsii)}=\bunmodel{\zeroone{}\otimes_\cc\markedstrip{u}^*T\mnfld}.
\end{displaymath}
Here the right hand side consists of $L^p$-sections over $\strip\setminus\ptsii$, and the $L^p$-norm is defined in a way similar to Definition \ref{dfn:3}.
The complex structure of $u^*TM$ over $(s,t)\in \strip$ is $J_t$.
The bundle comes equipped with a $C^0$-section
\begin{equation}
  \label{eq:12}
  \dbarstrip:\strips{1,p}{\strip,\brjmptypes}\to\banbundle{0,p}{\strip,\brjmptypes}
\end{equation}
defined by
\begin{displaymath}
  \dbarstrip(\markedstrip{u})=\frac{1}{2}\left[\pd{u}{s}+J_t(u)\left(\pd{u}{t}-\hamvf{H_t}\right)\right]\otimes (ds-idt).
\end{displaymath}
By Proposition \ref{prop:3}, $\conntrajpar{}{\chord_-,\chord_+;\brjmptypes}=\dbarstrip^{-1}(0_{section})$.
Restricting to the charts $\mapchart{\markedstrip{u_0},\rho}\times\mathcal C$ given by \eqref{eq:6}, $\dbarstrip$ becomes a smooth section and the linearization at a holomorphic strip is a Fredholm operator.
We summarize these important aspects with the following proposition, which is standard.
\begin{proposition}
  \label{prop:5}
  If $(\markedstrip{u},\ptsii)\in \mapchart{\markedstrip{u_0},\rho}\times\mathcal C$ and $\dbarstrip(\markedstrip{u})=0$, then the linearization of $\dbarstrip$ at $(\markedstrip{u},\ptsii)$ can naturally be identified with a Fredholm operator
  \begin{displaymath}
    D\dbarstrip=D_{(\markeddisc{u},\ptsii)}\dbarstrip: \mapmodel{\markedstrip{u}^*T\mnfld}\times T_{\ptsii}\mathcal C \to \bunmodel{\zeroone{}\otimes_\cc\markedstrip{u}^*T\mnfld}.
  \end{displaymath}
  The formula for the first component is
  \begin{displaymath}
    (\xi,0)\mapsto\frac{1}{2}\left\{\left[\nabla_s\xi+J_t(\nabla_t\xi-\nabla_\xi\hamvf{H_t})\right]-\frac{1}{2}J_t\nabla_\xi J_t\left[\pd{u}{s}-J_t\left(\pd{u}{t}-\hamvf{H_t}\right)\right]\right\}\otimes_\cc (ds-idt).
  \end{displaymath}
  The index is
  \begin{equation}
    \label{eq:13}
    \ind D\dbarstrip= \ind \chord_--\ind\chord_+-\sum \ind \brjmptypes(i)+|\ptsii|.
  \end{equation}
\end{proposition}

\section{Gromov compactification of moduli space of discs and strips}
\label{sec:grom-comp-moduli}
Fix admissible Floer data $\h,\J$ and $\chord_\pm\in\chords{\h}$ with $\chord_-\neq\chord_+$.
The ultimate goal of this section is to describe the Gromov compactification of the moduli space $\conntraj{}{\chord_-,\chord_+}$ of strips connecting $\chord_-$ to $\chord_+$.

First we need to describe the Gromov compactification of $J$-holomorphic discs with one incoming Type II marked point.
To this end, let $J$ be any (domain independent) almost complex structure (later $J$ will be $J_0$ or $J_1$ where $\J=\Jdep{t}{t}$).
Let $\brjmptypes:\sett{1}\to\brjmps$ be given.
We describe the compactification of $\holdiscs{}{\brjmptypes;J}$.
\begin{definition}
  \label{dfn:15}
  With $J$ and $\brjmptypes$ as above, the \textit{compactified moduli space of discs} $\comholdiscs{}{\brjmptypes;J}$
  consists of equivalence classes of pairs $(T,\dep{[\markeddisc{u_v}]}{}{v\in Ver(T)})$ where:
  \begin{itemize}
  \item $T$ is a planar tree with a distinguished root vertex $v_0\in Ver(T)$.
  \item For each vertex $v\in Ver(T)$, $[\markeddisc{u_v}]=[(u_v,\ptsii_v,\brjmptypes_v,\ell_v)]\in\holdiscs{}{\brjmptypes_v;J}$ is a stable marked disc\footnote{This means that if $u_v$ is constant then $|\ptsii_v|\geq3$. See also Remark \ref{rmk:5}.} with 
    \begin{displaymath}
      |\ptsii_v|=\left\{
        \begin{array}{ll}
          \mathrm{valency}(v) & v\neq v_0\\
          \mathrm{valency}(v)+1 & v=v_0
        \end{array}
      \right.
    \end{displaymath}
    and $\ptsii^-_v=\sett{\zii_1}$ (i.e., on each component the first marked point is incoming and the rest are outgoing).
  \item  The branch jump types among the discs are compatible in the sense that if two vertices $v_1$ and $v_2$ are connected by an edge, then the corresponding branch jump types agree.
    (More precisely: The fact that $T$ has a distinguished root vertex implies that the edges can be canonically oriented in the outward direction from the root. If we label the incoming edge at a vertex as the first edge, then because $T$ is planar each other edge gets a unique number by proceeding in counterclockwise order. If an outgoing edge at $v_1$ is the $i^{th}$ edge and it connects to $v_2$, we require that $\brjmptypes_{v_1}(i)=\brjmptypes_{v_2}(1)$.)
    In particular, this means that the domain discs can be glued together along the marked points and the resulting map is continuous on the glued domain.
  \end{itemize}
  Two pairs $(T,\dep{[\markeddisc{u_v}]}{}{v\in Ver(T)})$ and $(T',\dep{[\markeddisc{u'_v}]}{}{v\in Ver(T')})$ are equivalent if there exists an isomorphism $\phi:T\to T'$ of rooted planar trees such that $[\markeddisc{u_v}]=[\markeddisc{u'_{\phi(v)}}]$ for all vertices $v$ of $T$.
  We denote elements of $\comholdiscs{}{\brjmptypes;J}$ using bold letters, for example $\commarkeddisc{u}=[(T,\dep{[\markeddisc{u_v}]}{}{v\in Ver(T)})]$.
\end{definition}

\begin{proposition}
  \label{prop:7}
  $\comholdiscs{}{\brjmptypes;J}$ as described above is the Gromov compactification of $\holdiscs{}{\brjmptypes;J}$.
\end{proposition}
\begin{proof}
  This is Gromov's compactness theorem for discs $u$ on immersed Lagrangians with  boundary lifts $\ell$; see for example \cite{MR2785840}, \cite{MR1890078}.  Note that due to exactness of $\lag$, an a priori bound for the symplectic area is given by
  \begin{displaymath}
    \int u^*\omega = -\lagprim(p)+\lagprim(q),
  \end{displaymath}
  where $(p,q)=\brjmptypes(1)$; hence the space $\comholdiscs{}{\brjmptypes;J}$ is compact with respect to the usual Gromov topology.
  See also the proof of Proposition \ref{prop:8} for some remarks on the bubbling off analysis.
\end{proof}

Now we move on to strips.
Briefly, the compactification of the moduli space of strips consists of broken strips with trees of disc bubbles (in the sense of Definition \ref{dfn:15}) attached to Type II marked boundary points.
Here are the details.
\begin{definition}
  \label{dfn:16}
  Fix admissible Floer data $\h,\J$, and $\chord_\pm\in\chords{\h}$ with $\chord_-\neq\chord_+$.
  The \textit{compactified moduli space of strips} $\comconntraj{}{\chord_-,\chord_+}=\comconntraj{}{\chord_-,\chord_+;\h,\J}$ 
  consists of all tuples $$([\markedstrip{u_1}],\ldots,[\markedstrip{u_N}],\commarkeddisc{v^0_1},\ldots,\commarkeddisc{v^0_{k_0}},\commarkeddisc{v^1_1},\ldots,\commarkeddisc{v^1_{k_1}})$$ such that
  \begin{itemize}
  \item each $\markedstrip{u_i}=(u_i,\ptsii_i,\brjmptypes_i,\ell_i)$ is a marked strip connecting $\chord_{i,-}$ to $\chord_{i,+}$,
  \item each strip is stable in the sense that if $\chord_{i,-}=\chord_{i,+}$ then $\ptsii_i\neq \emptyset$,
  \item the total number of Type II marked points on the bottom of the strips is $k_0$, and the total number on the top is $k_1$,
  \item $\chord_-=\chord_{0,-}$, $\chord_{i,+}=\chord_{i+1,-}$, and $\chord_{N,+}=\chord_+$,
  \item $\commarkeddisc{v^0_i}$ is an element of $\comholdiscs{}{\brjmptypes^0_i;J_0}$ where $\brjmptypes^0_i:\sett{1}\to\brjmps$ is the branch jump type of the $i^{th}$ Type II marked point occurring along the bottom of the broken strips (starting from the left side and moving right); the marked point is considered incoming for $\commarkeddisc{v^0_i}$ (and outgoing for the strip, the element of $\brjmps$ is the same for both),
  \item $\commarkeddisc{v^1_i}$ is an element of $\comholdiscs{}{\brjmptypes^1_i;J_1}$ where $\brjmptypes^1_i:\sett{1}\to\brjmps$ is the branch jump type of the $i^{th}$ Type II marked point occurring along the top of the broken strips (starting from the right side and moving left); the marked point is considered incoming for $\commarkeddisc{v^1_i}$.
  \end{itemize}
  We denote elements of $\comconntraj{}{\chord_-,\chord_+}$ using bold letters, for example $$\commarkedstrip{u}=([\markedstrip{u_1}],\ldots,[\markedstrip{u_N}],\commarkeddisc{v^0_1},\ldots,\commarkeddisc{v^0_{k_0}},\commarkeddisc{v^1_1},\ldots,\commarkeddisc{v^1_{k_1}}).$$
\end{definition}

\begin{proposition}
  \label{prop:8}
  $\comconntraj{}{\chord_-,\chord_+}$ as described above is the Gromov compactification of $\conntraj{}{\chord_-,\chord_+}$.
\end{proposition}
\begin{proof}
  By definition (see Section (8g) of \cite{seidel-fcpclt} for details), the energy of a holomorphic strip $\markedstrip{u}$ is 
  \begin{equation}
    \label{eq:9}
    E(\markedstrip{u})=\int \biggl|\pd{u}{s}\biggr|^2ds\wedge dt.
  \end{equation}
  By exactness and Stokes' theorem, the energy of each strip in $\conntraj{}{\chord_-,\chord_+}$ is equal to $\action(\chord_-)-\action(\chord_+)$, where $\action$ is the action of a chord, defined by the formula
  \begin{equation}
    \label{eq:5}
    \action(\gamma)=-\int(\gamma^*\oneform+H_t(\gamma(t))dt)-\lagprim(\gamma(0))+\lagprim(\gamma(1)).
  \end{equation}
  Thus there is an energy bound and Gromov compactness can be applied in the usual way, for example as explained in \cite{MR1223659}, with a few modifications needed for the immersed case.

  The key difference in the immersed case is that when the rescaling procedure is applied to find a disc bubble, the resulting bubble is only smooth on the disc minus a point (equivalently, on the upper half-plane with the missing point being the point at infinity).
  The same consideration applies to the strip component: a priori, after the bubbling off, the remaining strip is only smooth away from the points where the bubbles appeared.
  If the Lagrangian is embedded, the removable singularities theorem implies that the maps extend smoothly over the missing points.
  The removable singularities theorem does not apply in the immersed case; instead, we need to appeal to Appendix \ref{sec:asymptotic-analysis} to conclude that the curves extend continuously over the missing points, and moreover the convergence is exponential in strip-like coordinates.
  See also the discussion in the proof of Proposition \ref{prop:3}.
  Exactness implies that there can be no disc bubbles without branch jumps, and no sphere bubbles.

  Note that having the boundary lifts $\ell$ for the curves is necessary to be able to apply the results of Appendix \ref{sec:asymptotic-analysis}.
  Otherwise we would have to deal with the possibility of curves switching branches arbitrarily many times (and possibly infinitely many times in the limit).
  See \cite{MR2684508} for a discussion of compactness (or lack thereof) in this case.
\end{proof}

\section{Regularity}
\label{sec:regularity}
The goal of this section is to prove that for generic Floer data $\h, \J$, the moduli spaces $\conntrajpar{}{\chord_-,\chord_+;\brjmptypes;\h,\J}$ are regular for all $\brjmptypes$ and all $\chord_\pm$.
Regularity means that the linearized operators
\begin{displaymath}
  D_{(\markedstrip{u},\ptsii)}\dbarstrip:\mapmodel{\markedstrip{u}^*T\mnfld}\times T_\ptsii\mathcal C \to \bunmodel{\Lambda^{0,1}\otimes \markedstrip{u}^*T\mnfld}
\end{displaymath}
 of Proposition \ref{prop:5} are surjective for all holomorphic strips in $(\markedstrip{u},\ptsii)$  in the moduli spaces $\conntrajpar{}{\chord_-,\chord_+;\brjmptypes;\h,\J}$.
Given an admissible $\h$ (that is, $\h$ satisfies the conditions following Definition \ref{dfn:1}), a time-dependent almost complex structure $\J=\Jdep{t}{t}$ is said to be regular if for all branch jump types $\brjmptypes$, and all chords $\chord_\pm\in\chords{\h}$, all the moduli spaces $\conntrajpar{}{\chord_-,\chord_+;\brjmptypes;\h,\J}$ are regular.

The existence of regular almost complex structures follows from the usual methods.
We give an outline here to highlight that the immersed case does not present any difficulties.
Also, we remark that achieving regularity \textit{for all} $\brjmptypes$ is not problematic.
First, the following ``somewhere injectivity-like'' result follows from the proof of Theorem 4.3 in \cite{MR1360618}.
(The theorem in \cite{MR1360618} is stated for curves $u:\rr^2\to \mnfld$ which satisfy Floer's equation and are 1-periodic in a certain sense in the second variable.
However, it is easy to check that the proof holds for any inhomogeneous holomorphic curve $u:\rr\times(0,1)\to\mnfld$ which converges to chords at $\pm\infty$.
Lagrangian boundary conditions are not even needed.)
\begin{proposition}
  \label{prop:9}
  Let $\markedstrip{u}\in \conntrajpar{}{\chord_-,\chord_+;\brjmptypes}$ and assume $u$ is nonconstant; that is $u(s,t)\neq \chord_-(t)$ for some $(s,t)$.
  Let $\regpts(\markedstrip{u})$ be the set of $(s,t)\in \rr\times(0,1)$ such that
  \begin{itemize}
  \item $\pd{u}{s}(s,t)\neq 0$,
  \item $u(s,t)\neq \chord_\pm(t)$, and
  \item $u(s,t)\notin u((\rr\setminus\sett{s})\times\sett{t})$.
  \end{itemize}
  Then $\regpts(\markedstrip{u})$ is an open dense subset of $\rr\times(0,1)$.
\end{proposition}
Note that by Remark \ref{rmk:7} any $\markedstrip{u}\in\conntrajpar{}{\chord_-,\chord_+;\brjmptypes}$ is nonconstant as long as $\h$ is admissible.

The next step is to consider the universal moduli spaces.
Namely, let $\mathcal J$ be some Banach space completion of the set of smooth time-dependent $\omega$-compatible almost complex structures.
Let
\begin{displaymath}
  \conntrajpar{}{\chord_-,\chord_+;\brjmptypes;\h,\mathcal J} = \coprod_{\J\in\mathcal J}\conntrajpar{}{\chord_-,\chord_+;\brjmptypes;\h,\J}\times\sett{\J}.
\end{displaymath}
As in Section \ref{sec:fredh-index-theory}, there exists a functional analytic framework in which this moduli space can be exhibited as the zero set of a $C^0$-section
\begin{displaymath}
  \univdbarstrip:\strips{1,p}{\strip,\brjmptypes}\times\mathcal J \to\banbundle{0,p}{\strip,\brjmptypes;\mathcal J}.
\end{displaymath}
The analogue of Proposition \ref{prop:5} is the following:
$\strips{1,p}{\strip,\brjmptypes}\times\mathcal J$ admits Banach manifold charts of the form $\mapchart{\markedstrip{u},\rho}\times\mathcal C\times\mathcal J$, and there are corresponding trivializations of $\banbundle{0,p}{\strip,\brjmptypes;\mathcal J}$ over these charts, and $\univdbarstrip$ is smooth when restricted to these charts.

Now cover $\strips{1,p}{\strip,\brjmptypes}\times\mathcal J$ by countably many of the charts $\mathcal U:=\mapchart{\markedstrip{u},\rho}\times\mathcal C\times\mathcal J$ from above.
Proposition \ref{prop:9} and standard techniques imply that the linearization of the universal section $\univdbarstrip$ is surjective at every element of $\conntrajpar{}{\chord_-,\chord_+;\brjmptypes;\h,\mathcal J} \cap \mathcal U.$
Now let $\Jreg(\gamma_-,\gamma_+;\brjmptypes)$ be the intersection (over all $\mathcal U$) of the sets of regular values of the projections $\conntrajpar{}{\chord_-,\chord_+;\brjmptypes;\h,\mathcal J} \cap \mathcal U \to\mathcal J.$
By the Sard-Smale theorem and the fact that we only use countably many $\mathcal U$'s, $\Jreg(\gamma_-,\gamma_+;\brjmptypes)$ is a Baire set.
Let $\Jreg_0$ be the intersection of all $\Jreg(\gamma_-,\gamma_+;\brjmptypes)$ over all $\gamma_\pm$ and all $\brjmptypes$.
Using an argument of Taubes, there exists a subset $\Jreg\subset\Jreg_0$ such that $\Jreg$ consists of smooth almost complex structures and is a Baire set in the $C^\infty$ topology (see the proof of Proposition 3.1.5 in \cite{MR2954391}).
Then we have:
\begin{proposition}
  \label{prop:10}
  Let $\h$ be admissible.
  Then $\Jreg$ is a Baire set consisting of smooth time-dependent regular almost complex structures.
  In particular, for $\J\in\Jreg$, each moduli space $\conntraj{}{\chord_-,\chord_+;\brjmptypes;\h,\J}$
  is a smooth manifold of dimension
  \begin{displaymath}
    \ind \chord_--\ind \chord_+-\sum_{i=1}^m\ind\brjmptypes(i)+m-1
  \end{displaymath}
  where $m$ is the number of Type II marked points.
\end{proposition}


\section{Floer cohomology}
\label{sec:floer-cohomology}
Our treatment of Floer cohomology follows the standard lines.
The Floer differential is defined by couting rigid strips which satisfy Floer's equation, just as in the embedded case.
However, because our Lagrangians are only immersed, extra details are needed to show that disc bubbling does not cause problems.
In general, disc bubbles attached to Type II points could appear in the Gromov compactification of the moduli space of strips, see Definition \ref{dfn:16}.
We use the positivity condition to show that generically this does not happen for the strips we are interested in.

Recall that for now we are assuming that $\imm$ has transverse double points.
Let $\h,\J$ be admissible Floer data with $\J\in\Jreg$.
The Floer cochain complex is
\begin{equation}
  \label{eq:16}
  \cf(\imm;\h,\J)=\cf(\imm)=\bigoplus_{k}\cf^k(\imm),\quad \cf^k(\imm)=\bigoplus_{\chord\ :\ \ind \chord=k}\zz_2\cdot\chord.
\end{equation}
The Floer complex is $\zz$-graded by Definition \ref{dfn:7}.
The Floer differential is
\begin{equation}
  \label{eq:17}
  \fd:\cf(\imm)\to\cf(\imm),\quad \chord_+\mapsto\sum_{\st{\chord_-}{\ind \chord_-=\ind\chord_++1}} |\conntraj{}{\chord_-,\chord_+}|\cdot\chord_-.
\end{equation}
The Floer cohomology is
\begin{equation}
  \label{eq:18}
  \hf(\imm)=\mathrm{H}(\cf(\imm),\fd).
\end{equation}

In the next subsection we will show that $\fd$ is well-defined and $\fd^2=0$.
In the subsequent subsection we will show that the Floer cohomology $\hf(\imm)$ is independent of the choice of $(\h,\J)$.

For simplicity, discussion is limited to the case of a single Lagrangian $\imm$.
$\hf(\imm)$ should be interpreted as the self-Floer cohomology of $\imm$ and could also be denoted as $\hf(\imm,\imm)$.
We could also consider two Lagrangian immersions $\imm$ and $\imm'$ and define $\hf(\imm,\imm')$ using the same techniques.

More generally, the immersions can be viewed as objects in the exact Fukaya category; this will be discussed in Section \ref{sec:units-a_infty-categ}.
With this point of view, we will show that $\imm$ and $\imm'$ are quasi-isomorphic objects in the Fukaya category if $\imm'=\Phi\circ\imm$ for some Hamiltonian diffeomorphism $\Phi$.
In fact, we will prove a stronger statement, that if $\imm'$ is an exact deformation of $\imm$ that meets certain requirements, then $\imm$ and $\imm'$ are quasi-isomorphic objects.
For embedded Lagrangians, the notion of exact deformation and Hamiltonian deformation agree; for immersed Lagrangians, Hamiltonian deformation is strictly stronger.

\subsection{CF$(\iota;\textsl{H},\textsl{J})$ is a chain complex}
\label{sec:cfiota-textslh-texts}
The positivity Condition \ref{ass:1}, the compactness result (Proposition \ref{prop:8}) and the regularity result (Proposition \ref{prop:10}) are the main tools that will be used in this section,

First, we show $\fd$ is well-defined.
\begin{lemma}
  \label{lemma:8}
  If $\chord_\pm\in\chords{\h}$ with $\ind\chord_-=\ind\chord_+ +1$ then $\conntraj{}{\chord_-,\chord_+}$ is a compact 0-dimensional manifold.
  Hence $\fd$ is well-defined.
\end{lemma}
\begin{proof}
  Let $[\markedstrip{u_n'}]\in\conntraj{}{\chord_-,\chord_+}$ be a sequence of strips. 
  We need to show that it has a convergent subsequence.
  By Proposition \ref{prop:8}, it has a subsequence converging to a broken strip with disc bubbles, call it
  \begin{displaymath}
    \commarkedstrip{u}=([\markedstrip{u_1}],\ldots,[\markedstrip{u_N}],\commarkeddisc{v^0_1},\ldots,\commarkeddisc{v^0_{k_0}},\commarkeddisc{v^1_1},\ldots,\commarkeddisc{v^1_{k_1}})\in\comconntraj{}{\chord_-,\chord_+}.
  \end{displaymath}
  Here, 
  \begin{displaymath}
    \markedstrip{u_i}=(u_i,\ptsii_i,\brjmptypes_i,\ell_i)\in\conntrajpar{}{\chord_{i,-},\chord_{i,+};\brjmptypes_i}
  \end{displaymath}
  are strips with $\chord_-=\chord_{1,-}$, $\chord_{i,+}=\chord_{i+1,-}$, and $\chord_{N,+}=\chord_+$; and $\commarkeddisc{v^0_j}$ are trees of $J_0$-holomorphic discs attached to the Type II marked points of the strips on the bottom boundary; and $\commarkeddisc{v^1_j}$ are trees of $J_1$-holomorphic discs attached to the Type II marked points of the strips on the the top boundary.

  The total symplectic area of a holomorphic disc $\commarkeddisc{v^i_j}$ is positive because it has one non-nodal Type II marked point and hence must be nonconstant.
  Thus Condition \ref{ass:1} implies the branch jump at which $\commarkeddisc{v^i_j}$ attaches, call it $(p^i_j,q^i_j)$, must satisfy
  \begin{equation}
    \label{eq:19}
    \ind (p^i_j, q^i_j)\geq 3.
  \end{equation}
  By Proposition \ref{prop:10}, each $[\markedstrip{u_i}]$ belongs to a smooth moduli space and hence\footnote{Note that constant strips are prohibited (see Remark \ref{rmk:7}) and hence $\rr$ translation is a free action on the parametrized moduli spaces, so these moduli spaces have dimension at least one.}
  \begin{equation}
    \label{eq:20}
    1\leq \dim \conntrajpar{}{\chord_{i,-},\chord_{i,+};\brjmptypes_i} = \ind\chord_{i,-}-\ind\chord_{i,+}-\sum_{j=1}^{|\ptsii_i|}\ind \brjmptypes_i(j)+|\ptsii_i|.
  \end{equation}
  Inequality \eqref{eq:19} implies $\ind \brjmptypes_i(j)\geq 3$ for all $i,j$.
  Thus
  \begin{eqnarray}
    \label{eq:29}
    1&=&\ind\chord_--\ind\chord_+=\sum_i (\ind\chord_{i,-}-\ind\chord_{i,+})\nonumber\\
    &\geq & N +\sum_{i,1\leq j\leq |\ptsii_i|}\ind \brjmptypes_i(j)-\sum_i|\ptsii_i|\geq N +\sum_i 2|\ptsii_i|.
  \end{eqnarray}
  This implies $N=1$ and  $\ptsii_1=\emptyset$.

  Thus $\commarkedstrip{u}=([\markedstrip{u_1}])$; that is, $[\markedstrip{u_n'}]$ has a subsequence converging to $[\markedstrip{u_1}]$.
\end{proof}

\begin{lemma}
  \label{lemma:9}
  $\fd^2=0$.
\end{lemma}
\begin{proof}
  The standard proof of this involves two parts, gluing and compactness.
  We prove the compactness part here; the gluing works in the same way as in the embedded case.
  The relevant compactness result we need to prove is the following: 
  If $\ind\chord_-=\ind\chord_++2$, then
  \begin{equation}
    \label{eq:21}
    \comconntraj{}{\chord_-,\chord_+}=\conntraj{}{\chord_-,\chord_+}\coprod\left(\coprod_{\ind\chord=\ind\chord_--1}\conntraj{}{\chord_-,\chord}\times\conntraj{}{\chord,\chord_+}\right).
  \end{equation}
  In other words, the moduli space can be compactified by adding in broken trajectories consisting of two strips and no disc bubbles.
  
  To prove this, arguing as in the proof of Lemma \ref{lemma:8}, let $[\markedstrip{u_n'}]\in\conntraj{}{\chord_-,\chord_+}$ be a sequence.
  It has a subsequence that converges to some $\commarkedstrip{u}\in\comconntraj{}{\chord_-,\chord_+}$ as before.
  The inequality \eqref{eq:29} in this case becomes
  \begin{displaymath}
    2\geq N +\sum_i 2|\ptsii_i|.
  \end{displaymath}
  Since $N\geq 1$, the only possibilities are $N=1$ or $2$ and all $\ptsii_i=\emptyset$.
  By regularity, if $N=2$ then necessarily $\ind \chord_{1,+}=\ind\chord_--1$.
  This proves the decomposition \eqref{eq:21} and hence the lemma.
\end{proof}

\subsection{Invariance of $\hf(\imm)$}
\label{sec:invariance-hfimm}
In this subsection we show that the Floer cohomology does not depend on the choice of generic Floer data $\h,\J$.
We use the standard continuation method to prove this.
Again, the difference from the embedded case is we need to show that disc bubbles, which might appear in the Gromov compactification attached to a Type II point, do not cause a problem.
The positivity Condition \ref{ass:1} again plays a key role.

First we introduce the notion of continuation data.
For the rest of this section, let $\h_0,\J_0$ and $\h_1,\J_1$ be two sets of admissible Floer data with $\J_0,\J_1$ both regular.
Smooth families
\begin{equation}
  \label{eq:22}
  \dep{\h}{s}{s}=\dep{H}{s,t}{s,t},\quad \dep{\J}{s}{s}=\dep{J}{s,t}{s,t},\quad 0\leq s,t\leq 1
\end{equation}
are called continuation data if $\h_s=\h_0$, $\J_s=\J_0$ for $s$ near $0$ and $\h_s=\h_1,\J_s=\J_1$ for $s$ near 1, and $H_{s,t}=0$ and $J_{s,t}=J_{\mnfld}$ near the boundary of $\mnfld$.
We can view $\h_s,\J_s$ as being defined for all $s\in\rr$.

Our goal is to prove the following proposition, which says the the Floer cohomology does not depend on the choice of Floer data.
\begin{proposition}
  \label{prop:6}
  A generic choice of continuation data induces a chain map
  \begin{equation}
    \label{eq:23}
    \Phi:\cf(\imm;\h_1,\J_1)\to \cf(\imm;\h_0,\J_0)
  \end{equation}
  which induces an isomorphism on cohomology.
\end{proposition}

Before defining $\Phi$ we first introduce some moduli spaces.
Given $\chord_0\in\chords{\h_0}$ and $\chord_1\in\chords{\h_1}$, and branch jump types $\brjmptypes$, let
\begin{equation}
  \label{eq:28}
  \conntraj{}{\chord_0,\chord_1;\brjmptypes;\dep{\h}{s}{s},\dep{\J}{s}{s}}
\end{equation}
denote the set of all pairs $(\markedstrip{u},\ptsii)$ where $\markedstrip{u}=\markedstripstd$ satisfies
\begin{gather}
  \label{eq:24}
  \pd{u}{s}+J_{s,t}(u)\left(\pd{u}{t}-\hamvf{H_{s,t}}\right)=0,\\
\int\biggl|\pd{u}{s}\biggr|^2<\infty,\nonumber \\
\lim_{s\to-\infty}u(s,t)=\chord_0(t),\quad \lim_{s\to\infty}u(s,t)=\chord_1(t)\quad \text{(uniformly in $t$)}.\nonumber
\end{gather}
Analogously to Section \ref{sec:fredh-index-theory}, this moduli space can be set up in a standard functional analytic framework.
The relevant result is that the moduli space is the zero set of an operator $\dbarstripcont$ with Fredholm index
\begin{equation}
  \label{eq:25}
  \ind\dbarstripcont = \ind\chord_1-\ind\chord_0-\sum_i\ind\brjmptypes(i)+|\ptsii|,
\end{equation}
and the linearization of $\dbarstripcont$ at a solution has the same form as in Proposition \ref{prop:5}.

The map $\Phi$ is defined by counting rigid elements of $\conntraj{}{\chord_0,\chord_1;\dep{\h}{s}{s},\dep{\J}{s}{s}}$ (i.e., \eqref{eq:28} with $\brjmptypes=\emptyset$ so there are no Type II points), see equation \eqref{eq:27}.
For this to make sense, the moduli space needs to be regular, which we discuss now.
First, we remark that constant strips pose a problem for regularity.
They are easy to avoid though---they can only exist if $\hamvf{H_{s,t}}$ is independent of $s$ on some subset of $\mnfld$ containing a chord.
We call $\dep{\h}{s}{s}=\dep{H}{s,t}{s,t}$ regular if it does not admit constant strips; generic $\dep{\h}{s}{s}$ are regular.
Since $\dep{\J}{s}{s}=\dep{J}{s,t}{s,t}$ is essentially a domain dependent almost complex structure, standard methods imply that there exists a Baire set of regular almost complex structures.
Then, for regular $\dep{\h}{s}{s}, \dep{\J}{s}{s}$, we have: for each $\chord_0,\chord_1,\brjmptypes$, the moduli space $\conntraj{}{\chord_0,\chord_1;\brjmptypes;\dep{\h}{s}{s},\dep{\J}{s}{s}}$ is a smooth manifold with dimension given by \eqref{eq:25}.
Note that the moduli space is not invariant under $\rr$ translation.

Now define 
\begin{displaymath}
\Phi=\Phi_{\dep{\h}{s}{s},\dep{\J}{s}{s}}:\cf(\imm;\h_1,\J_1)\to\cf(\imm;\h_0,\J_0)
\end{displaymath}
by the formula
\begin{equation}
  \label{eq:27}
  \Phi(\chord_1)=\sum_{\ind \chord_0=\ind \chord_1}|\conntraj{}{\chord_0,\chord_1;\dep{\h}{s}{s},\dep{\J}{s}{s}}|\cdot\chord_0.
\end{equation}
We show that $\Phi$ is well-defined; that is, $\conntraj{}{\chord_0,\chord_1;\dep{\h}{s}{s},\dep{\J}{s}{s}}$ is compact when $\ind\chord_0=\ind\chord_1$.
To this end, consider a sequence of elements.
The energy of a strip $\markedstrip{u}$ is 
\begin{equation}
  \label{eq:10}
  E(\markedstrip{u})=\int \biggl|\pd{u}{s}\biggr|^2ds\wedge dt.
\end{equation}
It is straightforward to show that this is bounded by a constant (depending on $\dep{\h}{s}{s}$) plus $\action(\chord_0)-\action(\chord_1)$, where $\action$ is defined as in \eqref{eq:5}.
Hence there exists a subsequence which converges to an element of $\comconntraj{}{\chord_0,\chord_1;\dep{\h}{s}{s},\dep{\J}{s}{s}}$.
Here, the compactification is similar to that described in Definition \ref{dfn:16}.
Precisely, elements of the compactification consist of up to three different types of broken strips, with disc bubbles, which connect together:
\begin{itemize}
\item Possibly an element $\commarkedstrip{u_0}\in\comconntraj{}{\chord_0,\chord_0';\h_0,\J_0}$.
\item An element $\markedstrip{u}\in\conntraj{}{\chord_0',\chord_1';\brjmptypes;\dep{\h}{s}{s},\dep{\J}{s}{s}}$ where $\brjmptypes:\sett{1,\ldots,m}\to\brjmps$ for some $m$, along with stable discs $\commarkeddisc{v^i_j}\in\comholdiscs{}{J_{s_j,i},\brjmptypes_j^i}$ that attach to the strip at branch jumps.
  Here $i$ is $0$ or $1$ depending whether the disc attaches to the bottom or top boundary of the strip, $s_j$ is the real part of the corresponding Type II marked point, and $\brjmptypes^i_j:\sett{1}\to\brjmps$ is the type of the branch jump at which the disc attaches.
\item Possibly an element $\commarkedstrip{u_1}\in\comconntraj{}{\chord_1',\chord_1;\h_1,\J_1}$.
\end{itemize}
If the limit has a $\commarkedstrip{u_0}$ component, then by regularity of $\h_0,\J_0$ and an argument similar to the proof of Lemma \ref{lemma:8} we have
\begin{displaymath}
  \ind\chord_0-\ind\chord_0'\geq 1.
\end{displaymath}
Likewise if $\commarkedstrip{u_1}$ is a component then
\begin{displaymath}
  \ind\chord_1'-\ind\chord_1\geq 1.
\end{displaymath}
By the regularity of $\dep{\h}{s}{s},\dep{\J}{s}{s}$,
\begin{displaymath}
  0\leq \ind\chord_0'-\ind\chord_1'-\sum_{j=1}^m\ind\brjmptypes(j)+m,
\end{displaymath}
hence
\begin{displaymath}
  \ind\chord_0'-\ind\chord_1'\geq \sum\ind\brjmptypes(j)-m\geq 2m.
\end{displaymath}
Since $\ind\chord_0=\ind\chord_1$, we must have $m=0$ and there is not a $\commarkedstrip{u_0}$ or $\commarkedstrip{u_1}$ component.
That is, the limit is an element of $\conntraj{}{\chord_0,\chord_1;\dep{\h}{s}{s},\dep{\J}{s}{s}}$ and compactness is proved.

To show that $\Phi$ is a chain map, consider the compactification of the moduli spaces 
\begin{displaymath} 
\conntraj{}{\chord_0,\chord_1;\dep{\h}{s}{s},\dep{\J}{s}{s}}
\end{displaymath} when $\ind\chord_0=\ind\chord_1+1$.
Arguing as above, we see that for the limit of a sequence, the only possibilities are that $m=0$, and there is either no $\commarkedstrip{u_0},\commarkedstrip{u_1}$ components or just one such component.
Furthermore, from the proof of Lemma \ref{lemma:8}, if one of these components exists then it must not have any Type II marked points.
In other words, $\conntraj{}{\chord_0,\chord_1;\dep{\h}{s}{s},\dep{\J}{s}{s}}$ can be compactified by adding in
\begin{displaymath}
  \coprod_{\ind\chord_0'=\ind\chord_0-1}\conntraj{}{\chord_0,\chord_0';\h_0,\J_0}\times \conntraj{}{\chord_0',\chord_1;\dep{\h}{s}{s},\dep{\J}{s}{s}}
\end{displaymath}
and
\begin{displaymath}
  \coprod_{\ind\chord_1'=\ind\chord_1+1}\conntraj{}{\chord_0,\chord_1';\dep{\h}{s}{s},\dep{\J}{s}{s}}\times \conntraj{}{\chord_1',\chord_1;\h_1,\J_1}.
\end{displaymath}
Along with a standard gluing result, this proves that $\Phi$ is a chain map.

Next, we want to show that $\Phi$ induces an isomorphism on cohomology.
To do this we first need to consider homotopies of continuation data.
Let $\dep{\h}{0,s}{s},\dep{\J}{0,s}{s}$ and $\dep{\h}{1,s}{s},\dep{\J}{1,s}{s}$ be two sets of continuation data from $\h_0,\J_0$ to $\h_1,\J_1$.
A homotopy between them consists of families
\begin{displaymath}
  \dep{\h}{r,s}{r,s}=\dep{H}{r,s,t}{r,s,t},\quad \dep{\J}{r,s}{r,s}=\dep{J}{r,s,t}{r,s,t}
\end{displaymath}
such that $\h_{r,0}=\h_0, \h_{r,1}=\h_1$ and $\J_{r,0}=\J_0, \J_{r,1}=\J_1$ for $0\leq r\leq 1$, and $H_{r,s,t}=0$ and $J_{r,s,t}=\jstd$ near the boundary of $\mnfld$.
In other words, for each fixed $r$, $\dep{\h}{r,s}{s}, \dep{\J}{r,s}{s}$ is a continuation from $\h_0,\J_0$ to $\h_1,\J_1$.

The homotopy $\dep{\h}{r,s}{r,s},\dep{\J}{r,s}{r,s}$ induces a chain homotopy
\begin{displaymath}
  \Psi:\cf(\imm;\h_1,\J_1)\to\cf(\imm;\h_0,\J_0)
\end{displaymath}
between $\Phi_{\dep{\h}{0,s}{s},\dep{\J}{0,s}{s}}$ and $\Phi_{\dep{\h}{1,s}{s},\dep{\J}{1,s}{s}}$, defined as follows.
Consider the parametrized moduli spaces
\begin{displaymath}
  \conntraj{}{\chord_0,\chord_1;\brjmptypes;\dep{\h}{r,s}{r,s},\dep{\J}{r,s}{r,s}}=\coprod_{0<r<1}\conntraj{}{\chord_0,\chord_1;\brjmptypes;\dep{\h}{r,s}{s},\dep{\J}{r,s}{s}}\times\sett{r},
\end{displaymath}
where the moduli spaces on the right-hand side are of the type in \eqref{eq:28}.
Similar to before, for generic homotopies this moduli space is regular and has dimension
\begin{displaymath}
  \ind\chord_0-\ind\chord_1-\sum_{i=1}^m\ind\brjmptypes(i)+m+1
\end{displaymath}
where $ \brjmptypes:\sett{1,\ldots,m}\to \brjmps$.
If $m=0$ we suppress $\brjmptypes$ from the notation.
$\Psi$ is defined by counting the number of rigid elements,
\begin{displaymath}
  \Psi(\chord_1)=\sum_{\ind\chord_0=\ind\chord_1-1}|\conntraj{}{\chord_0,\chord_1;\dep{\h}{r,s}{r,s},\dep{\J}{r,s}{r,s}}|\cdot\chord_0.
\end{displaymath}

We claim that $\Psi$ is well-defined and gives a chain homotopy between the two chain maps induced by the different sets of continuation data.
Let us just sketch how to prove that $\Psi$ is well-defined; the fact that $\Psi$ is a chain homotopy can be proved in a similar way by identifying the boundary components of the moduli space of dimension 1.
To prove that $\Psi$ is well-defined we need to show that $\conntraj{}{\chord_0,\chord_1;\dep{\h}{r,s}{r,s},\dep{\J}{r,s}{r,s}}$ is compact when $\ind\chord_0=\ind\chord_1-1$.
Let $(\markeddisc{u_n},r_n)$ be a sequence of elements; we need to show it has a convergent subsequence.
The energy is bounded for the same reason that \eqref{eq:10} is bounded.
By regularity of continuation data $\dep{\h}{r,s}{s},\dep{\J}{r,s}{s}$ for $r=0,1$, the sequence $r_n$ must be bounded away from $0$ and $1$.
Thus we may assume it converges to $r$ with $0<r<1$.
We may then assume that $\markeddisc{u_n}$ converges to an element of $\comconntraj{}{\chord_0,\chord_1;\dep{\h}{r,s}{s},\dep{\J}{r,s}{s}}$.
As before, elements of this compactified moduli space consists of broken strips plus disc bubbles.
Thus the limit of $\markeddisc{u_n}$ consists of
\begin{itemize}
\item Possibly a component $\commarkedstrip{u_0}\in\comconntraj{}{\chord_0,\chord_0';\h_0,\J_0}$.
\item A central component $\markeddisc{u}\in\conntraj{}{\chord_0',\chord_1';\brjmptypes;\dep{\h}{r,s}{s},\dep{\J}{r,s}{s}}$ plus disc bubbles $\commarkeddisc{v^i_j}$ attached to all the branch jumps determined by $\brjmptypes:\sett{1,\ldots,m}\to\brjmps$.
\item Possibly a component $\commarkedstrip{u_1}\in\comconntraj{}{\chord_1',\chord_1;\h_1,\J_1}$.
\end{itemize}
As before, if $\commarkedstrip{u_0}$ is present then $\ind\chord_0-\ind\chord_0'\geq 1$ and if $\commarkedstrip{u_1}$ is present then $\ind\chord_1'-\chord_1\geq 1$.
By regularity of the parametrized moduli space we also have
\begin{displaymath}
  0\leq \ind\chord_0'-\ind\chord_1'-\sum_{i=1}^m\ind \brjmptypes(i)+m+1,
\end{displaymath}
and thus 
\begin{displaymath}
  \ind \chord_0'-\ind\chord_1'\geq -1 +2m.
\end{displaymath}
Since we assumed that $\ind\chord_0-\ind\chord_1=-1$, the only possibility is that $m=0$ and $\commarkedstrip{u_0},\commarkedstrip{u_1}$ are not present.
Thus the limit curve is an element of the moduli space and compactness is proved.

Finally to complete the proof that the chain homotopy $\Phi$ induced by continuation data induces an isomorphism on cohomology, we appeal to Floer's original argument.
This argument goes as follows.
Suppose given three sets of Floer data
\begin{displaymath}
  \h_0,\J_0\quad \h_1,\J_1\quad \h_2,\J_2
\end{displaymath}
and two sets of continuation data
\begin{displaymath}
  \dep{\h}{s}{0\leq s \leq 1},\dep{\J}{s}{0\leq s \leq 1} \quad \dep{\h'}{s}{1\leq s\leq 2},\dep{\J'}{s}{1\leq s\leq 2}.
\end{displaymath}
For large $r$ define the concatenated continuation data $\dep{\h\#\h'}{r,s}{s}, \dep{\J\#\J'}{r,s}{s}$ by
\begin{displaymath}
  \h\#\h'_{r,s}=\left\{
    \begin{array}{ll}
      \h_{s+r} & s\leq -r+1\\
      \h_1 & -r+1\leq s\leq r+1\\
      \h_{s-r}' & s\geq r+1,
    \end{array}
  \right.
\end{displaymath}
and $\J\#\J'_{r,s}$ is defined in a similar way.

The concatenated continuations are clearly homotopic for different values of $r$, hence they all induce homotopic chain maps.
For large $r$, a gluing argument shows that the chain map induced by the concatenation is equal to the composition of the chain maps induced by the continuations $\dep{\h}{s}{s},\dep{\J}{s}{s}$ and $\dep{\h'}{s}{s},\dep{\J'}{s}{s}$.

Now given a continuation $\dep{\h}{s}{s},\dep{\J}{s}{s}$ let $\dep{\h'}{s}{s},\dep{\J'}{s}{s}$ be the continuation run backwards.
Composition gives a chain map 
\begin{displaymath}
  \cf(\imm;\h_0,\J_0)\to\cf(\imm;\h_0,\J_0).
\end{displaymath}
From above, this composition is chain homotopic to the map induced by the concatenation.
The concatenation is clearly homotopic to a small perturbation of the identity continuation, which induces the identity map on cohomology.
Thus the maps $\Phi$ induced by continuations must induce isomorphisms on cohomology.
This completes the proof of Proposition \ref{prop:6}.

\subsection{Euler characteristic}
\label{sec:thurst-benn-numb}
In this section we make some simple remarks on the Euler characteristic of $\hf(\imm)$.
\begin{lemma}
  \label{lemma:12}
  Assume $\lag$ is orientable and let $n=\dim\lag$.
  If $n$ is odd then $\ec(\hf(\imm))=0$, and if $n$ is even then 
  \begin{displaymath}
    \ec(\hf(\imm))=\ec(\lag)+(-1)^{n/2}2\wi(\imm),
  \end{displaymath}
  where $\wi(\imm)$ is the algebraic number of self-intersection points of $\imm$.
\end{lemma}
\begin{proof}
First observe that choosing a Hamiltonian as in Lemma \ref{lemma:11} shows that $\ec(\hf(\imm))=\ec(\lag)+\sum_{(p,q)\in\brjmps}(-1)^{\ind(p,q)}$.
Since $\ind(p,q)+\ind(q,p)=n$, if $n$ is odd then $\ec(\hf(\imm))=\ec(\lag)+0=\ec(\lag)=0$.

Assume now that $n$ is even and let $x\in\imbrjmps$ with $\imm^{-1}(x)=\sett{p,q}$.
Define $\sign(x)=\pm1$ depending on whether $\imm_*T_p\lag\oplus \imm_*T_q\lag=\pm T_x\mnfld$ as oriented vector spaces.
Note that $\imm_*T_p\lag\oplus \imm_*T_q\lag=\imm_*T_q\lag\oplus \imm_*T_p\lag$ as oriented vector spaces, and changing the orientation of $\lag$ does not change the orientation of this direct sum.
Thus $\sign(x)$ is well-defined as long as $\lag$ is orientable, and does not depend on the actual orientation.
By definition $\wi(\imm)=\sum_{x\in\imbrjmps}\sign(x)$.
From Definition \ref{dfn:8} it follows that $\sign (x)=(-1)^{\ind (p,q)+n(n-1)/2}=(-1)^{\ind (p,q)+n/2}$ (since $n$ is even).
Thus $\sum_{(p,q)}(-1)^{\ind(p,q)} =(-1)^{n/2}2\sum_{x}\sign (x)=(-1)^{n/2}2\wi(\imm)$.
\end{proof}

Generally speaking, whenever $\imm:\lag\to\cc^n$ is a regular immersion of an oriented $n$-manifold, it is shown in \cite{MR0103478} that $2\wi(\imm)=-(-1)^{n(n-1)/2}\ec(\nu)$, where $\chi(\nu)$ is the Euler characteristic of the oriented normal bundle $\nu$ of the immersion.\footnote{The equation proved in \cite{MR0103478} has a different sign due to different orientation conventions.}
If $\imm$ is a Lagrangian immersion then $\nu=(-1)^{n(n-1)/2}T\lag$ as an oriented vector bundle.
It is interesting to note that this puts topological restrictions on possible Lagrangian immersions and embeddings; for example, the torus is the only Riemann surface which can be embedded into $\cc^2$ as a Lagrangian.
More generally, we have
\begin{lemma}
  \label{lemma:13}
  Let $\imm:\lag\to\mnfld$ be an orientable Lagrangian immersion (not necessarily graded or exact) and let $n=\dim\lag$ be even.
  Then $\ec(\lag)+(-1)^{n/2}2\wi(\imm)=[\imlag]\cdot[\imlag]$.
  In particular, if $\sh_n(\mnfld,\qq)=0$ then this number is $0$.
\end{lemma}
\begin{proof}
  Perturbing $\imlag$ slightly by a Hamiltonian of the type in Lemma \ref{lemma:11} and arguing as in the proof of the previous lemma shows that the self-intersection number (as a homology class) of $[\imlag]=\imm_*[\lag]$ is $[\imlag]\cdot[\imlag]=\ec(\lag)+(-1)^{n/2}2\wi(\imm)$.
  Since $[\imlag]\cdot[\imlag]$ depends only the homology class $[\imlag]\in \sh_n(\mnfld,\qq)$, if $\sh_n(\mnfld,\qq)=0$ then $[\imlag]\cdot[\imlag]=0$.
\end{proof}

\section{$A_\infty$-structure}
\label{sec:a_infty-structure-1}
In this section we describe the $A_\infty$ structure on $\cf(\imm)$.
The $A_\infty$ structure consists of maps $\m{k}:\cf(\imm)^{\otimes k}\to\cf(\imm)$ for $k\geq 1$ which satisfy the $A_\infty$ relations.
As in \cite{seidel-fcpclt}, the maps are defined by counting inhomogeneous holomorphic curves.
In our terminology this means that curves contain only Type I points, which are required to converge to Hamiltonian chords.
We call these curves holomorphic polygons.
To prove that the maps $\m{k}$ are well-defined and satisfy the $A_\infty$ relations we need to introduce some moduli spaces which are not needed in the embedded Lagrangian case.
These auxilliary moduli spaces, which we still call holomorphic polygons, are allowed to contain Type I and Type II points.
However, regularity and the positivity condition allow us to prove the following key statement:
a zero- or one-dimensional moduli space of holomorphic polygons without Type II marked points can be compactified without introducing any Type II marked points.
This is analogous to the case of strips (2-gons) used in Section \ref{sec:floer-cohomology} to define the Floer cohomology.

\subsection{Moduli spaces of domains and perturbation data}
\label{sec:moduli-spac-doma}
Let $\domainmod{d+1}$ be the moduli space of discs with $d+1$ Type I marked boundary points, $\comdomainmod{d+1}$ its Deligne-Mumford-Stasheff compactification obtained by adding nodal discs, $\univermod{d+1}\to\domainmod{d+1}$ the universal bundle, and $\comunivermod{d+1}\to\comdomainmod{d+1}$ the corresponding universal bundle.
Readers are referred to Section 9 of \cite{seidel-fcpclt} for details.
The fiber over $r\in\domainmod{d+1}$ is a boundary punctured disc $\univermodfiber{r}{d+1} = \disc\setminus \ptsi$,
where $\ptsi$ is a list of $d+1$ marked points representing the class $r\in\domainmod{d+1}$.
We view the marked points $\ptsi=(\zi_0,\ldots,\zi_d)$ as Type I points, with $\ptsi^-=\sett{\zi_0}$ and $\ptsi^+=\sett{\zi_1,\ldots,\zi_d}$.
We also assume that the marked points are cyclically counterclockwise oriented.

Fix a consistent universal choice of strip-like ends (see Lemma 9.3 in \cite{seidel-fcpclt}) and fix Floer data $\h=\dep{H}{t}{t},\J=\dep{J}{t}{t}$.
A choice of perturbation data consists of, for each $d$,
\begin{itemize}
\item a fiber-wise one-form $\pdK{d+1} \in \Omega^1_{\univermod{d+1}/\domainmod{d+1}}(C^\infty_c(\mnfld))$ such that for any $r\in\domainmod{d+1}$ and $V$ a vector tangent to the boundary of $\univermodfiber{r}{d+1}$, the function $\pdK{d+1}(V)\in C^\infty_c(\mnfld)$ vanishes on $\imlag$,
\item and a domain-dependent almost complex structure $\pdJ{d+1}=\sett{J^{d+1}_z}_{z\in\univermod{d+1}}$.
\end{itemize}
On strip like ends, it is also required that $\pdK{d+1}=H_tdt$ and $J_z^{d+1}=J_t$ where $z=s+it$ is the strip-like coordinate.
We always assume the choice of perturbation data is consistent in the sense of Lemma 9.5 in \cite{seidel-fcpclt}.
Consistent perturbation data induces perturbation data over $\comunivermod{d+1}$, which we will continue to denote as $\pdK{d+1},\pdJ{d+1}$.

Now we incorporate Type II marked points.
All Type II marked points are outgoing.
Let $\domainmodII{d+1,m}$ be the set
\begin{displaymath}
  \domainmodII{d+1,m}=\set{(r,\ptsii)\in\domainmod{d+1}\times\config{m}{\disc}}{\ptsii\text{ is valid Type II points for }\univermodfiber{r}{d+1}}.
\end{displaymath}
By valid Type II points we mean that $\ptsii$ is disjoint from the set of Type I points, and also they are labelled in counterclockwise order with the first point being the one that appears first after the zeroeth Type I marked point.
There is a corresponding universal bundle $\univermodII{d+1,m}\to\domainmodII{d+1,m}$.
We do not consider compactified versions (indeed, as will be seen later, these moduli spaces play an auxiliary role).
Perturbation data $\pdK{d+1},\pdJ{d+1}$ induces perturbation data for the moduli spaces with Type II points simply by taking it to be constant in the $\config{m}{\disc}$ direction.
We denote this perturbation data using the same symbols, so
\begin{displaymath}
  \pdK{d+1}\in\Omega^1_{\univermodII{d+1,m}/\domainmodII{d+1,m}}(C^\infty_c(\mnfld)),\quad \pdK{d+1}|\univermodfiberII{r}{d+1,m}\times\sett{\ptsii}:=\pdK{d+1}|\univermodfiberII{r}{d+1},
\end{displaymath}
and
\begin{displaymath}
  \pdJ{d+1}=\sett{J^{d+1}_{(z,\ptsii)}}_{(z,\ptsii)\in\univermodII{d+1,m}}, \quad J^{d+1}_{(z,\ptsii)}:=J^{d+1}_z.
\end{displaymath}

For each $d,m$ we can cover $\domainmodII{d+1,m}$ by countably many open sets such that all the fibers over the points in each open set can be given strip-like ends for the Type II points in a smooth way.
The end result is that all the uncompactified moduli spaces of curves can be covered by countably many open sets, and over each open set the universal bundle can be trivialized in such a way that the strip-like ends for Type I and Type II points are trivialized in the sense of equation (9.1) in \cite{seidel-fcpclt}.
The moduli spaces containing only Type I marked points agree with the corresponding moduli spaces in \cite{seidel-fcpclt}.

\subsection{Marked polygons}
\label{sec:marked-polygons}
In this section we set up the space of maps in which the holomorphic polygons used to define the $A_\infty$ operations will lie in.
The maps we will study will generalize the notion of a marked strip (see Section \ref{sec:mark-strips-polyg}) by allowing more Type I points.
To distinguish from a marked disc (see Section \ref{sec:marked-discs}), which has no Type I points, we will call this generalization a \textit{marked polygon}.

First we discuss the domain of a marked polygon.
Let $\ptsi=(\zi_0,\ldots,z_d)\subset\disc$ be a list of Type I points, cyclically counterclockwise oriented, $d\geq 2$. 
Let $\ptsi^-=\sett{z_0}$, $\ptsi^+=\sett{z_1,\ldots,z_d}$.
Fix some choice of strip like ends for these marked points.

\begin{definition}
  \label{dfn:20}
  Fix a Hamiltonian $\h$.
  A \textit{$C^0$-marked polygon} $\markedpoly{u}=(u,\ptsii,\brjmptypes,\ell)$ with domain $\disc\setminus\ptsi$ connecting $\chord_0,\ldots,\chord_d\in\chords{\h}$ consists of the following data.
  \begin{itemize}
  \item A continuous map $u:(\disc\setminus\ptsi,\bdy(\disc\setminus\ptsi))\to(\mnfld,\imlag)$ such that
    \begin{displaymath}
      \lim_{s_j\to \pm\infty} u(s_j,t_j)=\chord_j(t_j),\quad \text{uniformly in $t_j$}.
    \end{displaymath}
    Here $s_j+it_j$ are strip-like coordinates on the strip-like ends for the Type I marked points $\zi_j$, and the limit condition $\pm\infty$ is chosen depending on whether $\zi_j$ is incoming (i.e., $j=0$) or outgoing (i.e., $1\leq j\leq d$).
  \item $\ptsii=(\zii_1,\ldots,\zii_m)\subset\bdy(\disc\setminus\ptsi)$ is a list of Type II marked points, all outgoing. 
    Also, the points are ordered counterclockwise in such a way that $\zii_1$ is the first point of $\ptsii$ that appears after $\zi_0$ in the counterclockwise direction.
  \item A map $\brjmptypes:\sett{1,\ldots,m}\to\brjmps$.
  \item A continuous map $\ell:\bdy(\disc\setminus\ptsi\amalg\ptsii)\to\lag$ such that
    \begin{displaymath}
      \imm\circ\ell = u |\bdy(\disc\setminus\ptsi\amalg\ptsii).
    \end{displaymath}
  \end{itemize}
  Moreover, $\ell$ has a branch jump at each $\zii_i\in\ptsii$ of type $\brjmptypes(i)$.
\end{definition}

The Maslov index of a marked polygon is defined as in Section \ref{sec:maslov-index}.
The analogue of Lemma \ref{lemma:5} is:
\begin{lemma}
  \label{lemma:10}
  The Maslov index of the marked polygon $\markedpoly{u}=(u,\ptsii,\brjmptypes,\ell)$ connecting $\chord_0,\ldots,\chord_d$ is
  \begin{displaymath}
    \mu(\markedpoly{u})=\ind\chord_0-\sum_{i=1}^d\ind\chord_i-\sum_{i=1}^m\brjmptypes(i).
  \end{displaymath}
\end{lemma}

Let $\domainmodII{d+1,m}$ be as in Section \ref{sec:moduli-spac-doma}, with $d+1\geq3,m\geq0$.
In particular, this implies a choice of strip-like ends has been chosen for the Type I points $\ptsi$ for each fiber $\univermodfiberII{(r,\ptsii)}{d+1,m}=\univermodfiber{r}{d+1}=:\disc\setminus\ptsi$ of the universal bundle.
\begin{definition}
  Let $\polys{}{\univermodII{d+1,m},\brjmptypes}$ denote the set of all tuples $(\markedpoly{u},r,\ptsii)$ such that
  \begin{itemize}
    \item $(r,\ptsii)\in\domainmodII{d+1,m}$, and
    \item $\markedpoly{u}=(u,\ptsii,\brjmptypes,\ell)$ is a $C^0$-marked polygon with domain $\univermodfiber{r}{d+1}$.
  \end{itemize}
\end{definition}

As in Section \ref{sec:analyt-setup-mark}, define the subset $\polys{1,p}{\univermodII{d+1,m},\brjmptypes}\subset \polys{}{\univermodII{d+1,m},\brjmptypes}$
consisting of elements of regularity $W^{1,p}$ with respect to a volume form which is $ds\wedge dt$ on each strip-like end. This requires trivializing the domains of the marked polygons in both the Type I and Type II directions; we use the trivializations of $\univermodII{d+1,m}$ discussed at the end of Section \ref{sec:moduli-spac-doma} to do this.
The process is similar to the case of marked strips.

Given perturbation data $\pdK{d+1},\pdJ{d+1}$, there exists a Banach bundle and section
\begin{equation}
  \label{eq:30}
  \dbarpoly:\polys{1,p}{\univermodII{d+1,m},\brjmptypes}\to\polysbundle{0,p}{\univermodII{d+1,m},\brjmptypes}.
\end{equation}
As usual, this is a $C^0$-section of a $C^0$-Banach manifold; there are special trivializations with respect to which the section becomes a Fredholm operator.
The index is
\begin{equation}
  \label{eq:31}
  \ind\chord_0-\sum_{i=1}^d\ind\chord_i-\sum_{i=1}^m\ind\brjmptypes(i)+d+m-2.
\end{equation}

\subsection{Moduli spaces of holomorphic curves}
\label{sec:moduli-spac-holom}
Now we define the moduli spaces of holomorphic polygons.
The holomorphic polygons with only Type I points will be used to define the $A_\infty$ structure, just as in \cite{seidel-fcpclt}.
Suppose given Floer data $\h,\J$, perturbation data $\pdK{d+1},\pdJ{d+1}$, and strip-like ends for $\univermod{d+1}\to\domainmod{d+1}$.
This induces perturbation data $\pdK{d+1},\pdJ{d+1}$ on $\univermodII{d+1,m}\to\domainmod{d+1,m}$ for all $m\geq0$ as in Section \ref{sec:moduli-spac-doma}.
Also, all Type I marked points inherit a consistent choice of strip-like ends.
\begin{definition}
  \label{dfn:21}
  Let $\chord_0,\ldots,\chord_d\in\chords{\h}$ and $\brjmptypes:\sett{1,\ldots,m}\to\brjmps$ be given.
  Then
  \begin{displaymath}
    \holpolys{\chord_0,\ldots,\chord_d;\brjmptypes}=\holpolys{\chord_0,\ldots,\chord_d;\brjmptypes;\pdK{d+1},\pdJ{d+1}}
  \end{displaymath}
  denotes the set of all $(\markedpoly{u},r,\ptsii)\in \polys{}{\univermodII{d+1,m},\brjmptypes}$ such that
  \begin{itemize}
    \item $u|\univermodfiber{r}{d+1}\setminus\ptsii$ is smooth,
    \item $u$ connects the chords $\chord_0,\ldots,\chord_d$ (in the sense of the first bullet point of Definition \ref{dfn:20}),
    \item $u$ satisfies the inhomogeneous Cauchy-Riemann equation
      \begin{displaymath}
        \left(du-\hamvf{\pdK{d+1}}\right)^{(0,1)}=0,
      \end{displaymath}
    \item and $u$ has finite energy:
      \begin{equation}
        \label{eq:14}
        E(u)=\frac{1}{2}\int_{\univermodfiber{r}{d+1}\setminus\ptsii} |du-\hamvf{\pdK{d+1}}|^2<\infty.
      \end{equation}
  \end{itemize}
  Here $\hamvf{\pdK{d+1}}$ denotes the 1-form on $\univermodfiber{r}{d+1}$ with values in $u^*T\mnfld$ defined by
  \begin{displaymath}
    \hamvf{\pdK{d+1}}(z)(\xi)=\hamvf{\pdK{d+1}(z)(\xi)}(u(z)),\quad \xi\in T_z\univermodfiber{r}{d+1}.
  \end{displaymath}
  Note that $\pdK{d+1}(z)(\xi)$ is a Hamiltonian function.
  In case $m=0$ we simply denote this space as
  \begin{displaymath}
    \holpolys{\chord_0,\ldots,\chord_d}=\holpolys{\chord_0,\ldots,\chord_d;\pdK{d+1},\pdJ{d+1}}.
  \end{displaymath}

\end{definition}

\begin{proposition}
  \label{prop:12}
  For each $d\geq2$, there exists a generic set of consistent perturbation data $\pdK{d+1},\pdJ{d+1}$ such that the moduli spaces $\holpolys{\chord_0,\ldots,\chord_d;\brjmptypes}$ are regular for all $\chord_0,\ldots,\chord_d$ and all $\brjmptypes:\sett{1,\ldots,m}\to\brjmps$.
  In particular, $\holpolys{\chord_0,\ldots,\chord_d;\brjmptypes}$ is a smooth manifold of dimension
  \begin{displaymath}
    \ind\chord_0-\sum_{i=1}^d\ind\chord_i-\sum_{i=1}^m\ind\brjmptypes(i)+d+m-2.
  \end{displaymath}
\end{proposition}
\begin{proof}
  Theorem \ref{thm: asymptotic behavior} implies that elements of the moduli space have $W^{1,p}$-regularity near Type II points.
  It is standard that they have similar regularity near Type I points.
  Thus $\holpolys{\chord_0,\ldots,\chord_d;\brjmptypes}$ coincides with the zero-set of the section in \eqref{eq:30}.
  To prove regularity, we need to show that, for each $d\geq 2$, generic $\pdJ{d+1}$ and $\pdK{d+1}$ can be found so that the operators for all $m\geq 0$ in $\eqref{eq:30}$ are transverse to the zero-section.
  For generic $\pdK{d+1}$, the moduli spaces cannot contain maps which are constant away from the strip-like ends.
  Away from the strip-like ends, $\pdJ{d+1}$ is a domain dependent complex structure. 
  For each $m$, there is thus a Baire set of $\pdK{d+1},\pdJ{d+1}$ which makes $\eqref{eq:30}$ transverse to the zero-section.
  Taking the intersection over all $m$ of these Baire sets then results in a Baire set which makes all the operators transverse.
  The index calculation is a combination of those appearing in \cite{seidel-fcpclt} and \cite{MR2785840}.
\end{proof}

Now we discuss Gromov compactness of $\holpolys{\chord_0,\ldots,\chord_d}$ (Gromov compactness for $\brjmptypes\neq\emptyset$ is not needed).
As in Section \ref{sec:grom-comp-moduli}, the difference from the embedded Lagrangian case in \cite{seidel-fcpclt} is that now disc bubbles attached to Type II points can appear in the compactification.
First we remark that every element of $\holpolys{\chord_0,\ldots,\chord_d}$ has an a priori energy bound (energy defined by \eqref{eq:14}).
In fact, more generally, an element  $(\markedpoly{u},r,\ptsii)\in \holpolys{\chord_0,\ldots,\chord_d;\brjmptypes}$ has energy
\begin{equation}
  \label{eq:15}
  E(u)=\action(\chord_0)-\sum_{i\geq 1}\action(\chord_i)-\sum_{j}E(\brjmptypes(j))-\int_{\univermodfiber{r}{d+1}}R_{\pdK{d+1}}(u).
\end{equation}
This is a generalization of equation (8.12) in \cite{seidel-fcpclt} and follows from Stokes' theorem, the conditions imposed on the perturbation data $\pdK{d+1}$ in Section \ref{sec:moduli-spac-doma}, and the fact that the perturbation data $\pdK{d+1,m}$ on moduli spaces of domains with Type II points is just the pull back under the forgetful map of the perturbation data $\pdK{d+1}$ defined on moduli spaces of domains with no Type II points.
In the last integral, $R_{\pdK{d+1}}$ is a 2-form on $\univermodfiber{r}{d+1}$ with values in compactly supported functions on $\mnfld$; it can be interpreted as the curvature of a Hamiltonian fibration, see Section (8g) of \cite{seidel-fcpclt}.
If, at the point $z\in\univermodfiber{r}{d+1}$, $R_{\pdK{d+1}}(z)$ is the form $H ds\wedge dt$, then $R_{\pdK{d+1}}(u)(z)$ denotes the form $H(u(z)) ds\wedge dt$. 
In particular, this makes sense even at Type II points, where $u$ is continuous but not smooth.
Also, $R_{\pdK{d+1}}(u)$ vanishes on the strip-like ends of $\univermodfiber{d+1}{r}$.
Thus the curvature integral in $\eqref{eq:15}$ can be bounded independently of $u$.
Since the perturbation data on the moduli space of domains is chosen to be consistent with the compactification of the moduli space of domains (Lemma 9.5 in \cite{seidel-fcpclt}), the integral can also be bounded independently of $r\in\domainmod{d+1}$ (because the bound varies continuously with $r$, and by consistency the bound can be extended continuously for $r$ in the compactified moduli space).
Thus \eqref{eq:15} can be bounded independently of $u$.

This takes care of the energy bound.
Gromov compactness of polygons is then a generalization of the compactification of strips described in Definition \ref{dfn:16}.
Briefly, elements can be described as trees of holomorphic polygons with disc bubbles attached at Type II points.
The disc bubbles are the feature not present in the embedded Lagrangian case.
The precise statement is
\begin{proposition}
  \label{prop:13}
  Elements of the compactified moduli space $\comholpolys{\chord_0,\ldots,\chord_d}$
  consist of (equivalence classes of) tuples $(T,\sett{((\markedpoly{u_i},r_i,\ptsii_i),F_i,\commarkeddisc{v^{i}_1},\ldots,\commarkeddisc{v^{i}_{m_i}})}_{i\in Ver(T)})$ where 
  \begin{itemize}
    \item $T$ is a planar tree with a distinguished root vertex and a distinguished ordering $1,2,\ldots$ of the edges connected to the root. The distinguished ordering respects the planar structure in the sense that the order proceeds in counterclockwise direction.
    \item $(\markedpoly{u_i}=(u_i,\ptsii_i,\brjmptypes_i,\ell_i),r_i,\ptsii_i)\in \holpolys{\chord_0^i,\ldots,\chord_{d_i}^i;\brjmptypes_i}$ for each vertex $i$ of $T$, 
    \item for each nonroot vertex $i\in Ver(T)$, $F_i\subset \sett{1,\ldots,d_i}$; and for the root vertex $r\in Ver(T)$, $F_r\subset \sett{0,\ldots,d_r}$ with $0\in F_r$; moreover $|F_i|+\mathrm{valency}(i)=d_i+1$ for all vertices $i\in Ver(T)$ ($F_i$ encodes which chords of a component are among the original set of chords $\chord_0,\ldots,\chord_d$),
    \item $\commarkeddisc{v^{i}_1},\ldots,\commarkeddisc{v^{i}_{m_i}}$ are elements of the compactified moduli spaces of holomorphic discs with one incoming Type II marked point (see Definition \ref{dfn:15}) that attach to the branch jump points of $\markedpoly{u_i}$ as specified by $\brjmptypes_i:\sett{1,\ldots,m_i}\to\brjmps$; moreover, each $\commarkeddisc{v^{i}_j}$ is $J_{\zii_{i,j}}$ holomorphic, where $J_{\zii_{i,j}}$ is domain independent and determined by the perturbation data $\pdJ{d_i+1}$ in the obvious way.
  \end{itemize}
  Moreover, the chords are required to be compatible in the following sense:
  \begin{itemize}
  \item $\sum_{i\in Ver(T)}|F_i|=d+1$.
  \item    The fact that $T$ has a distinguished root vertex implies that the edges can be canonically oriented in the outward direction from the root. If we label the incoming edge at a non-root vertex as edge 0, then because $T$ is planar each other edge gets a unique number, starting at $1$, by proceeding in counterclockwise order. 
    The outgoing edges from the root get a number from the first bullet point above.
    If an outgoing edge at $i\in Ver(T)$ is edge number $a\geq 1$, and it connects to vertex $j\in Ver(T)$, we require that $\chord_b^i=\chord_0^j$, where $b$ is the $a^{th}$ number in the set $\sett{1,\ldots,d_i}\setminus F_i$.
  \item The external chords, namely the ones which aren't matched up in pairs by the edges of $T$, in counterclockwise order are precisely $\chord_0,\ldots,\chord_d$. 
    Also, if $r\in Ver(T)$ is the root, then $\chord_0^r=\chord_0$.
  \end{itemize}
  In particular, this means that the domain discs can be glued together along the marked points as specified by the tree, and the resulting map is continuous on the glued domain, and there are $d+1$ remaining Type I marked boundary points which converge to $\chord_0,\ldots,\chord_d$.
\end{proposition}

\begin{proof}[Sketch of proof]
First we briefly explain how the compactness works in the embedded Lagrangian case from \cite{seidel-fcpclt}, where discs attached to Type II points are not needed.
Suppose given a sequence of elements from the moduli space.
We may take a subsequence of maps such that the underlying sequence of moduli of domains is convergent with limit $r_\infty\in\comdomainmod{d+1}$.
Each map in the sequence has domain which is the fiber $\comunivermod{d+1}_{r_n}$ over $r_n$ of the universal bundle $\comunivermod{d+1}\to\comdomainmod{d+1}$.
Since $r_n\to r_\infty$, the gluing procedure explained in Section (9e) of \cite{seidel-fcpclt} allows one to identify any compact subset of $\comunivermod{d+1}_{r_\infty}$ which is disjoint from the marked points and the nodal points with a similar subset of $\comunivermod{d+1}_{r_n}$, for $n$ large enough.
Now take a sequence of such compact subsets which exhausts $\comunivermod{d+1}_{r_\infty}$.
If, on the compact subsets, the sequence of maps have a gradient bound (gradient with respect to a metric on the domain which is standard in strip-like ends and the thin part of the thick-thin decomposition from Section (9e) of \cite{seidel-fcpclt}) then there exists a convergent subsequence.
This gives a limit map on the domain $\comunivermod{d+1}_{r_\infty}$.
The limit map satisfies the correct inhomogeneous Cauchy-Riemann equation, namely the one determined by the perturbation data $\pdK{d+1}|\comunivermod{d+1}_{r_\infty}$, because the perturbation data was chosen to be consistent with compactification (of moduli space of domains).
Bubbling off analysis and exactness of the Lagrangian can be used to show that a gradient bound must exist.
The remaining thing to consider is accumulation of energy in the thin part of the domains.
This results in the thin part breaking into a strip (or a sequence of strips) which satisfy Floer's equation.
The strips may be internal in that they connect different polygon components, or external in the sense that they connect a polygon to a chord.
In either case, they are part of the compactification described by the proposition: they correspond to polygon components with $d_i=1$ in the second bullet point.

Now we consider the immersed case.
The proof proceeds in the same way as before, except that now in the bubbling off analysis it is possible for a holomorphic disc bubble to appear. 
A bubble must attach to a holomorphic polygon at a branch point, i.e. a Type II point.
The last thing to check is that the limit curve satisfies the correct inhomogeneous holomorphic curve equation.
This follows because the perturbation data $\pdK{d_i+1,m_i}$ for each polygon component of the limit domain is the pull back under the forgetful map of the perturbation $\pdK{d_i+1}$ for the corresponding domain with no Type II points.
(In other words: the limit curve satisfies the correct inhomogeneous holomorphic curve equation for the same reason that the limit curve in the embedded case does.)
\end{proof}

\subsection{$A_\infty$-structure}
\label{sec:a_infty-structure}
In this subsection we define the $A_\infty$ structure using the moduli spaces of holomorphic polygons with only Type I points.
The definition is the same as the one given in \cite{seidel-fcpclt} for embedded Lagrangians.
However, to show that the maps are well-defined and satisfy the $A_\infty$ relations, we need to show that disc bubbling does not cause a problem.
Similar to the case of Floer cohomology considered in Section \ref{sec:floer-cohomology}, the key tools needed to show this are regularity, Gromov compactness, and the positivity condition.

Given regular Floer data $\h,\J$, the underlying $\zz_2$-vector space for the $A_\infty$ algebra is 
\begin{displaymath}
  \cf(\imm)=\cf(\imm;\h,\J)=\bigoplus_{\chord\in\chords{\h}}\zz_2\cdot\chord.
\end{displaymath}
The grading is given by the index of the chords as defined in Definition \ref{dfn:7}.
For $k\geq1$, define operations 
\begin{displaymath}
  \m{k}:\cf(\imm)^{\otimes k}\to\cf(\imm)
\end{displaymath}
by
\begin{equation}
  \label{eq:33}
  \m{k}:\chord_1\otimes\cdots\otimes\chord_k\mapsto \sum_{\ind\chord_0=\ind\chord_1+\cdots\chord_k+2-k}|\holpolys{\chord_0,\ldots,\chord_k}|\cdot\chord_0.
\end{equation}
These operations are said to satisfy the \textit{$A_\infty$ relations} if, for each $k\geq1$ and all $x_1,\ldots,x_k\in\cf(\imm)$,
\begin{displaymath}
  \sum_{i,j}\m{k-j+1}(x_1,\ldots,x_i,\m{j}(x_{i+1},\ldots,x_{i+j}),x_{i+j+1},\ldots,x_k)=0.
\end{displaymath}

\begin{theorem}
  \label{thm:1}
  $\m{k}$ are well-defined and satisfy the $A_\infty$ relations.
\end{theorem}
\begin{proof}
  We give a sketch of the proof following the usual lines.
  The proof relies on standard gluing results, the regularity Proposition \ref{prop:12}, and the compactness Proposition \ref{prop:13}; it can be seen as a generalization of the arguments used in Section \ref{sec:floer-cohomology}.
  The difference from \cite{seidel-fcpclt} is that we need to rule out disc bubbling; that is, we need to show that zero- and one-dimensional moduli spaces involving only Type I marked points have compactifications which still only involve Type I marked points.

  First consider well-definedness: we need to show that $\holpolys{\chord_0,\ldots,\chord_k}$ is compact when $\ind\chord_0=\sum\ind\chord_i+2-k$.
  Consider a sequence of elements, we may assume they converge to an element of the compactified moduli space described in Proposition \ref{prop:13}.
  Call this element $(T,\sett{\markedpoly{u_i},\commarkeddisc{v^i_1},\ldots,\commarkeddisc{v^i_{m_i}}}_{i\in Ver(T)})$.
  Let $\markedpoly{u_i}=(u_i,\ptsii_i,\brjmptypes_i,\ell_i)$ connect the chords $\chord_{i,0},\ldots,\chord_{i,k_i}$.
  By regularity of $\underline{{u_i}}$, for each $i$ we have
  \begin{displaymath}
    \ind\chord_{i,0}-\sum_{j=1}^{k_i}\ind\chord_{i,j}-\sum_{j=1}^{m_i}\brjmptypes_i(j)+k_i+m_i-2\geq 0.
  \end{displaymath}
  Combining this inequality with the positivity Condition \ref{ass:1}, we get
  \begin{displaymath}
    \ind\chord_{i,0}-\sum_{j=1}^{k_i} \ind\chord_{i,j}\geq \sum_{j=1}^{m_i} \brjmptypes_i(j)-k_i-m_i+2\geq 2m_i-k_i+2.
  \end{displaymath}
  Summing over all $i\in Ver(T)$ then gives
  \begin{eqnarray*}
    2-k&=&\ind\chord_0-\sum_{j=1}^k\ind\chord_j\geq 2\sum_{i=1}^{|Ver(T)|}m_i+2|Ver(T)|-\sum_{i=1}^{|Ver(T)|} k_i\\
    &=& 2\sum_{i=1}^{|Ver(T)|} m_i+2|Ver(T)|-(|Ver(T)|-1+k)\\
    &=&2\sum_{i=1}^{|Ver(T)|} m_i+|Ver(T)|+1-k.
  \end{eqnarray*}
  Thus
  \begin{equation}
    \label{eq:32}
    1\geq 2\sum_{i=1}^{|Ver(T)|}m_i+|Ver(T)|.
  \end{equation}
  The only possibility is that $T$ has one vertex and all $m_i=0$; that is, the limit curve actually lies in $\holpolys{\chord_0,\ldots,\chord_k}$ and compactness is proved.

  To prove the $A_\infty$ relations, consider the moduli spaces $\holpolys{\chord_0,\ldots,\chord_k}$ which are one dimensional; that is, with $\ind\chord_0=\ind\chord_1+\cdots\ind\chord_k+2-k+1$.
  Arguing as before, inequality \eqref{eq:32} now becomes
  \begin{displaymath}
    2\geq 2\sum_{i=1}^{|Ver(T)|} m_i+|Ver(T)|.
  \end{displaymath}
  Since $|Ver(T)|\geq 1$, all $m_i=0$ and $T$ consists of 1 or 2 vertices.
  The elements that have trees with 2 vertices form the boundary of the moduli space; each such tree corresponds to a term in an $A_\infty$ relation.
  A gluing argument gives the converse, and hence the $A_\infty$ relations follow by the fact that the boundary of a 1-dimensional manifold consists of an even number of points.
\end{proof}

\subsection{Units, $A_\infty$ categories, and invariance}
\label{sec:units-a_infty-categ}
In this section we explain why the $A_\infty$ algebra of $\imm$ is independent, up to equivalence, of the choices used to construct it, such as $\h,\J$ and the perturbation data and strip-like ends on each $\univermod{d+1}$.
Two $A_\infty$ algebras are equivalent if there is an $A_\infty$ morphism between them that induces an isomorphism on cohomology.
Such a morphism is actually a homotopy equivalence; conversely, a homotopy equivalence induces an isomorphism on cohomology.

Perhaps the best way to approach this is to view the Lagrangian immersion as an object of the Fukaya category and proceed as in Section 10 of \cite{seidel-fcpclt}.
We present a sketch of the argument here for the convenience of the reader.
First, note that previous arguments can be used to define more general $A_\infty$ structure maps
\begin{displaymath}
  \m{k}:\cf(L_0,L_1)\otimes \cdots \otimes \cf(L_{k-1},L_k)\to\cf(L_0,L_k),
\end{displaymath}
where each $L_i$ is an immersed (or embedded) exact Lagrangian satisfying Assumption \ref{ass:1}.
$\m{k}$ is defined by counting inhomogeneous holomorphic polygons in the usual way.
To ensure that the $\m{k}$'s satisfy the $A_\infty$ relations, the perturbation data for the holomorphic polygons needs to be chosen in a consistent way.
This includes Floer data for each pair of Lagrangians.
We refer the readers to \cite{seidel-fcpclt} for details.
The end result is that we get an $A_\infty$ category in which the immersion $\imm$ is an object.

Suppose that we have chosen Floer data and perturbation data and strip-like ends.
We can think of another choice of this data, for the same $\imm$, as a different object which we will call $\imm'$ for clarity.
To show that the $A_\infty$ algebras for $\imm$ and $\imm'$ are equivalent we need to find elements $f\in\cf(\imm,\imm')$ and $g\in \cf(\imm',\imm)$ such that $\m{2}(f,g)$ is cohomologous to the unit in $\hf(\imm,\imm)$ and $\m{2}(g,f)$ is cohomologous to the unit in $\hf(\imm',\imm')$.
Indeed, existence of such $f$ and $g$ implies that the inclusion of the category with only object $\imm$ (or $\imm'$) into the category with objects $\imm$ and $\imm'$ is an equivalence.
Here, the categories we mean are the honest categories obtained by taking cohomology of the $A_\infty$ categories.
The equivalence of the $A_\infty$ algebras for $\imm$ and $\imm'$ then follows from Theorem 2.9 in \cite{seidel-fcpclt}.

The units of $\hf(\imm,\imm)$ and $\hf(\imm',\imm')$ can be constructed in the following way.
Consider $\imm$.
Define an element $e_\imm\in\cf(\imm,\imm)$ by
\begin{equation}
  \label{eq:37}
  e_\imm=\sum_{\chord:\ind\chord=0}|\holpolys{\chord}|\cdot\chord.
\end{equation}
Here, $\holpolys{\chord}$ is the space of inhomogeneous holomorphic polygons with one incoming Type I marked point that converges to $\chord$.
In Definition \ref{dfn:20} when we defined the notion of holomorphic polygon we assumed that the number of Type I marked points was greater than or equal to 3.
In the present case there is only one Type I marked point; this is dealt with by assuming that the position of the Type I marked point is fixed, say $z_0=-1$, and then the rest of the definition is the same.
The inhomogeneous equation should agree with Floer's equation for the Floer data for $\imm$ on a strip-like end for the marked point, and should be generic elsewhere.
Methods very similar to Section \ref{sec:invariance-hfimm} show that $\m{1}e_\imm=0$, hence $[e_\imm]\in\hf(\imm,\imm)$, and also $[e_\imm]$ is a unit with respect to the multiplication induced on $\hf(\imm,\imm)$ by $\m{2}$.

Likewise, define $f=e_{\imm,\imm'}\in\cf(\imm,\imm')$ using the same equation, except that on the strip-like end the inhomogeneous equation should be Floer's equation for the Floer data for the pair $\imm,\imm'$. 
Similarly, define $g=e_{\imm',\imm}$.
The same methods then show that $[\m{2}(f,g)]=[e_{\imm}]\in\hf(\imm,\imm)$ and $[\m{2}(g,f)]=[e_{\imm'}]\in\hf(\imm',\imm')$.
Thus the $A_\infty$ algebras for $\imm$ and $\imm'$ are equivalent.

By incorporating moving Lagrangian boundary conditions, the same technique can be applied to show that $\imm$ and $\imm'=\phi\circ\imm$ are quasi-isomorphic objects of the Fukaya category when $\phi$ is a Hamiltonian diffeomorphism (we may assume $\phi$ is the identity outside of a compact subset).
Moving Lagrangian boundary conditions are explained in Section (8k) of \cite{seidel-fcpclt}, the immersed case works the same way as the embedded case.
Here is a brief overview.
As before, define $f\in\cf(\imm,\imm')$ by counting inhomogeneous holomorphic discs with one fixed Type I marked point and that satisfy Floer's equation near the marked point.
Since $\imm\neq\imm'$, the discs satisfy a moving Lagrangian boundary condition, determined by picking some path of Hamiltonian diffeomorphisms $\dep{\phi}{t}{t}$ from the identity to $\phi$, and choosing some identification of the parameter space $\sett{t}$ with the boundary of the disc outside the strip-like end.
Then the Lagrangian boundary condition at $t$ is $\phi_t\circ\imm$.
This means that the boundary lift $\ell$ of the disc $\markeddisc{u}=(u,\ptsii,\brjmptypes,\ell)$ should satisfy $\phi_t\circ\imm\circ\ell (t)= u(t)$ for $t\in\bdy\disc\setminus\sett{z_0}$.
The perturbation data needs to have slightly different boundary conditions in the moving Lagrangian case, see formula (8.21) of \cite{seidel-fcpclt}.
The end result is that $[f]\in\hf(\imm,\imm')$ is an isomorphism, hence $\imm$ and $\imm'$ are quasi-isomorphic objects in the Fukaya category.
In particular, the abstract theory (Theorem 2.9 in \cite{seidel-fcpclt}) implies that the $A_\infty$ algebras $\cf(\imm)$ and $\cf(\imm')$ are equivalent.
One could also copy Section (10c) of \cite{seidel-fcpclt} to construct concrete higher order maps that would give an $A_\infty$ morphism between $\cf(\imm)$ and $\cf(\imm')$.
\begin{remark}
  \label{rmk:2}

  There exists a set $\sett{\phi_z}$ of Hamiltonian diffeomorphisms of $\mnfld$ parameterized by $z\in\disc$ such that if   $u:\disc\to\mnfld$ is a map contributing to the $f\in\cf(\imm,\imm')$ defined above, then the map $z\mapsto \phi_z(u(z))$ is a disc with non-moving Lagrangian boundary conditions that satisfies an inhomogeneous holomorphic curve equation.
  See formula (8.22) of \cite{seidel-fcpclt} for more details.
  Thus the analytic details of the moving Lagrangian boundary condition can be reduced to the non-moving case.
  Later on we will consider deformations which are exact but not Hamiltonian.
  In this case, the procedure just outlined does not work.
  Another (less serious) difficulty is that formula (8.21) in \cite{seidel-fcpclt} no longer makes sense.
  Nevertheless, in some cases we will still be able to construct a quasi-isomorphism $f$.
\end{remark}

In summary:
\begin{theorem}
  \label{thm:9}
  Different choices of Floer data and perturbation data will result in quasi-isomorphic $A_\infty$ algebras for the immersed Lagrangian $\imm$.
  Also, if $\phi$ is a Hamiltonian diffeomorphism, then $\imm$ and $\phi\circ\imm$ are quasi-isomorphic objects of the Fukaya category.
  In particular, they have quasi-isomorphic $A_\infty$ algebras.
\end{theorem}

\begin{remark}
  \label{rmk:14}
  Also, independence on $\jstd$ can be shown in the following way.
  Let $\jstd'$ be another valid choice.
  Let $\tilde \mnfld$ be the non-compact manifold obtained by adding a (symplectically) conical end to $\mnfld$.
  There exists complex structures $J$ and $J'$ which agree outside a compact set, are cylindrical outside of a compact set, are compatible with the contact form $\sigma$ on $\bdy M$ (far enough out on the conical end), and satisfy $J|M=\jstd$ and $J'|M=\jstd'$.
  Then working inside $\tilde \mnfld$ the arguments of this section can be applied to show that using $\jstd$ and $\jstd'$ lead to quasi-isomorphic objects.
\end{remark}

Next, we consider invariance under certain exact deformations.
First, some general remarks.
Let $\dep{\imm}{t}{t}$ be a family of Lagrangian immersions.
We call the family an \textit{exact deformation} if, for each $t$, the one form $\beta_t$ on $\lag$ defined by
\begin{displaymath}
  \beta_t=\imm_t^*\sympl\left(\pd{\imm_t}{t},\cdot\right)
\end{displaymath}
is exact.
If one of the immersions, say $\imm_0$, is exact, then the family is an exact deformation if and only if all of the immersions are exact.
Indeed, if $df_0=\imm_0^*\sigma$ and $dh_t=\beta_t$, then
\begin{equation}
  \label{eq:35}
  f_t=f_0+\int_0^t\left(h_s+\sigma\left(\pd{\imm_s}{s}\right)\circ\imm_s\right)ds
\end{equation}
satisfies $df_t=\imm_t^*\sigma$.
A family of immersions is called a Hamiltonian deformation if there exists a family of Hamiltonian diffeomorphisms $\dep{\phi}{t}{t}$ with $\phi_0=id$ such that $\imm_t=\phi_t\circ\imm_0$.
For immersed Lagrangians, not every exact deformation is a Hamiltonian deformation.
This differs from the embedded case where the two concepts coincide.

Now consider the energy and indices of branch jumps.
Suppose $\dep{\imm}{t}{t}$ is a family of exact immersed Lagrangians, and suppose all the immersions have only transverse double point self-intersections.
Let $\brjmps_t=\set{(p,q)\in\lag\times\lag}{\imm_t(p)=\imm_t(q),\ p\neq q}$.
There are canonical bijections $\brjmps_t\cong\brjmps_0$ for all $t$ and we can view branch jumps $(p_t,q_t)\in\brjmps_t$ as varying continuously with $t$.
Since the index of a branch jump takes values in $\zz$, it follows that $\ind (p_t,q_t)=\ind (p_0,q_0)$ for all $t$.
The energy however can vary.
By definition and equation \eqref{eq:35},
\begin{eqnarray}
  \label{eq:36}
  E(p_t,q_t)&=&-f_t(p_t)+f_t(q_t)\\&=&-f_0(p_t)+f_0(q_t)+\nonumber\\
  &&\int_0^t\left(-h_s(p_t)+h_s(q_t)-\oneform\left(\pd{\imm_s}{s}(p_t)\right)+\oneform\left(\pd{\imm_s}{s}(q_t)\right)\right)ds.\nonumber
\end{eqnarray}
Note that if the deformation is Hamiltonian, then $p_t=p_0,q_t=q_0$ for all $t$, and $h_s$ and $\pd{\imm_s}{s}$ are the restrictions of a function and a vector field defined globally on $\mnfld$.
Hence $E(p_t,q_t)=E(p_0,q_0)$ for all $t$.

A consequence of this behavior is that the positivity condition may hold for $t=0$ but may fail for $t$ far enough away from $0$.
As observed in \cite{MR2785840} Section 13.4, this can be interpreted as a type of wall-crossing behavior.
In the language of \cite{fooo} and \cite{MR2785840}, Lagrangians satisfying the positivity condition are automatically unobstructed.
If such a Lagrangian is deformed to the point where the positivity condition fails, it may become obstructed.
Thus Lagrangian immersions separated by a ``wall'' can have very different Floer theoretic properties.
It would be interesting to find an explicit example of this.

The proof of Theorem \ref{thm:9} does not work for general exact deformations because Gromov compactness (from \cite{MR1890078}) does not hold for moving Lagrangian boundary conditions that contain an instant where the immersed Lagrangian has non-transverse self-intersection.
If we rule this case out, then we can prove invariance under exact deformation.
\begin{theorem}
  \label{thm:12}
  Let $\dep{\imm}{t}{a\leq t\leq b}$ be a family of exact immersions.
  Suppose that each $\imm_t$ has only transverse double points.
  Moreover, suppose that each $\imm_t$ satisfies Condition \ref{ass:1}.
  Then all the different $\imm_t$ are quasi-isomorphic objects in the Fukaya category.
\end{theorem}
\begin{proof}
  We will show that given any $t_0\in[a,b]$, there exists an open neighborhood $U$ of $t_0$ in $[a,b]$ such that if $t\in U$ then $\imm_t$ and $\imm_{t_0}$ are quasi-isomorphic objects.
  The theorem will then follow immediately.

  The idea of the proof is to exhibit a bijection between the moduli spaces for various $\imm_t$ used to define $\m{1}$ and $\m{2}$.
  We appeal to regularity and Gromov compactness to prove this bijection.
  Once this is done, we note that there are obvious bijections between Hamiltonian chords with endpoints on the $\imm_t$, and hence given a unit $e_{t_0,t_0}\in\cf(\imm_{t_0},\imm_{t_0})$ there are corresponding $e_{t,t'}\in\cf(\imm_t,\imm_{t'})$ for each $t,t'$ near $t_0$.
  The fact that the moduli spaces used to define $\m{1}$ and $\m{2}$ are bijective for the various $\imm_t$ then implies that $e_{t,t'}$ is a quasi-isomorphism.

  Without loss of generality assume $t_0=0$.
  We need to find $\delta>0$ such that if $0\leq t<\delta$ then $\imm_0$ and $\imm_t$ are quasi-isomorphic.
  Choose regular Floer data $\h,\J$ for $\imm_0$, and let $\phi=\hamdiff{\h}{1}$ be the time-1 flow of $\h$.
  Let $\chords{\h}(\imm_t,\imm_{t'})$ denote the set of time-1 chords of $\h$ that start on $\imm_t$ and stop on $\imm_{t'}$.
  Choose $\delta>0$ such that continuation gives a bijection between $\chords{\h}(\imm_0,\imm_0)$ and $\chords{\h}(\imm_t,\imm_{t'})$ for any $0\leq t,t' <\delta$.
  We may assume that the endpoints of the chords always miss the self-intersection points of $\imm_t$ and $\imm_{t'}$, since this is true for $t=t'=0$.

  We denote by $\chord_i$ chords in $\chords{\h}(\imm_0,\imm_0)$ and $\chord_i^{t,t'}$ chords in $\chords{\h}(\imm_t,\imm_{t'})$.
  Generally, any undecorated notation will be with respect to $\imm_0$, and notation decorated with $t,t',\ldots$ will be with respect to $\imm_t,\imm_{t'},\ldots$.
  $A_\infty$ operations depend upon choices of Floer data and perturbation data, which will be specified later.

  First, choose some perturbation data $\pdK{d+1},\pdJ{d+1}$ for $\imm_0$ (for all $d$), so the $A_\infty$ algebra for $\imm_0$ is defined.
  Consider the moduli spaces used to define $\m{1}$ and $\m{2}$, which are
  \begin{gather*}
    \holpolys{\chord_0,\chord_1;\h,\J},\quad \ind\chord_0=\ind\chord_1+1,\\
    \holpolys{\chord_0,\chord_1,\chord_2;\pdK{3},\pdJ{3}},\quad \ind\chord_0=\ind\chord_1+\ind\chord_2,
  \end{gather*}
  respectively.
  Since these moduli spaces are regular, and compact and zero dimensional, and there are only finitely many of them, any holomorphic curve in one of these moduli spaces will persist under slight perturbation of its boundary conditions.
  In particular, by shrinking $\delta$ if necessary, there are injective maps
  \begin{gather}
    \label{eq:38}
    \holpolys{\chord_0,\chord_1;\h,\J}\to\holpolys{\chord_0^{t,t'},\chord_1^{t,t'};\h,\J},\\
    \holpolys{\chord_0,\chord_1,\chord_2;\pdK{3},\pdJ{3}}\to\holpolys{\chord_0^{t,t''},\chord_1^{t,t'},\chord_2^{t',t''};\pdK{3},\pdJ{3}}\nonumber
  \end{gather}
  for 0 dimensional moduli spaces and $0\leq t,t',t''<\delta$.
  Note that we do not yet change the perturbation data for the moduli spaces involving $\imm_t,\imm_{t'},\imm_{t''}$.
  Elements from all the moduli spaces have an a priori energy bound by equations \eqref{eq:15}, \eqref{eq:5}, \eqref{eq:35} and \eqref{eq:36}.
  We claim that Gromov compactness implies that, for small enough $\delta$, these maps are surjective as well.
  We explain this just for the first map.
  If not, there exists sequences of numbers $t_n,t'_n\to 0$ and a sequence of holomorphic strips $[\markedstrip{u_n}]\in\conntraj{}{\chord_0^{t_n},\chord_1^{t'_n};\h,\J}$ that are not in the image of the first map.
  Applying Gromov compactness (the compactness theorem proved in \cite{MR1890078} works in this setting since $\imm_t\to\imm_0$ and $\imm_0$ has transverse self-intersections), we get a stable strip with boundary on $\imm_0$.
  Since its index is $0$, this stable strip consists of a single strip, that is, it is an element $[\markedstrip{u}]\in\conntraj{}{\chord_0,\chord_1;\h,\J}$.
  Since $\conntraj{}{\chord_0,\chord_1;\h,\J}$ is a finite set, the implicit function theorem implies that $[\markedstrip{u}]$ maps to $[\markedstrip{u_n}]$ under the first map for large enough $n$, a contradiction. This also implies $\holpolys{\chord_0^{t,t'},\chord_1^{t,t'};\h,\J}$ and $\holpolys{\chord_0^{t,t''},\chord_1^{t,t'},\chord_2^{t',t''};\pdK{3},\pdJ{3}}$ are regular.

  The above argument does not say anything about the spaces the spaces $\holpolys{\chord_0^{t,t'},\chord_1^{t,t'};\brjmptypes;\h,\J}$ and $\holpolys{\chord_0^{t,t''},\chord_1^{t,t'},\chord_2^{t',t''};\brjmptypes;\pdK{3},\pdJ{3}}$ for $\brjmptypes\neq\emptyset$.
  We need these to be regular too, which can be achieved by a slight perturbation of the the perturbation data.
  So we now choose perturbation data and Floer data for each $\imm_t$ with $0\leq t<\delta$, and also each tuple of such Lagrangians, so that we have a well-defined Fukaya category which includes all $\imm_t$ as objects.
  We may assume (by first fixing a set of chords and fixing $t,t',t''$ and looking at small deformations of Floer data and perturbation data) that Floer data and perturbation data are chosen sufficiently close to $\h,\J,\pdK{3},\pdJ{3}$ so that there are canonical bijections
  \begin{gather}
    \label{eq:39}
    \holpolys{\chord_0^{t,t'},\chord_1^{t,t'};\h,\J}\cong\holpolys{\chord_0^{t,t'},\chord_1^{t,t'};\h_{t,t'},\J_{t,t'}},\\
    \holpolys{\chord_0^{t,t''},\chord_1^{t,t'},\chord_2^{t',t''};\pdK{3},\pdJ{3}}\cong\holpolys{\chord_0^{t,t''},\chord_1^{t,t'},\chord_2^{t',t''};\pdK{3}_{t,t',t''},\pdJ{3}_{t,t',t''}}\nonumber
  \end{gather}
  of $0$ dimensional moduli spaces.
  (Again, part of the reason this can be done is that these moduli spaces contain only finitely many elements and there are only finitely many such moduli spaces.)

  Since all the different sets $\chords{\h}(\imm_t,\imm_{t'})$ are bijective to each other, the same is true for all $\chords{\h_{t,t'}}(\imm_t,\imm_{t'})$ after shrinking $\delta$ if necessary.
  Thus all the vector spaces $\cf(\imm_t,\imm_{t'})$ are isomorphic to each other in a natural way, and furthermore the bijections \eqref{eq:39} imply that the $A_\infty$ operations $\m{1}^t$ and $\m{2}^{t,t'}$ are all identified under these isomorphisms.
  In particular, if $e_{0,0}\in\cf(\imm_0,\imm_0)$ is a (cohomological) unit, then the corresponding $e_{t,t'}\in\cf(\imm_t,\imm_{t'})$ are quasi-isomorphisms.

\end{proof}

\begin{remark}
  \label{rmk:3}
  Alternatively, a proof can be given using moving Lagrangian boundary conditions, which we now sketch.
  First, construct an element $e\in\cf(\imm_a,\imm_b)$ (which will turn out to be a quasi-isomorphism) in the same way as was done for Hamiltonian deformations prior to Remark \ref{rmk:2}.
  Namely, $e$ is a count of rigid inhomogeneous holomorphic discs with one fixed incoming Type I marked point and moving Lagrangian boundary conditions. 
  Take perturbation data $K$ on the disc which, for simplicity, is $0$ near the boundary of the disc.
  For Hamiltonian moving Lagrangian boundary conditions, we took $K$ to satisfy $d(K(\xi)|\imlag_z)=\omega(\del_\xi\imlag_z,\cdot)$ for $\xi\in T_z\bdy\disc$ (see formula (8.21) in \cite{seidel-fcpclt}); this doesn't make sense in the non-Hamiltonian case.
  The slight difficulty this causes is that the energy equation \eqref{eq:15}, which is $E(u)=\action(\gamma_0)-\int_\disc R_K(u)$ in this case, no longer holds for a curve $u$ converging to the chord $\gamma_0$ on $\imm_a,\imm_b$.
  To see this, start with
  \begin{displaymath}
    E(u)=\frac{1}{2}\int_\disc |du-X_K|^2d\mathrm{vol}_\disc=\int_\disc \biggl(u^*\sympl+d(K(u))-R_K(u)\biggr).
  \end{displaymath}
  We want to use Stokes' theorem on the right-hand side.
  To this end, parameterize the portion of $\bdy\disc$ with moving Lagrangian boundary conditions as $[a,b]$, so the Lagrangian for $t\in[a,b]\subset\bdy\disc$ is $\imm_t$.
  Let $\Phi:[a,b]\times\lag\to\mnfld$ be the map $\Phi(t,x)=\imm_t(x)$, and $f:[a,b]\times\lag\to\mnfld$ be $f(t,x)=f_t(x)$, where $df_t=\imm_t^*\oneform$.
  Then $\Phi^*\oneform=df-hdt$, where $h:[a,b]\times\lag\to\rr$ is the function such that $d(h(t,\cdot))=\beta_t$ (see prior to equation \eqref{eq:35}).
  Indeed, using equation \eqref{eq:35}, we have
  \begin{eqnarray*}
    (df)(t,\cdot) & = & df_t+\dot f_tdt=\imm_t^*\oneform+\dot f_tdt\\
    &=& \Phi^*\oneform -\langle \del_t\imm_t,\oneform\rangle dt+(h_t+\langle \del_t\imm_t,\oneform\rangle)dt\\
    &=& \Phi^*\oneform+h_tdt.
  \end{eqnarray*}
  Using Stokes' theorem, $\sympl=d\oneform$, and the definition of action \eqref{eq:5}, we then get
  \begin{equation}
    \label{eq:26}
    E(u)=\action(\gamma_0)-\int_\disc R_K(u)-\int_{[a,b]}h(t,\ell(t))dt
  \end{equation}
  where $\ell$ is the boundary lift to $\lag$ of $u$.
  This still gives an a priori energy bound.

  To show that $e\in\cf(\imm_a,\imm_b)$ is well-defined, we need to show that Gromov compactness holds for a sequence of elements from this moduli space.
  The energy bound was provided above.
  It remains to consider disc bubbling.
  Using a graph construction (see Section 8 of \cite{MR2954391}), we can view $u$ as an honest holomorphic section (with respect to an appropriate $J$) of the bundle $\mnfld\times\disc\to\disc$.
  The boundary conditions becomes the fixed totally real immersion $\lag\times\bdy \disc\to \mnfld\times\disc$, $(x,z)\mapsto (\imm_z(x),z)$ with $\imm_z$ the immersion for $z\in\bdy \disc$.
  The immersion now has a clean self-intersection instead of a transverse self-intersection (here is where we use $\imm_t$ has transverse self-intersection for all $t$).
  The Gromov compactness proved in \cite{MR1890078} still holds in this setting.
  The result is that in the limit disc bubbles may appear in fibers of the section, i.e. lying on specific $\imm_z$, and must attach to the main disc component at a branch jump point.
  This is the expected behavior, and it matches what happens in the Hamiltonian deformation case.
  The positivity condition and regularity can then be used to rule out unwanted disc bubbling.
  The next section, where we consider cleanly immersed Lagrangians, contains the necessary analytic results for this.
  
  The rest of the proof proceeds in the same way as the Hamiltonian deformation case outlined before Remark \ref{rmk:2}.
\end{remark}

\section{Clean self-intersection}
\label{sec:clean-self-inters-1}
Up to this point we have assumed $\imm$ has transverse double points.
In this section we explain why the theory works equally well if the double points are clean instead of transverse, in the following sense.

\begin{definition}
  \label{dfn:23}
  An immersion $\imm:\lag\to\imlag\subset\mnfld$ has \textit{clean self-intersection}, or is a \textit{clean immersion}, if it has the following properties:
  \begin{itemize}
  \item The preimage of any point in $\imlag$ consists of 1 or 2 points.
  \item The set $\imbrjmps=\set{x\in\imlag}{|\imm^{-1}(x)|=2}$ is a closed submanifold of $\mnfld$. 
  \item For any $p,q\in\lag$ with $\imm(p)=\imm(q)=:x$ and $p\neq q$, we have 
      $\imm_*T_p\lag\cap \imm_*T_q \lag= T_x\imbrjmps$.
  \end{itemize}
\end{definition}

These conditions imply that the set $\imm^{-1}(\imbrjmps)$ is a closed submanifold of $\lag$, and if $p\in\imm^{-1}(\imbrjmps)$ then $\imm_*T_p\imm^{-1}(\imbrjmps)=T_{\imm(p)}\imbrjmps$.
$\imbrjmps$ may have several different components of different dimensions.
The preimage of a component may or may not be connected.
For a discussion of a more general definition of clean immersion, see Section \ref{sec:more-than-double}.

\begin{remark}
\label{rmk:11}
This definition allows the case that $\imm:\lag\to\imlag$ is a 2-1 covering space.
However, this case needs to be handled slightly differently than the other cases.
In order to streamline the discussion we thus assume that $\imm$ is not a covering; in other words, $\imbrjmps$ is not all of $\imlag$.
We will return to this exceptional case later on.
\end{remark}

As before, let $\brjmps=\set{(p,q)\in\lag\times\lag}{\imm(p)=\imm(q),\ p\neq q}$
be the set of all branch jump types, and we will use $\brjmptypes$ to denote functions
 $\brjmptypes:\sett{1,\ldots,m}\to \brjmps$.
We continue to assume that $\lag$ is exact with primitive $\lagprim:\lag\to\rr$ and graded with grading $\laggrad:\lag\to\rr$.

The notion of Lagrangian angle (see Section \ref{sec:maslov-index}) can be extended to the case of clean intersection as follows.
Suppose given two Lagrangian planes $\Lambda_0, \Lambda_1$ in a symplectic vector space $(V,\omega)$, and let $\Lambda=\Lambda_0\cap\Lambda_1$.
Choose a path of Lagrangian planes $\Lambda_t$ from $\Lambda_0$ to $\Lambda_1$ such that
\begin{itemize}
\item $\Lambda\subset\Lambda_t\subset \Lambda_0+\Lambda_1$ for all $t$, and
\item the image $\overline{\Lambda_t}$ of $\Lambda_t$ inside the symplectic vector space $(\Lambda_0+\Lambda_1)/\Lambda$ is the positive definite path from $\overline{\Lambda_0}$ to $\overline{\Lambda_1}$.
\end{itemize}
Let $\theta_t$ be a path of real numbers such that $e^{2\pi i \theta_t}=\phase(\Lambda_t)$.
Then define the angle between $\Lambda_0$ and $\Lambda_1$ by the formula $2\cdot\Angle(\Lambda_0,\Lambda_1):=\theta_1-\theta_0$.

With this definition in hand we extend the notion of the index of a branch jump to the clean intersection case by using the same formula as the transverse case.
The \textit{index} of $(p,q)\in\brjmps$ is defined to be 
\begin{displaymath}
  \ind(p,q)=\dim\lag+\laggrad(q)-\laggrad(p)-2\cdot\Angle(\imm_*T_p\lag,\imm_*T_q\lag ).
\end{displaymath}

Now assume that positivity Condition \ref{ass:1} holds; that is if $(p,q)\in\brjmps$ and $E(p,q)=-\lagprim(p)+\lagprim(q)>0$, then $\ind(p,q)\geq 3$.
Let us analyze this condition a little bit.
Note that both $E$ and $\ind$ are locally constant as functions $\brjmps\to\rr$.
Let $\overline C\subset \overline R$ be a component such that the preimage consists of two components $C_1,C_2$.
Then $E(p,q)=-\lagprim(p)+\lagprim(q)$ is independent of $(p,q)\in \brjmps$ with $p\in C_1$ and $q\in C_2$.
Hence define $E(C_1,C_2):=-\lagprim(p)+\lagprim(q)$.
Likewise, $\ind(p,q)$ is independent of $(p,q)$ and we can define $\ind (C_1,C_2):=\ind (p,q)$.

If the preimage of $\overline C$ consists of a single component $C$ (and hence $C\to\overline C$ is a non-trivial double cover) then more can be said.
Since any $(p,q)\in\brjmps$ with $p,q\in C$ can be connected to $(q,p)$ in $\brjmps$, we must have $E(p,q)=E(q,p)$.
Since $E(p,q)=-E(q,p)$ from the formula, we have $E(p,q)=0$.
Thus we define $E(C,C)=0$.
Similar considerations apply to the index:
Since $\ind(p,q)+\ind(q,p)=\dim\lag+\dim C$,
we define $\ind(C,C):=\ind(p,q)=\ind(q,p)=\frac{1}{2}\left(\dim\lag+\dim C\right)$.
$C$ vacuously satisfies the positivity assumption because $E(C,C)=0$.
Thus the positivity condition can be rephrased as
\begin{condition}
  \label{sec:clean-self-inters}
  If $\overline C$ is a component of $\overline R$ such that the preimage consists of two distinct components $C_1$ and $C_2$ and $E(C_1,C_2)>0$ then $\ind (C_1,C_2)\geq 3$.
\end{condition}

\subsection{Marked strips, polygons and discs}
We continue to call Floer data $\h,\J$ that satisfies Definition \ref{dfn:1} admissible.
Recall from Remark \ref{rmk:11} that for now we do not allow $\imbrjmps$ to be all of $\imlag$, hence admissible Floer data exists.
Assume, in addition to Floer data, that we are given strip-like ends and perturbation data $\pdK{d+1},\pdJ{d+1}$ for the families $\univermod{d+1,m}\to\domainmod{d+1,m}$ for all $d+1\geq3,m\geq0$, as in Section \ref{sec:moduli-spac-doma}.

Given any fixed $\brjmptypes:\sett{1,\ldots,m}\to\brjmps$ define analogs of the Banach manifolds $\strips{1,p}{\strip,\brjmptypes}$, $\polys{1,p}{\univermod{d+1,m},\brjmptypes}$, and $\discs{1,p}{\disc,\brjmptypes}$
that we defined before, and corresponding Banach bundles over these spaces along with Cauchy-Riemann sections.
The main difference is that now weighted Sobolev spaces need to be used in order to get a good Fredholm theory due to the non-transversality of the self-intersections.
To do this, pick some small $\delta>0$.
A $\delta$-weighted $L^p$ norm is a norm on a strip (or half infinite strip) that is of the form
\begin{displaymath}
  ||\xi||_{L^{p;\delta}}=\int |\xi|^pe^{\delta|s|}dsdt,
\end{displaymath}
where $s+it$ are strip-like coordinates.
In the setup of our Banach manifolds we now require the regularity to be $W^{1,p;\delta}$.
For instance, Definition \ref{dfn:3} is modified by requiring $\xi$ to be in $W^{1,p;\delta}$ over each strip-like end.
For $\delta>0$ small enough, the inclusion of the weights makes no difference for the index or regularity theory of the Type I points (note that we have already fixed the Floer data $\h,\J$).
Theorem \ref{thm: asymptotic behavior} implies that bounded $L^2$ energy of holomorphic curves is enough to get $W^{1,p;\delta}$ regularity near Type II marked points; it is standard that the same statement holds near Type I marked points.
One minor technical point that needs to be addressed is that we need to choose the reference metric $g_{fix}$ on $\mnfld$ in such a way that $\imlag$ is totally geodesic.
This can be done by the following lemma.
\begin{lemma}
  Let $\imlag$ be cleanly immersed.
  Then there exists a metric $g$ on $\mnfld$ so that $\imlag$ is totally geodesic.
\end{lemma}
\begin{proof}
  As the proof will make evident, we may assume $\imbrjmps$ consists of a single component.
  By the definition of cleanly immersed, a tubular neighborhood of $\imbrjmps$ inside $\mnfld$ can be identified with a neighborhood of the zero-section in the normal bundle $N=N_{\imbrjmps,\mnfld}$ in such a way that $\imlag$ maps to (open neighborhoods in) subbundles $E_1$ and $E_2$ of $N$.
  The intersection of $E_1$ and $E_2$ is the zero-section and the corank of $E_1\oplus E_2$ equals the dimension of $\imbrjmps$.
  Choose another subbundle $E'$ so that $E_1\oplus E_2\oplus E'=N$.
  Choose a metric on $N$ of the form $h_1\oplus h_2\oplus h'$, and choose a metric on the manifold $\imbrjmps$.
  These metrics induce a metric $g$ on the total space of $N$ (i.e., for each $x\in N$, $g$ gives a map $g_x: T_xN\otimes T_xN\to \rr$).
  Then the following diffeomorphisms $N\to N$ are all isometries with respect to $g$:
  \begin{displaymath}
    (e_1,e_2,e')\mapsto (-e_1,-e_2,-e'),\quad    (e_1,e_2,e')\mapsto (e_1,-e_2,-e'),\quad    (e_1,e_2,e')\mapsto (-e_1,e_2,-e').
  \end{displaymath}
  The fixed point set of the first map is $\imbrjmps$, the fixed point set of the second is $E_1$, and the fixed point set of the third is $E_2$.
  Thus $\imbrjmps$, $E_1$ and $E_2$ are totally geodesic with respect to the metric $g$.
  Now pull the metric $g$ back to the tubular neighborhood of $\imbrjmps$ in $\mnfld$.
  This defines $g$ near the singular points of $\imlag$.
  Then the arguments in the proof of Lemma \ref{lemma:14} can be applied to show that an extension of $g$ to all of $\mnfld$ can be found that makes $\imlag$ totally geodesic.
\end{proof}

The addition of the weights makes maps in the spaces converge to the values of $\imm\circ\brjmptypes$ on either side of the Type II points.
We want to allow these values to vary; that is, we want to allow $\brjmptypes$ to vary continuously.
To describe this we first set up some notation.
Let $\brjmpcomps$ be the set of all $(C_1,C_2)$ such that for some $(p,q)\in\brjmps$, $C_1$ is the component of $\imm^{-1}(\imbrjmps)$ containing $p$ and $C_2$ is the component containing $q$.
Note that any $\brjmptypes:\sett{1,\ldots,m}\to\brjmps$ induces a map $\brjmpcomptypes:\sett{1,\ldots,m}\to\brjmpcomps$.
Given $\brjmpcomptypes$, define
\begin{displaymath}
  \strips{1,p;\delta}{\strip,\brjmpcomptypes}=\coprod_{\brjmptypes}\strips{1,p;\delta}{\strip,\brjmptypes}\times\sett{\brjmptypes}
\end{displaymath}
where the union is over all $\brjmptypes$ that induce the given $\brjmpcomptypes$.
Similarly we can define $\polys{1,p;\delta}{\univermod{d+1,m},\brjmpcomptypes}$ and $\discs{1,p;\delta}{\disc,\brjmpcomptypes}$.
To incorporate the varying of $\brjmptypes$ into the Banach space setup we enlarge the local model for the coordinate charts by adding on an open subset of $\rr^{d_i}$ for each $1\leq i\leq m$, where $d_i=\dim \brjmpcomptypes(i)$.
These spaces also have Banach bundles and Cauchy-Riemann sections.

We now define moduli spaces $\holpolys{\chord_0,\ldots,\chord_d;\brjmpcomptypes}$
as in Definition \ref{dfn:21}. 
These moduli spaces coincide with the zero-sets of Cauchy-Riemann sections
$$\dbar:\polys{1,p;\delta}{\univermod{d+1},\brjmpcomptypes}\to\polysbundle{0,p;\delta}{\univermod{d+1,m},\brjmpcomptypes}.$$
Similarly, given any $J$ we can define the moduli space of discs $\holdiscs{}{\brjmpcomptypes;J}$.
The linearization of the Cauchy-Riemann operator at an element $(\markedpoly{u},r,\ptsii,\brjmptypes)$ of $\holpolys{\chord_0,\ldots,\chord_d;\brjmpcomptypes}$ has the form
\begin{displaymath}
  D\dbar: \mapmodelwt{\markedstrip{u}^*T\mnfld}\times T_{(r,\ptsii)}\domainmod{d+1,m}\times T_{\brjmptypes}\imbrjmps \to \bunmodel{\zeroone{}\otimes_\cc\markedstrip{u}^*T\mnfld},
\end{displaymath}
where $T_{\brjmptypes}\imbrjmps := \bigoplus_{i=1}^mT_{\imm(\brjmptypes(i))}\imbrjmps$.
In case $d+1=2$ (so the domain is a strip), replace the $T_{(r,\ptsii)}\domainmod{d+1,m}$ term with $T_{\ptsii}\config{m}{\strip}$ (see equation \eqref{eq:42}).
This operator is Fredholm and has index
\begin{equation}
  \label{eq:34}
  \ind D\dbar = \ind\chord_0-\sum_{i=1}^d\ind\chord_i-\sum_{i=1}^m\ind \brjmpcomptypes(i)+m+\left\{
    \begin{array}{ll}
      d-2 & d\geq 2\\
      0 & d=1.
    \end{array}\right.
\end{equation}
The same methods as before prove the following proposition.
\begin{proposition}
  \label{prop:14}
  Generic perturbation data is regular.
  For regular perturbation data, all the spaces $\holpolys{\chord_0,\ldots,\chord_d;\brjmpcomptypes}$ are smooth manifolds whose dimensions are determined by \eqref{eq:34} (in case $d=1$ we subtract 1 from the index to get the dimension of the moduli space because of the $\rr$ action).

  Furthermore, the moduli spaces $\holpolys{\chord_0,\ldots,\chord_d}$ have compactifications similar to those described in Proposition \ref{prop:13}.
\end{proposition}

\subsection{Floer cohomology and $A_\infty$ structure}
Condition \ref{sec:clean-self-inters} can be phrased in the following way: If $(C_1,C_2)\in\brjmpcomps$ and $E(C_1,C_2)>0$ then $\ind(C_1,C_2)\geq 3$.
Thus a clean immersion satisfying this assumption formally resembles the transverse self-intersection case under Condition \ref{ass:1}, with $\brjmptypes$ replaced with $\brjmpcomptypes$ and $\brjmps$ replaced with $\brjmpcomps$.
In particular, notice the similarity between equation \eqref{eq:34} and equations \eqref{eq:13}, \eqref{eq:31}.

Now define the $A_\infty$ structure in the same way as before: pick generic Floer data $\h,\J$ and perturbation data, take $\cf(\imm)$ to be generated by the chords on $\imm$, and then define the $\m{k}$ by counting inhomogeneous holomorphic polygons.
It follows that the proofs used to establish existence of Floer cohomology and $A_\infty$ structure for the transverse self-intersection case carry over immediately to the clean intersection case.
\begin{theorem}
  \label{thm:2}
  Let $\imm:\lag\to\imlag$ be an immersed Lagrangian with clean self-intersection in the sense of Definition \ref{dfn:23}.
  Assume that $\imbrjmps\neq\imlag$ and the immersion satisfies Condition \ref{sec:clean-self-inters}.
  Then, for generic Floer data $(\h,\J)$ and generic perturbation data, the Floer cohomology $\hf(\imm)$ and the $A_\infty$ algebra $(\cf(\imm),\sett{\m{k}}_{k\geq 1})$ can be defined as in Sections \ref{sec:floer-cohomology} and \ref{sec:a_infty-structure}.
  Moreover, the Floer cohomology does not depend on the choice of Floer data or perturbation data, and the same is true of the $A_\infty$ algebra up to homotopy equivalence.
\end{theorem}
The theorem also holds when $\imbrjmps=\imlag$.
This will be proved in the next section.

\subsection{Case when $\imbrjmps=\imlag$: Equivalence with local systems}
\label{sec:case-when-imbrjmps}
We now consider the case when $\imbrjmps=\imlag$; in other words, $\imlag$ is an embedded Lagrangian submanifold of $\mnfld$ and $\imm:\lag\to\imlag$ is a covering space.
This in fact implies that $\imlag$ is exact also, so this case turns out to be easier than the non-covering space case.
In Definition \ref{dfn:23} we assumed that immersed points are at most double points.
This however is not necessary, so assume now only that $\lag\to\imlag$ is a $d$-fold covering space and $\lag$ is connected.
Actually since $\lag$ is compact it is automatically a finite cover.
Given Hamiltonian $\h$, recall that $\chords{\h}$ is the set of (time 1) Hamiltonian chords that stop and start on $\imlag$.
In Definition \ref{dfn:1} we required that none of the chords started or stopped on $\imbrjmps$; this can never be true now.
To account for this, we change the definition of $\chords{\h}$.
\begin{definition}
  \label{dfn:22}
  Let $\chords{\h}$ denote the set of triples $\tilde\chord=(\chord,p_0,p_1)$ such that
  \begin{itemize}
  \item $\chord:[0,1]\to\mnfld$ is a Hamiltonian chord that starts and stops on $\imlag$, and
  \item $p_0,p_1\in \lag$ are such that $\imm(p_0)=\chord(0),\imm(p_1)=\chord(1)$.
  \end{itemize}
\end{definition}

We now say that a Hamiltonian is admissible as long as $\hamdiff{\h}{1}(\imlag)$ is transverse to $\imlag$.
If $\h$ is admissible then $\chords{\h}$ is a finite set with $d^2$ elements for each point of $\hamdiff{\h}{1}(\imlag)\cap\imlag$.
In this case the Floer cochain complex is defined to be
\begin{equation}
  \label{eq:40}
  \cf(\imm;\h,\J)=\bigoplus_{\lchord\in\chords{\h}}\zz_2\cdot \lchord.
\end{equation}
The grading of $\imm$ defines a grading of chords $\lchord\in\chords{\h}$ in much the same way as before; the only difference is that in Definition \ref{dfn:7}, $T_{\chord(0)}\imlag$ needs to be replaced with $\imm_*T_{p_0}\lag$ and $T_{\chord(1)}\imlag$ with $\imm_*T_{p_1}\lag$.
(Actually, Lemma \ref{lemma:2} implies that $\imlag$ can be graded and the grading on $\imm$ is the pull back of a grading on $\imlag$, so the index of a chord $\lchord$ is in fact determined by $T_{\chord(i)}\imlag$, $i=0,1$.)
This defines a grading on $\cf(\imm;\h,\J)$.

We define the moduli spaces of strips $\conntraj{}{\lchord_-,\lchord_+;\brjmptypes}$, polygons $\holpolys{\lchord_0,\ldots,\lchord_k;\brjmptypes}$ and discs $\holdiscs{}{\brjmptypes,J}$ much in the same way as before.
The difference is that now the boundary lift $\ell$ of a holomorphic curve must agree on either side of a Type I marked point with the lifts $p_0,p_1$ of the endpoints $\chord(0),\chord(1)$ of the chord $\lchord=(\chord,p_0,p_1)$ corresponding to the marked point.
The order in which they agree should depend on whether the point is incoming or outgoing.
The same convention as for Type II marked points applies:
If $\ell_-\in\lag$ denotes the limiting value of $\ell$ before the Type I point (with respect to counterclockwise orientation), and $\ell_+$ the value after, then if the marked point is outgoing $\ell_-=p_0,\ell_+=p_1$.
For incoming, $\ell_-=p_1,\ell_+=p_0$.

The regularity theory of Section \ref{sec:regularity} carries over immediately.
Thus Floer data $\h,\J$ and perturbation data that make all the uncompactified moduli spaces of strips and polygons regular exists.
The compactness theory of Section \ref{sec:grom-comp-moduli} holds as well with the obvious modification:
the matching condition at a Type I nodal point of a broken curve needs to include matching of the chord \textit{plus} matching of the lifts of the endpoints of the chord.
This is automatically kept track of with the new notation since a chord $\chord$ is replaced with $\lchord=(\chord,p_0,p_1)$.

We would now like to conclude that the $A_\infty$ algebra $(\cf(\imm),\sett{\m{k}})$ for $\imm$ is well-defined and, up to equivalence, is independent of the various choices used to construct it.
However, there is one point that still needs addressing.
The condition $[\hamdiff{\h}{1}(\imlag)\cup(\hamdiff{\h}{1})^{-1}(\imlag)]\cap\imbrjmps=\emptyset$ in Definition \ref{dfn:1} is used to rule out the existence of constant strips with disc bubbles attached.
In general, such strips will not be regular; see Proposition \ref{prop:9} and the preceding discussion.
Since the covering space is finite, the map $\imm^*:\mathrm{H}^1(\imlag,\rr)\to\mathrm{H}^1(\lag,\rr)$ is injective, and thus $\imlag$ is exact.
Any non-constant element of $\holdiscs{}{\imm;\brjmptypes,J}$ corresponds to a non-constant holomorphic disc with boundary on $\imlag$; since this latter type of disc does not exist $\holdiscs{}{\imm;\brjmptypes,J}$ consists only of constant discs for any $J$.
In particular, if $\brjmptypes:\sett{1}\to\brjmps$ (i.e., there is only one marked point) then $\holdiscs{}{\imm;\brjmptypes,J}=\emptyset$ because there are no constant discs with only one marked point.
Thus the Gromov compactness of the moduli spaces $\holpolys{\lchord_0,\ldots,\lchord_k}$ takes a particularly strong form; namely, elements of the compactified moduli space never contain disc bubble components.
It follows that constant strips never enter into the compactified moduli spaces needed to define the $A_\infty$ structure, and hence they do not cause a problem.
Thus our previous arguments go through to prove the following theorem.
\begin{theorem}
  \label{thm:5}
  Suppose $\imm:\lag\to\imlag$ is a covering space and $\lag$ is compact.
  Then the $A_\infty$ algebra $(\cf(\imm),\sett{\m{k}})$ is well-defined and independent of the choices involved, up to homotopy equivalence.
  Here, $\cf(\imm)$ is defined as in \eqref{eq:40} and $\m{k}$ is defined using the moduli spaces $\holpolys{\lchord_0,\ldots,\lchord_k}$.
\end{theorem}

Now we want to relate the above construction to the construction of Floer theory with coefficients in a local system.
First, we sketch the general theory as it applies to $\imlag$.
Since $\imlag$ is an exact embedded Lagrangian, the $A_\infty$ algebra for $\imlag$ is well-defined in the usual way (note that by this $A_\infty$ algebra we do not mean $\cf(\imm)$).
Let $\locsys\to\imlag$ be a $\zz_2$ local system of rank $r$ over $\imlag$.
Explicitly, pick a basepoint $x_0\in\imlag$ and view $\locsys$ as a $\zz_2$ vector bundle with fiber isomorphic to $\locsys_{x_0}=V=(\zz_2)^{\oplus r}$ (with $r$ finite).
Let $\hol(c):V\to V$ denote parallel translation over the loop $c$ in $\imlag$ with basepoint $x_0$.
Then the map $\rho=\rho_\locsys:\pi_1(\imlag)=\pi_1(\imlag,x_0)\to \Aut_{\zz_2}(V)$ defined by $\rho(c)v=(\hol(c))^{-1}v$ is a representation of $\pi_1(\imlag)$.
Conversely, given a representation $\rho$, we can construct a bundle $\locsys=\locsys_\rho\to\imlag$ by setting $\locsys=\widetilde{\imlag}\times_\rho V=\widetilde{\imlag}\times V/\sim$, where $(g \cdot  x, \rho(g)v)\sim (x,v)$.
Here, $\widetilde \imlag$ is the universal cover of $\imlag$ and $g \cdot x$ denotes the left action of $\pi_1(\imlag)$ on $\widetilde\imlag$.
(Our convention is that multiplication in $\pi_1(\imlag)$ is concatenation from left to right of loops, so $gh$ means first $g$ and then $h$.
If $\widetilde\imlag$ is thought of as homotopy classes of paths in $\imlag$ with basepoint $x_0$, then $g\cdot x$ is the path which is first $g$ and then $x$.)
These two constructions are inverse to each other, in the sense that given $\rho$, we have $\rho_{\locsys_\rho}\cong \rho$; and given $\locsys$, we have $\locsys_{\rho_\locsys}\cong\locsys$.

Next, suppose given $\h$ such that $\hamdiff{\h}{1}(\imlag)$ is transverse to $\imlag$.
Define $\chords{\h}$ to be the set of all chords $\chord$ starting and stopping on $\imlag$.
Then define
\begin{displaymath}
  \cf(\imlag,\locsys)=\bigoplus_{\chord\in\chords{\h}} \Hom_{\zz_2}(\locsys_{\chord(0)},\locsys_{\chord(1)}).
\end{displaymath}
If we assume $\imlag$ is graded, then $\cf(\imlag,\locsys)$ can be graded in the usual way, namely the summand $\Hom_{\zz_2}(\locsys_{\chord(0)},\locsys_{\chord(1)})$ is given the grading $\ind(\gamma)$.
One could perhaps define a slightly more general notion of grading but this will suffice for our purposes.
Since we assume that $\imm$ is graded, the next lemma implies that $\imlag$ is in fact also graded.
\begin{lemma}
  \label{lemma:2}
  $\imlag$ can be graded if and only if $\imm$ can be graded.
  In particular, any grading on $\imm$ is invariant with respect to the group of deck transformations of the covering space $\lag\to\imlag$.
\end{lemma}
\begin{proof}
  If $\imlag$ is graded, then $\imm$ can be graded by pulling back the grading function on $\imlag$.
  The pullback is obviously invariant with respect to deck transformations.

  Conversely, suppose that $\imm$ has a given grading.
  Then $\imm_*(\pi_1(\lag))$ is contained in the kernel of the map $(\Det_{\hvf}^2)_*:\pi_1(\imlag)\to\pi_1(S^1)$.
  Since $|\pi_1(\imlag):\imm_*(\pi_1(\lag))|=|\lag:\imlag|<\infty$, it follows that $(\Det_{\hvf}^2)_*$ is the zero map, and hence $\imlag$ can be graded.
  By the first part of the proof there is thus an invariant grading on $\imm$.
  Since any two gradings on $\imm$ differ by an integer, the given grading on $\imm$ is invariant also.
\end{proof}

$A_\infty$ structure maps are defined in the following way: For $x_i\in\Hom(\locsys_{\chord_i(0)},\locsys_{\chord_i(1)})$, define
\begin{displaymath}
  \m{k}(x_1,\ldots,x_k)=\sum_{\chord_0}\sum_{[u]\in\holpolys{\chord_0,\ldots,\chord_k}}\hol(u;x_1,\ldots,x_k).
\end{displaymath}
Here the outer sum is over all $\chord_0$ such that $\ind\chord_0=\ind\chord_1+\cdots+\ind\chord_k+2-k$.
$\hol(u;x_1,\ldots,x_k)$ is defined in the following way.
For $y\in\locsys_{\chord_0(0)}$, parallel translate $y$ over the boundary arc of $u$ between the marked points $\zi_0$ and $\zi_1$; call the result $y_1\in\locsys_{\chord_1(0)}$.
Then $x_1(y_1)\in\locsys_{\chord_1(1)}$.
Parallel translate this vector over the boundary arc of $u$ between $\zi_1$ and $\zi_2$ to get $y_2\in\locsys_{\chord_2(0)}$.
Apply $x_2$ to get $x_2(y_2)\in\locsys_{\chord_2(1)}$.
Now repeat the process until the clockwise side of $\zi_0$ is reached.
The resulting vector is $\hol(u;x_1,\ldots,x_k)(y)\in\locsys_{\chord_0(1)}$, the image of $y$ under the holonomy map.

The result is that $(\cf(\imlag,\locsys),\sett{\m{k}})$ is an $A_\infty$ algebra; more generally $(\imlag,\locsys)$ can be viewed as an object of the Fukaya category.
Note that for any exact embedded Lagrangian $\lag'$ we can take $\locsys'=\lag'\times\zz_2$ and then $\lag'$ and $(\lag',\locsys')$ are the same object of the Fukaya category.

Now consider $\imm:\lag\to\imlag$.
Pick a basepoint $x_0\in\imlag$ and let $B=\imm^{-1}(x)$.
By parallel transport, for each loop $c\in\pi_1(\imlag)$, the covering space $\imm$ defines a bijection $\hol(c):B\to B$.
Let $\rho:\pi_1(\imlag)\to \Aut_{\mathrm{Set}}(B)$ be the homomorphism defined by $\rho(c)=(\hol(c))^{-1}$.
Let $V$ be the $\zz_2$ vector space formally defined to have basis $B$.
Then we can view $\Aut_{\mathrm{Set}}(B)\subset \Aut_{\zz_2}(V)$, and hence view $\rho$ as a representation into $\Aut_{\zz_2}(V)$.
Thus we get a local system with fiber $V$ which we denote as $\locsys_\imm$.

\begin{theorem}
  \label{thm:6}
  Let $\imm:\lag\to\imlag$ be a covering space and assume that $\imm$ (and hence $\imlag$) is exact.
  Then $\imm$ and $(\imlag,\locsys_\imm)$ are isomorphic objects of the Fukaya category.
  In fact, their $A_\infty$ algebras are isomorphic in an obvious way if the same perturbation data is used to define both.
  Moreover,
  \begin{displaymath}
    \hf(\imm)\cong\hf(\imlag,\locsys_\imm)\cong \mathrm{H}^*_{sing}(\imlag,\locsys_\imm^*\otimes\locsys_\imm)\cong \mathrm{H}^*_{sing}(\lag\amalg\brjmps,\zz_2)
  \end{displaymath}
  as graded vector spaces.
  Here the second to last group denotes singular cohomology with coefficients in the local system $\locsys_\imm^*\otimes\locsys_\imm$.
  Note also that $\lag\amalg\brjmps\cong \lag\times_{\imlag}\lag$ by equation \eqref{eq:41}.
\end{theorem}
\begin{proof}
  Assume that the same $\h$ is used to define $\cf(\imm)$ and $\cf(\imlag,\locsys_\imm)$.
  Then the map
  \begin{displaymath}
    \tilde\chord=(\chord,p_0,p_1) \mapsto p_0^*\otimes p_1\in\Hom(\locsys_{\imm,\chord(0)},\locsys_{\imm,\chord(1)})
  \end{displaymath}
  induces a bijection $\cf(\imm)\cong \cf(\imlag,\locsys_\imm)$.
  (Here we are using the fact that, for each $x\in\imlag$, $\imm^{-1}(x)$ can be canonically identified with a basis of the fiber $\locsys_{\imm,x}$.
  Note also that $\locsys_\imm\cong \locsys_\imm^*$.)
  Moreover, if the same perturbation data is used to define both $A_\infty$ algebras then this bijection intertwines the two sets of $A_\infty$ operations.
  Hence the $A_\infty$ algebras are isomorphic.

  It follows that $\hf(\imm)\cong\hf(\imlag,\locsys_\imm)$.
  The isomorphism $$\hf(\imlag,\locsys_\imm)\cong \mathrm{H}_{sing}(\imlag,\locsys_\imm^*\otimes\locsys_\imm)$$ is well-known and follows because $\imlag$ is exact.
  We sketch a proof here; actually, we will sketch $\hf(\imm)\cong \mathrm{H}_{sing}(\imlag,\locsys_\imm^*\otimes\locsys_\imm)$.

  First, consider just $\imlag$ with no local system.
  Take a $C^2$-small time independent Hamiltonian $H$ and let $h=H|\imlag$, and let $\phi$ be the time-1 flow of $X_H$.
  Let $J$ be a time independent almost complex structure.
  For $H$ small enough, the Hamiltonian chords on $\imlag$ correspond to the intersection points $\phi(\imlag)\cap\imlag$ (see the paragraph after Definition \ref{dfn:7}), and these in turn correspond to the critical points of $h$.
  If in addition $h$ has non-degenerate critical points, then the intersection will be transverse.
  In this situation, Floer showed in \cite{MR1001276} that holomorphic strips which satisfy Floer's equation \eqref{eq:8} correspond to gradient flow lines of $h$ with respect to the metric $\omega(\cdot,J\cdot)$.
  The correspondence is easy to describe: the gradient flow line $x(s)$ (which satisfies $\dot x(s)=\nabla h(x(s))=J X_H(x(s))$), corresponds to the strip $u(s,t)=x(s)$.
  It follows that the Morse cohomology of $\imlag$ with respect to the function $-h$ is isomorphic to the Floer cohomology of $\imlag$ with respect to $H$.
  Indeed, if $p_0$ and $p_1$ are critical points of $-h$, then the matrix element $\langle p_0,\delta_{\mathrm{Morse}}p_1\rangle$ for the Morse cohomology differential equals the number of $-\nabla (-h)=\nabla h$ flow lines from $p_0$ to $p_1$.
  By the above correspondence, the same is true of the Floer differential.
  By Lemma \ref{lemma:11}, this isomorphism is a graded isomorphism.

  Now we consider the Floer cohomology of $\imm$, with $J,H$ as above, and the singular cohomology of the local system $\locsys_{\imm}^*\otimes\locsys_{\imm}$ on $\imlag$.
  We will use the fact that $\lag\amalg\brjmps$ can be identified with the fiber product $\lag\times_{\imlag} \lag$, see equation \eqref{eq:41}.
  First note that the generators of $\cf(\imm)$ can be identified with the critical points of the function $\tilde h =h\circ\imm:\lag\times_{\imlag}\lag\to\rr$ via the chord $\tilde\chord=(\chord,p_0,p_1)$ corresponds to the critical point $(p_0,p_1)\in\lag\times_{\imlag}\lag$.
  Second, strips satisfying Floer's equation \eqref{eq:8} with boundaries on $\imm$ (i.e. the ones used to compute $\hf(\imm)$) are nothing more than strips on $\imlag$ with bottom and top boundary lifts to $\lag$.
  The strip on $\imlag$ is a gradient flow line of $h$.
  The bottom and top boundary lifts together can be thought of as giving a lift to $\lag\times_{\imlag}\lag$; if the bottom lift is $\ell_0$ and the top lift is $\ell_1$, think of this as the lift $(\ell_0,\ell_1)$ to $\lag\times_{\imlag}\lag$.
  It follows that $\hf(\imm)$ is isomorphic to the Morse cohomology of $-\tilde h$.
  It is a graded isomorphism because Lemma \ref{lemma:2} implies that the index of $\tilde\chord=(\chord,p_0,p_1)$ equals the index of $\chord$, and Lemma \ref{lemma:11} implies that the index of $\chord$ equals the index of the corresponding critical point of $-h$, which in turn equals the index of the critical point of $-\tilde h$.
  The Morse cohomology is isomorphic to the singular cohomology of $\lag\times_{\imlag}\lag$.
  A simple adaptation of Example 3H.2 in \cite{hatcher-at} shows that the singular cohomology of $\lag\times_{\imlag}\lag$ is isomorphic to the singular cohomology of $\imlag$ with coefficients in $\locsys_\imm^*\otimes\locsys_\imm$.
\end{proof}

We now examine the invariance of $\imm$ under exact, non-Hamiltonian deformations arising from Morse functions.
To begin with, let $h:\lag\to\rr$ be a function.
Given $h$, define $h_\brjmps:\brjmps\to\rr$ by $h_\brjmps(p,q)=-h(p)+h(q)$, and assume that $h$ is such that $h_\brjmps$ is Morse.
Let $\dep{\imm}{t}{t}$ be an exact deformation of $\imm_0:=\imm$ that satisfies $dh=\imm_t^*\sympl(\pd{\imm_t}{t},\cdot)$, so $\dep{\imm}{t}{t}$ is a deformation of $\imm$ in the direction of $-J\nabla h$ (modulo reparametrization of the domain).
Let $f_t:\lag\to\rr$ be as in equation \eqref{eq:35}, so it satisfies $df_t=\imm_t^*\oneform$.
The condition that $h_\brjmps$ is Morse implies that $\imm_t$ has only transverse double points for small $t\neq 0$.
Moreover, for any fixed $t\neq0$, there is a bijective correspondence between $\brjmps_t=\set{(p,q)\in\lag\times\lag}{\imm_t(p)=\imm_t(q),\ p\neq q}$ and the critical points of $h_\brjmps$.

The immersions $\imm_t$ for $t\neq 0$ may or may not satisfy the positivity condition.
We will show that if they do, then they are all quasi-isomorphic objects of the Fukaya category; in particular, they are all quasi-isomorphic to $\imm$.
This is quite interesting because the topology of $\imm$ and $\imm_t$ are different; for instance, by taking different functions it could be possible to construct immersions with different numbers of self-intersection points.

Before giving the proof we analyze the positivity condition a little further.
First, consider the index of self-intersection points of $\imm_t$.
\begin{lemma}
  \label{lemma:1}
  Let $L_0=\rr\subset\cc$ and $L_t=e^{it}\cdot L_0$.
  Then $\dep{L}{t}{t}$ is a deformation of $L_0$ corresponding to the function $h:L_0\to\rr$, $h(x)=-\frac{1}{2}x^2$.
  Equip $\cc$ with the complex volume form $dz$ and give $L_0$ a grading $\theta_0$, and then grade $L_t$ with $\theta_t$ by continuation.
  Then for $t>0$, $\ind((T_0L_0,\theta_0),(T_0L_t,\theta_t))=1$.
  
  If instead $L_t=e^{-it}\cdot L_0$, then the function is $h(x)=\frac{1}{2}x^2$ and $$\ind((T_0L_0,\theta_0),(T_0L_t,\theta_t))=0$$ for $t>0$.
\end{lemma}
\begin{proof}
  This is a simple calculation using equation \eqref{eq:7}.
\end{proof}
\begin{corollary}
  \label{cor:2}
  For the deformation $\dep{\imm}{t}{t}$, it follows that, for $t>0$,
  \begin{displaymath}
    \ind (p_t,q_t)=\ind (h_\brjmps,(p_0,q_0)),
  \end{displaymath}
  where the right-hand side denotes the Morse index of the critical point $(p_0,q_0)\in \brjmps$ which corresponds to $(p_t,q_t)\in\brjmps_t$.
\end{corollary}
\begin{proof}
  Consider first the simple example of $\lag=\rr\amalg\rr$, $\imlag=\rr\subset\cc$.
  Let $p,q\in\lag$ denote the two preimages of $0\in\imlag$, and let $x$ denote the coordinate on the first copy of $\rr$ in $\lag$ (the one containing $p$), and $y$ the coordinate on the second copy.
  Take $h:\lag\to\rr$ to be $h(x)=0,h(y)=y^2/2$.
  The component $\set{(x,y)}{x=y}$ of $\brjmps$ contains $(p,q)$, and on this component $h_\brjmps=x^2/2=y^2/2$.
  The flow of $-i\nabla h$ fixes the first branch of $\imlag$ and rotates the second branch in the clockwise direction.
  Thus, assuming both branches have the same grading to begin with, after perturbation the index of $(p_t,q_t)$ will be $0$ by Lemma \ref{lemma:1}, which is also the Morse index of $h_\brjmps$ at $(p,q)$.
  Taking instead $h(x)=0,h(y)=-y^2/2$ will make the indices be $1$.
  Thus the result holds in this example.

  By taking a direct sum of one dimensional examples, the result holds for $\lag=\rr^n\amalg\rr^n$ and $\mnfld=\cc^n$, assuming that both branches of $\lag$ have the same grading to begin with.
  For general $\lag$ and $\mnfld$, again assuming that branches of $\lag$ have the same grading, the result follows from the local calculation by taking a Weinstein neighborhood.

  It remains to show that all branches of $\lag$ have the same grading; that is, that $\laggrad(p)=\laggrad(q)$ whenever $\imm(p)=\imm(q)$.
  This follows from Lemma \ref{lemma:2}.
\end{proof}

Next consider the energy $E(p_t,q_t)$ of $(p_t,q_t)\in\brjmps_t$.
Recall that $E(p_0,q_0)=0$ since $\imm_0$ is a covering of an embedded exact Lagrangian.
By equation \eqref{eq:36} 
\begin{eqnarray*}
  \pd{}{t}\biggr|_{t=0}E(p_t,q_t) & = &  -df_0(\dot p_0)+df_0(\dot q_0)-h(p_0)+h(q_0)\\
  &&-\oneform\left(\pd{\imm_t}{t}\biggr|_{t=0}(p_0)\right)+\oneform\left(\pd{\imm_t}{t}\biggr|_{t=0}(q_0)\right)\\
&=& -\oneform\left(D\imm_0(\dot p_0)+\pd{\imm_t}{t}\biggr|_{t=0}(p_0)\right) +\oneform\left(D\imm_0(\dot q_0)+\pd{\imm_t}{t}\biggr|_{t=0}(q_0)\right) \\
&&-h(p_0)+h(q_0)\\
&=&h_\brjmps(p_0,q_0).
\end{eqnarray*}
Here we have used $\imm_0^*\oneform=df_0$ and $\imm_t(p_t)=\imm_t(q_t)$ for all $t$ (or rather, the derivative at $t=0$ of this identity).
Since $h_\brjmps(p,q)=-h_\brjmps(q,p)$, it follows that the global maximum of $h_\brjmps$ is positive and the global minimum is negative.
The Morse index of the global minimum is $0$, therefore the corresponding self-intersection point of $\imm_t$ has index $0$ for small $t>0$.
By the above calculation, it has negative energy.
Likewise, the self-intersection point corresponding to the global maximum of $h_\brjmps$ has index $n$ and positive energy.
Hence, for small $t>0$, this point will satisfy the positivity condition provided $n=\dim\lag\geq 3$.
Similar considerations apply to deforming for time $t<0$: the self-intersection points of $\imm_t$ corresponding to the global min and max of $h_\brjmps$ will automatically satisfy the positivity condition provided $n\geq 3$.
This can be seen by replacing $h$ with $h'=-h$, because moving in negative time for $h$ is the same as moving in positive time for $h'$.
Also, the global min of $h_\brjmps$ becomes the global max of $h'_\brjmps$ and vice versa, and the Morse indices are related by $\ind(h_\brjmps,(p_0,q_0))=n-\ind(h'_\brjmps,(p_0,q_0))$.
More generally, for small $t\neq 0$, $E(p_t,q_t)>0$ if and only if $E(q_{-t},p_{-t})>0$; and $\ind(p_t,q_t)=\ind(q_{-t},p_{-t})$.
Thus $\imm_t$ satisfies the positivity condition if and only if $\imm_{-t}$ does.

We now prove the theorem mentioned before.
Note that this theorem can profitably be combined with Theorem \ref{thm:12}.
\begin{theorem}
  \label{thm:13}
  Let $h$ be as above and let $\dep{\imm}{t}{t}$ be the corresponding family of immersions.
  Assume that for small $t\neq 0$, all the $\imm_t$ satisfy Assumption \ref{ass:1}.
  Then there exists $\epsilon>0$ such that $\imm_t$ and $\imm_0$ are quasi-isomorphic when $|t|<\epsilon$. 
\end{theorem}
\begin{proof}
  Let us change notation and denote $t$ by $\tau$, so $\dep{\imm}{\tau}{\tau}$ is the family of Lagrangian immersions.
  We do this so as not to confuse $\tau$ with the $t$ in families $\dep{H}{t}{t}$ and $\dep{J}{t}{t}$ of time-dependent Floer data.

  The idea of the proof is to count curves with moving Lagrangian boundary conditions to get elements $e_{0,\tau}\in\cf(\imm,\imm_{\tau})$ and $e_{\tau,0}\in\cf(\imm_{\tau},\imm)$ which are quasi-isomorphisms.
  This is the standard method when $\dep{\imm}{\tau}{\tau}$ is a Hamiltonian deformation, and we also used it in Remark \ref{rmk:3} to give an alternative proof to Theorem \ref{thm:12}.
  The difficulty in the present case is that the family of Lagrangian immersions undergoes a change in topology at $\tau=0$.
  In particular, the angles between the branches at a self-intersection point go to $0$ as $\tau\to 0$.
  The compactness theorem of \cite{MR1890078} does not apply in this situation.
  To get around this, we will lift curves to $T^*\lag$ in order to prove compactness.
  We emphasize that the main non-standard point of the proof is the compactness.

  The first step is to define a Lagrangian lift $\lag_\tau$ of $\imlag_\tau$ to $T^*\lag$.
  To do this, fix a Weinstein neighborhood $W$ of $\imlag$ inside $\mnfld$; $W$ may also be viewed as a neighborhood of $\imlag$ inside $T^*\imlag$.
  The covering space map $\imm:\lag\to\imlag$ induces a local symplectomorphism $I:T^*\lag\to T^*\imlag$; let $W'=I^{-1}(W)$.
  Choose $\epsilon>0$ so that $\imlag_\tau=\imm_\tau(\lag)$ is contained inside the Weinstein neighborhood $W$ for $|\tau|\leq\epsilon$.
  Let $\lag_\tau$ be the graph of $-\tau dh$.
  Then $I$ restricts to give an immersion $\lag_\tau\to\mnfld$.
  Let $\Pi:T^*\lag\to\lag$ be the projection, and let $\Pi_\tau=\Pi|\lag_\tau:\lag_\tau\to\lag$.
  Without loss of generality (by Theorem \ref{thm:12}), we may assume that $\imm_\tau=I\circ\Pi_\tau^{-1}$.
  So $\imlag_\tau=\imm_\tau(\lag)=I(\lag_\tau)$, and $\lag_\tau$ is a lift of $\imlag_\tau$ to $T^*\lag$.

  The second step is to define Floer data for each pair $(\imm_\tau,\imm_{\tau'})$ of Lagrangian immersions in $\mnfld$, and corresponding Floer data for each pair of embedded Lagrangians $(\lag_\tau,\lag_{\tau'})$ in $T^*\lag$, with $|\tau|,|\tau'|<\epsilon$.
  First, fix some Floer data $\h,\J$ for the pair of Lagrangians $(\imm,\imm)$, with $\h=0$ outside a compact subset of $W$.
  For the purposes of making the setup easy to visualize, we choose $\h$ as follows.
  Take $\h=\dep{H=H}{t}{t}$ to be time independent and $C^2$-small.
  Also, near $\imlag$ in the Weinstein neighborhood $W$, take $H$ to be constant in the direction normal to $\imlag$.
  Then the flow of $X_{\h}$ will move in the direction normal to $\imlag$, and $\hamdiff{\h}{1}(\imlag)$ will be the graph of an exact 1-form over $\imlag$.
  We may assume that the time-1 chords of $X_{\h}$ from $\imlag$ to $\imlag$ are constant chords with images equal to the points of $\hamdiff{\h}{1}(\imlag)\cap\imlag$.
  Take $\epsilon$ to be small compared to $\h$, and assume $\h$ is chosen so that the intersection points $\hamdiff{\h}{1}(\imlag)\cap\imlag$ are away from the self-intersection points of $\imlag_\tau$ for $\tau\neq 0$.
  Then, for each pair of numbers $\tau,\tau'$, perturb $\J$ slightly to get $\J_{\tau,\tau'}$ which makes $\h,\J_{\tau,\tau'}$ regular Floer data for the pair $(\imm_\tau,\imm_{\tau'})$.

  Next define Floer data for each pair of Lagrangians $(\lag_\tau,\lag_{\tau'})$ inside $T^*\lag$ by lifting the Floer data for the pair $(\imm_\tau,\imm_{\tau'})$ as follows.
  Let $H'=H\circ I:W'\to\rr$.
  Since $H$ vanishes near the boundary of $W$, $H'$ can be extended by $0$ to all of $T^*\lag$.
  Then take $\h'=\dep{H'}{t}{t}$ as the time-independent Hamiltonian on $T^*\lag$.
  The complex structure $\J_{\tau,\tau'}$ can also be lifted to a complex structure on $W'$.
  Extend it in any reasonable way to all of $T^*\lag$ (for example, make it contact type outside of a compact subset), call the result $\J_{\tau,\tau'}'$.
  The result is that we get Floer data $\h',\J_{\tau,\tau'}$ for each pair of Lagrangians $(\lag_\tau,\lag_{\tau'})$ in $T^*\lag$.
  
  The third step is to define some moduli spaces.
  For the rest of the proof we assume $\tau>0$ (the $\tau<0$ case is similar).
  We will need to consider moduli spaces with one Type I point and also moduli spaces of strips.
  Consider the former moduli spaces first.
  Let $S=\disc\setminus\sett{z_0=-1}$; view $z_0=-1$ as a fixed incoming Type I point.
  Let $\lambda$ be a coordinate for $\bdy S\cong\rr$ such that $\lambda$ increases in the counterclockwise direction around $\bdy S$.
  Let $\chi:\bdy S\to[0,\tau]$ be smooth, increasing, and surjective with $\dot\chi(\lambda)=0$ for $|\lambda|$ large.  
  Pick some generic perturbation data $\pdK{1}_{0,\tau}$ and complex structure $\J^1_{0,\tau}$ (so $\pdK{1}_{0,\tau}$ is a 1-form on $S$ with values in the Hamiltonian functions on $\mnfld$, and $\J^1_{0,\tau}$ is an $S$-dependent almost complex structure on $\mnfld$) that agrees with the Floer data $\h,\J_{0,\tau}$ on some strip-like coordinates near $z_0$.
  We may choose $\pdK{1}_{0,\tau}$ so that $\pdK{1}_{0,\tau}=0$ outside the Weinstein neighborhood $W$, and also $\pdK{1}_{0,\tau}(\xi)|\imlag\equiv 0$ for $\xi\in T\bdy S$.

  A generator of $\cf(\imm_0,\imm_\tau)$ is a tuple $\tilde\chord=(\chord,p_0,p_1)$ with $\chord$ a time-1 $X_{\h}$ chord connecting $\imlag$ to $\imlag_\tau$, and $p_0,p_1\in\lag$ with $\imm(p_0)=\chord(0)$, $\imm_\tau(p_1)=\chord(1)$.
  For a generator $\tilde\chord$, let $\holpolys{\tilde\chord;\sett{\imm_{\chi(\lambda)}}_\lambda}$ consist of equivalence classes of tuples $(u,\ell)$ with $u:\disc\to \mnfld$ a finite energy map such that $(du-X_{\pdK{1}_{0,\tau}})^{0,1}=0$, $u$ converges to the chord $\chord$ at the point $z_0$, and $u$ satisfies the moving Lagrangian boundary conditions $u(\lambda)\in\imlag_{\chi(\lambda)}$ for $\lambda\in\bdy S$.
  The lift $\ell:\bdy S\to\lag$ satisfies $u(\lambda)=\imm_{\chi(\lambda)}(\ell(\lambda))$, $\ell(-\infty)=p_0$ and $\ell(+\infty)=p_1$.

  Next, define a similar moduli space $\holpolys{\chord';\sett{\lag_{\chi(\lambda)}}_\lambda}$ of curves into $T^*\lag$ with moving Lagrangian boundary conditions.
  For perturbation data, use data $\pdK{1'}_{0,\tau}$ and $\J^{1'}_{0,\tau}$ which in $W'$ is a lift through $I$ of the data $\pdK{1}_{0,\tau}$, $\J^{1}_{0,\tau}$.
  Near the marked point $z_0$, the data should agree with the previously defined Floer data $\h',\J_{0,\tau}'$.
  Since $\lag$ and $\lag_\tau$ are embedded, $\chord'$ simply denotes a time-1 $X_{\h'}$ chord from $\lag$ to $\lag_\tau$, and the lifts $\ell$ are not needed as part of the data.

  Next, we define two moduli spaces of strips with non-moving Lagrangian boundary conditions, analagous to Definition \ref{dfn:10} but with no Type II points.
  The first moduli space is $\holpolys{\tilde\chord_-,\tilde\chord_+;\imm_0,\imm_\tau}$.
  It consists of pairs $(u,\ell)$, where $u$ is a strip into $\mnfld$ with bottom boundary on $\imlag$ and top boundary on $\imlag_\tau$, and $\ell$ is a lift to $\lag$ of the boundary conditions.
  The map $u$ satisfies Floer's equation with the Floer data $\h,\J_{0,\tau}$ for the pair $(\imm_0,\imm_\tau)$.
  The second moduli space is $\holpolys{\chord_-',\chord_+';\lag,\lag_\tau}$.
  It consists of strips $u'$ into $T^*\lag$ with bottom boundary on $\lag$ and top boundary on $\lag_\tau$.
  The map $u'$ satisfies Floer's equation with the Floer data $\h',\J_{0,\tau}'$.

  The fourth step is to define bijections between the moduli spaces for $(\imm,\imm_\tau)$ and those for $(\lag,\lag_\tau)$.
  Consider $\holpolys{\tilde\chord;\sett{\imm_{\chi(\lambda)}}_\lambda}$ and $\holpolys{\chord';\sett{\lag_{\chi(\lambda)}}_\lambda}$ first.
  Fix a chord $\tilde\chord=(\chord,p_0,p_1)$ and consider an element $(u,\ell)$ from the moduli space $\holpolys{\tilde\chord;\sett{\imm_{\chi(\lambda)}}_\lambda}$.
  By an analog of equation \eqref{eq:26}, the energy of $u$ is
  \begin{eqnarray}
    \label{eq:100}
    E(u)&=&\frac{1}{2}\int |du-X_{\pdK{1}_{0,\tau}}|^2=\action(\tilde\chord)-\int_SR_{\pdK{1}_{0,\tau}}(u)\\
    &&-\int_\rr h(\chi(\lambda),\ell(\lambda))\dot\chi(\lambda)d\lambda\nonumber\\
    &=&-\int_0^1(\chord^*\oneform+H(\chord(t)))dt-f_0(p_0)+f_\tau(p_1)-\int_S R_{\pdK{1}_{0,\tau}}(u)\nonumber\\
    &&-\int_\rr h(\chi(\lambda),\ell(\lambda))\dot\chi(\lambda)d\lambda.\nonumber
  \end{eqnarray}
  Here $f_0,f_\tau:\lag\to\rr$ are the primitives for $\imm^*\oneform,\imm_\tau^*\oneform$ defined by \eqref{eq:35}.
  By construction of $\h$, the chord $\chord$ is short and approximately equal to a constant chord.
  Also, because $\imlag$ is exact, $f_0$ is the pullback through $\imm$ of a function on $\imlag$.
  Moreover, $f_\tau$ is a deformation of $f_0$.
  Thus $E(u)$ can be made arbitrarily small by taking $\epsilon,\h,\pdK{1}$ small.
  By standard Gromov monotonicity, if $u$ leaves the Weinstein neighborhood $W$, then it must have energy larger than some fixed constant $E_0>0$ (because the perturbation data $\pdK{1}_{0,\tau}$ is $0$ outside of $W$).
  Thus, by taking the data small enough, $u$ cannot leave $W$.
  Let $\chord'$ be the time-1 $X_{\h'}$ chord in $W'\subset T^*\lag$ which maps to $\chord$ under the projection $W'\to W$ and satisfies $\chord'(1)=\Pi_\tau^{-1}(p_1)$.
  Since the image of $u$ lies inside $W$, and $u$ converges to $\chord$ at $z_0$, $u$ can be lifted to a map $u':S\to W'$ which converges to $\chord'$ at $z_0$.
  By construction, $I u'(\lambda)=u(\lambda)=\imm_{\chi(\lambda)}(\ell(\lambda))=I\Pi_{\chi(\lambda)}^{-1}(\ell(\lambda))$.
  Since $I$ is a covering map and $ u'(+\infty)=\Pi_{\chi(+\infty)}^{-1}(p_1)=\Pi_{\chi(+\infty)}^{-1}(\ell(+\infty))$, it follows that $ u'(\lambda)=\Pi_{\chi(\lambda)}^{-1}(\ell(\lambda))$ for all $\lambda$.
  In particular, $ u'(\lambda)\in \lag_{\chi(\lambda)}$.
  We thus get $u'\in\holpolys{\chord';\sett{\lag_{\chi(\lambda)}}_\lambda}$.

  Conversely, suppose given a curve $u'\in\holpolys{\chord';\sett{\lag_{\chi(\lambda)}}_\lambda}$.
  An energy calculation analogous to the one above shows that $u'$ has small energy, and hence by monotonicity cannot leave the neighborhood $W'$.
  Let $p_0=\Pi u'(-\infty)$, $p_1=\Pi u'(+\infty)$ for $\pm\infty\in\bdy S$, and let $\tilde\chord=(I\chord,p_0,p_1)$.
  Then $\tilde u=(u=I u',\ell= u'|\bdy S)$ is an element of the moduli space $\holpolys{\tilde\chord;\sett{\imm_{\chi(\lambda)}}_\lambda}$.
  The lifting and pushing down constructions are inverses to each other, and thus we get a bijection $\holpolys{\tilde\chord;\sett{\imm_{\chi(\lambda)}}_\lambda}\cong\holpolys{\chord';\sett{\lag_{\chi(\lambda)}}_\lambda}$.

  Next consider the moduli spaces of strips.
  A bijection can be constructed in a similar way, but with one minor tweak, as follows.
  Let $\holpolys{\chord_-',\chord_+';\lag,\lag_\tau}^+$ be the set of pairs $(u',p_{+,0})$  with $u'\in\holpolys{\chord_-',\chord_+';\lag,\lag_\tau}$ and $p_{+,0}\in\lag$ satisfying $\imm(p_{+,0})=\imm(\chord'_+(0))$.
  Consider an element $(u,\ell=\ell^0\amalg\ell^1)$ from $\holpolys{\tilde\chord_-,\tilde\chord_+;\imm_0,\imm_\tau}$.
  Write $\tilde\chord_{\pm}=(\chord_\pm,p_{\pm,0},p_{\pm,1})$ with $\chord_\pm$ chords in $\mnfld$ from $\imlag$ to $\imlag_\tau$, and $p_{\pm,0},p_{\pm,1}\in\lag$ satisfying $\imm(p_{-,0})=\chord_-(0),\imm(p_{+,0})=\chord_+(0),\imm_\tau(p_{-,1})=\chord_-(1), \imm_\tau(p_{+,1})=\chord_+(1)$.
  Let $\chord_\pm'$ be the lift of $\chord_\pm$ with $\chord_\pm'(1)=\Pi_\tau^{-1}(p_{\pm,1})$.
  Then $u$ can be lifted to get the element $u'\in\holpolys{\chord'_-,\chord'_+;\lag,\lag_\tau}$.
  Note that it is not necessarily true that $\chord'_-(0)=p_{-,0}$ and $\chord'_+(0)=p_{+,0}$ (this is the reason for encoding the extra data $p_{+,0}$).
  The correspondence $\holpolys{\tilde\chord_-,\tilde\chord_+;\imm_0,\imm_\tau}\to\holpolys{\chord_-,\chord_+;\lag,\lag_\tau}^+$ is given by $(u,\ell)\mapsto (u',p_{+,0})$.

  Conversely, given an element $(u',p_{+,0})\in \holpolys{\chord_-',\chord_+';\lag,\lag_\tau}^+$, let $u=Iu'$, $\ell^1(s)=\Pi u'(s,1)$, and let $\ell^0$ be defined by $\ell^0(+\infty)=p_{+,0}$ and $\imm\ell^0(s)=\imm u'(s,0)$.
  Let $\tilde\chord_-=(I\chord_-',\ell^0(-\infty),\ell^1(-\infty))$ and $\tilde\chord_+=(I\chord_+',p_{+,0},\ell^1(+\infty))$.
  Then $\tilde u=(u,\ell=\ell^{0}\amalg\ell^{1})$ is an element of the moduli space $\holpolys{\tilde\chord_-,\tilde\chord_+;\imm,\imm_\tau}$.
  We thus get bijections $\holpolys{\chord_-',\chord_+';\lag,\lag_\tau}^+\cong \holpolys{\tilde\chord_-,\tilde\chord_+;\imm,\imm_\tau}$.

  The fifth step is to show that the moduli spaces $\holpolys{\tilde\chord;\sett{\imm_{\chi(\lambda)}}_\lambda}$ are regular for generic data, and to explain why Gromov compactness holds.
  (For the moduli spaces of strips both of these things are standard because there are no moving Lagrangian boundary conditions.)
  For regularity, note that an element $(u,\ell)$ in the moduli space will be smooth up to the boundary, even at the point in the boundary where the corresponding Lagrangian boundary condition $\imm_{\chi(\lambda)}$ collapses onto $\imm_0$, because of the existence of $\ell$.
  Hence the moduli space can be set up in the usual functional analytic framework and standard methods imply that it will be regular.
  In particular, it will be a smooth manifold of dimension $\ind\tilde\chord$.
  Note that the moduli space $\holpolys{\chord',\sett{\lag_{\chi(\lambda)}}_\lambda}$ will also be regular.
  The reason is that a curve is regular if and only if the image of the linearized operator is surjective; the map $I$ induces an obvious isomorphism between the linearized operators of corresponding curves in the two moduli spaces, hence a curve is regular if and only if its corresponding curve is regular.
  
  Now consider Gromov compactness.
  Let $(u_n,\ell_n)$ be a sequence of curves in $\holpolys{\tilde\chord,\sett{\imm_{\chi(\lambda)}}_\lambda}$.
  They can be lifted to a sequence of curves $u_n'$ in $\holpolys{\chord',\sett{\lag_{\chi(\lambda)}}_\lambda}$.
  By usual Gromov compactness for embedded Lagrangians, the sequence $u_n'$ has a subsequence which converges to an element $u_\infty'\in\holpolys{\chord'_1,\sett{\lag_{\chi(\lambda)}}_\lambda}$ along with a broken strip (with bottom boundary on $\lag$, top boundary on $\lag_\tau$) connecting $\chord'$ to $\chord'_1$.
  By exactness there are no disc components.
  The curve $u_\infty'$ can be pushed down to give a curve $(u_\infty,\ell_\infty)\in\holpolys{\tilde\chord_1;\sett{\imm_{\chi(\lambda)}}_\lambda}$.
  Similarly, the broken strip can be pushed down to get a broken strip in $\mnfld$ on $(\imm,\imm_\tau)$.
  Note that in this situation there is no ambiguity for the choice of the lift to $\lag$ of the bottom of the broken strip, because the right-most point of the lift needs to equal $\ell_\infty(-\infty)$.

  The sixth and final step is to complete the proof.
  Since the rest of the details are standard (with the help of the positivity condition), we will only say a few words.
  Define $e_{0,\tau}\in\cf(\imm_0,\imm_\tau)$ analagously to equation \eqref{eq:37} by counting elements of the moduli spaces $\holpolys{\tilde\chord;\sett{\imm_{\chi(\lambda)}}_\lambda}$ with $\ind\tilde\chord=0$, and similarly $e_{\tau,0}\in\cf(\imm_\tau,\imm_0)$.
  (The definition of $e_{\tau,0}$ requires setting up moduli spaces similar to the ones explained above, but with the curves connecting generators of $\cf(\imm_{\tau},\imm)$ instead of $\cf(\imm,\imm_\tau)$.)
  By Gromov compactness and regularity for $0$ dimensional moduli spaces, $e_{0,\tau}$ and $e_{\tau,0}$ are well-defined.
  By Gromov compactness and regularity for $1$ dimensional moduli spaces, these elements are $\m{1}$ closed.
  To deduce that $\m{2}(e_{0,\tau},e_{\tau,0})$ is cohomologous to the unit $e_{0,0}\in\cf(\imm,\imm)$, one uses a gluing theorem to conclude that $\m{2}(e_{0,\tau},e_{\tau,0})$ is a count of elements from moduli spaces $\holpolys{\tilde\chord;\sett{\imm_{\chi_1(\lambda)}}_\lambda}$, with $\chi_1$ determined by concatenation of previous boundary conditions, and satisfying $\chi_1(\pm\infty)=0$.
  Then one considers families of moduli spaces $\holpolys{\tilde\chord;\sett{\imm_{\chi_r(\lambda)}}_\lambda}$  parameterized by $r\in[0,1]$, with $\chi_0\equiv 0$.
  The moduli space with $r=0$ defines the unit $e_{0,0}$, and hence the family of moduli spaces can be used to show that $\m{2}(e_{0,\tau},e_{\tau,0})$ is cohomologous to $e_{0,0}$.
  A similar construction shows that the unit $e_{\tau,\tau}\in\cf(\imm_\tau,\imm_\tau)$ is cohomologous to $\m{2}(e_{\tau,0},e_{0,\tau})$.
  In this case, $\chi_0\equiv \tau$, and $\chi_r(\pm\infty)=\tau$ for all $r$.
  
  This is essentially the end of the proof but we need to point out a few subtleties about the moduli spaces $\holpolys{\tilde\chord;\sett{\imm_{\chi_r(\lambda)}}_\lambda}$ before we are done.
  The first point is that we need to use the lifting procedure described in step four above in order to prove Gromov compactness (again, this is because of moving Lagrangian boundary conditions).
  In the case where the moduli space describes curves attached to a generator of $\cf(\imm,\imm)$, there is no difficulty in uniquely lifting a curve.
  However, in the case that a curve $(u,\ell)$ attaches to a generator of $\cf(\imm_\tau,\imm_\tau)$, it is not immediately clear that a lift $u'$ which, near $z_0$, has top boundary on $\lag_\tau$ will also have bottom boundary on $\lag_\tau$.
  The reason this is not clear is that $I^{-1}(\imlag_\tau)$ contains $\lag_\tau$ as a component, but also has several other components.
  However, the existence of $\ell$ and the contractibility of the domain of $u$ imply that this will not be a problem.

  The second point is that in order to apply the lifting procedure (to prove Gromov compactness) we need to ensure that curves in $\holpolys{\tilde\chord;\sett{\imm_{\chi_r(\lambda)}}_\lambda}$ stay inside the Weinstein neighborhood $W$.
  Before, we showed this was the case by arguing that the energy given by equation \eqref{eq:100} could be made small by taking the parameters to be small.
  In the present situation, the term  $h(\chi(\lambda),\cdot)\dot\chi(\lambda)$ appearing in the last integral of \eqref{eq:100} is replaced by a function $j:\lag\times\bdy S\to\rr$ satisfying $d(j(\lambda,\cdot))=\imm_{\chi_r(\lambda)}^*\sympl(\pd{}{\lambda}(\imm_{\chi_r(\lambda)}),\cdot)$.
  Thus, by defining $\chi_r$ appropriately, the energy can still be made small.
\end{proof}

\section{More than double self-intersections}
\label{sec:more-than-double}
In Section \ref{sec:case-when-imbrjmps} we showed that when $\imm:\lag\to\imlag$ is a covering space, the Floer theory of $\imm$ can be interpreted as the Floer theory of $(\imlag,\locsys_\imm)$, where $\imlag$ is viewed as an embedded exact Lagrangian submanifold of $\mnfld$ and $\locsys_\imm$ is a local system corresponding to the cover $\imm$.
In particular, we did not require that immersed points are only double points.

In fact, a similar consideration applies to any clean immersion $\imm$; in other words, we do not need to require that the preimage of a point $x\in\imlag$ consists of only one or two points.
A more general definition of clean immersion than that given in Definition \ref{dfn:23} is the following: $\imm:\lag\to\imlag\subset M$ is a clean immersion if $\imm$ is an immersion and the fiber product $\lag\times_{\imlag}\lag=\set{(p,q)\in\lag\times\lag}{\imm(p)=\imm(q)}$ is a smooth submanifold of $\lag\times\lag$ such that $\imm_*(T_{(p,q)} \lag\times_{\imlag}\lag) =\imm_*T_p\lag\cap \imm_*T_q\lag$  whenever $(p,q)\in \lag\times_{\imlag}\lag$.
Note that $\lag\times_{\imlag}\lag$ always contains a distinguished diagonal component which is diffeomorphic to $\lag$ (regardless of whether or not the fiber product is a manifold).
The union of the other components is $\brjmps$.
The Gromov compactness theorem proved in \cite{MR1890078} holds for this more general definition of clean immersion.

First consider the case that $\brjmps$ is $0$-dimensional.
The proof of Lemma \ref{lemma:14} can be easily modified to show that $\imlag$ is totally geodesic for some metric.
Then all other proofs go through in the same way to show that the $A_\infty$ algebra $(\cf(\imm),\m{k})$ is well-defined.
Note that if $x\in\imlag$ is a singular point with $d$ branches going through it, then $\brjmps$ has $d(d-1)$ elements which map to $x$ under $\imm$.

If $\brjmps$ is not $0$-dimensional and is more complicated than the type considered in Section \ref{sec:clean-self-inters-1} it is not clear that $\imlag$ is totally geodesic for some metric.
Thus more work needs to be done to set up the Banach spaces of maps; perhaps it can be done in some way by using families of metrics which are domain dependent so that only certain branches of $\imlag$ need to be totally geodesic at any given time.
For any given immersion, if this difficulty can be overcome then $(\cf(\imm),\m{k})$ will be well-defined.

\appendix

\section{Asymptotic Analysis}
\label{sec:asymptotic-analysis}
In this section, we prove some asymptotic estimates that show holomorphic curves have appropriate decay on the strip-like ends for Type II marked points.
The main result is Theorem \ref{thm: asymptotic behavior}, which is used in Propositions  \ref{prop:3} and \ref{prop:8}.
The result is also needed for the corresponding results in the clean intersection case (Section \ref{sec:clean-self-inters-1}).
The hard analysis needed in the proof of Theorem \ref{thm: asymptotic behavior} is a $W^{1,p}$-estimate for holomorphic curves with boundary on an immersed Lagrangian; such an estimate is proved in \cite{MR1890078}.

We begin with a lemma that relates estimates on a disc to estimates on a strip.
Let $S^{+}=\mathbb{R}^{+}\times[0,1]\subset\mathbb{C}$ be the half
strip, and let $D^{+}=\set{w\in\mathbb{C}}{\Ima w\geq0,|w|\leq1}$
be the upper half unit disc. Let $\varphi:S^{+}\to D^{+}\backslash\{0\}$
be the biholomorphic map defined by $w=\varphi(z)=e^{-\pi z}.$ Then
$\varphi^{-1}(w)=-\frac{1}{\pi}\log(w).$ 
\begin{lemma}
\label{lem:f w1p} Suppose a function $f:S^{+}\to\mathbb{R}$ satisfies
$\left\Vert f\circ\varphi^{-1}\right\Vert _{W^{1,p}(D^{+})}<\infty$
for some $p>2.$ Assume $f\circ\varphi^{-1}(0):=\lim_{z\to 0}f\circ\varphi^{-1}(z)=0.$
Then $\left\Vert f\right\Vert _{W^{1,p;\delta}(S^{+})}<\infty$ for
any $\delta\in\left(0,(p-2)\pi\right)$. \end{lemma}
\begin{proof}
By the Sobolev embedding theorem, $f\circ\varphi^{-1}\in C^{0,\mu}(D^{+})$
with $\mu=1-\frac{2}{p}>0$. Since $f\circ\varphi^{-1}(0)=0,$ one
has $\left|f\circ\varphi^{-1}(w)\right|\leq C\left|w\right|^{\mu}$
for some constant $C>0.$ Hence $\left|f(z)\right|\leq C\left|\varphi(z)\right|^{\mu}=Ce^{-\pi\mu s},$
where $s=\mathrm{Re\ } z.$ Then 
\[
\int_{S^{+}}|f(z)|^{p}e^{\delta s}\cdot\frac{\sqrt{-1}}{2}dz\wedge d\bar{z}\leq\int_{S^{+}}Ce^{\left(-\pi p\mu+\delta\right)s}\cdot\frac{\sqrt{-1}}{2}dz\wedge d\bar{z}<\infty
\]
 as long as $-\pi p\mu+\delta<0$; that is, $\delta<\pi p\mu=(p-2)\pi.$
Also 

\begin{eqnarray*}
 \int_{S^{+}}\left|\frac{\partial f}{\partial z}\right|^{p}e^{\delta s}\cdot\frac{\sqrt{-1}}{2}dz\wedge d\bar{z}  & = & \int_{S^{+}}\left|\frac{\partial}{\partial w}(f\circ\varphi^{-1})\frac{\partial\varphi}{\partial z}\right|^{p}e^{\delta s}\cdot\frac{\sqrt{-1}}{2}dz\wedge d\bar{z}\\
 & \leq & C'\int_{D^{+}}\left|\frac{\partial}{\partial w}(f\circ\varphi^{-1})\right|^{p}\left|w\right|^{p-2-\frac{\delta}{\pi}}\cdot\frac{\sqrt{-1}}{2}dw\wedge d\bar{w}\\
 & \leq & C'\left\Vert f\circ\varphi^{-1}\right\Vert _{W^{1,p}(D^{+})}^{p}
\end{eqnarray*}
as long as $p-2-\frac{\delta}{\pi}>0$.
A similar calculation holds for the $L^{p;\delta}(S^+)$-norm of $\pd{f}{\bar z}$. \end{proof}

\begin{lemma}
\label{lem:f pointwise}
There exists a constant $C>0$ such that the following is true:
For any $C^1$ function $f:S^+\to\mathbb{R}$, $p>2$ and $\delta>0$, we have
$$\left|f(s,t)\right|\leq Ce^{-\delta (s-1)/(p+2)}\left\Vert f\right\Vert_{L^{p;\delta}(S^+)}(1+\left\Vert \nabla f \right\Vert_{C^0(S^+)})$$ for all $s\geq0$.
\end{lemma}
\begin{proof}
If $\left\Vert \nabla f\right\Vert_{C^0}=\infty$ then the statement is obvious, so we may assume that $\nabla f$ is $C^0$ bounded.
Let $U=[0,1]^{2}$ be the unit square in $\mathbb{R}^{2}.$ Consider
a $C^{1}$ function $g:U\to\mathbb{R}$.
Suppose $g$ is not constant.
Let $x_{0}\in U$ satisfy $\left|g(x_{0})\right|=\left\Vert g\right\Vert _{C^{0}}>0$
and let $K=\left\Vert \nabla g\right\Vert _{C^{0}}>0.$ Then for
$x\in U$ with $\left|x-x_{0}\right|<\frac{\left|g(x_{0})\right|}{2K},$
we have $\left|g(x)-g(x_{0})\right|\leq K\left|x-x_{0}\right|\leq\frac{\left|g(x_{0})\right|}{2}.$
Thus $\left|g(x)\right|\geq\frac{\left|g(x_{0})\right|}{2}=\frac{\left\Vert g\right\Vert _{C^{0}}}{2}$.
It follows that 
\begin{eqnarray*}
\int_{U}\left|g\right|^{p}  &\geq&  \int_{U\cap B\left(x_{0},\frac{|g(x_{0})|}{2K}\right)}\frac{1}{2^{p}}\left\Vert g\right\Vert _{C^{0}}^{p}
 \geq  \frac{1}{2^{p}}\left\Vert g\right\Vert _{C^{0}}^{p}\min\biggl\{\,\frac{\pi}{4},\frac{\pi}{4} \left(\frac{|g(x_{0})|}{2K}\right)^{2}\biggr\}
\\
&=& \frac{1}{2^{p}}\left\Vert g\right\Vert _{C^{0}}^{p}\min\biggl\{\,\frac{\pi}{4},\frac{\pi\left\Vert g\right\Vert _{C^0}^2}{16K}\biggr\}.
\end{eqnarray*}
(The min appears because $U\cap B(x_0,\frac{|g(x_0)|}{2K})$ contains at least a quarter circle of radius either $1$ or $|g(x_0)|/2K$.)
Hence if $0<\left\Vert \nabla g\right\Vert _{C^{0}}=K\leq K',$
we have either
\begin{eqnarray}
  \label{eq:g C0}
  \left\Vert g\right\Vert _{C^{0}}^{1+2/p}&\leq& \frac{16^{1/p}\cdot 2}{\pi^{1/p}}(K')^{2/p}\left\Vert g\right\Vert _{L^{p}}\leq 32 (K')^{2/p}\left\Vert g\right\Vert_{L^p} \quad\text{or}\\
  \left\Vert g \right\Vert_{C^0}&\leq &\frac{2\cdot 4^{1/p}}{\pi^{1/p}} \left\Vert g \right\Vert _{L^p}\leq 4 \left\Vert g\right\Vert_{L^p}.\nonumber
\end{eqnarray}

If $g$ is constant (so $\left\Vert \nabla g\right\Vert _{C^{0}}=0$), we have 
\begin{equation}
\left\Vert g\right\Vert _{C^{0}}=\left\Vert g\right\Vert _{L^{p}}\label{eq:g C0'}.
\end{equation}

Now consider $g_{n}(s,t):=f(s+n,t)|_{U}.$
By assumption, there exists $K':=\left\Vert \nabla f\right\Vert _{C^0(S^+)}<\infty$ such that $\left\Vert \nabla g_{n}\right\Vert _{C^{0}}\leq K'$
for all $n$.
Also  
\begin{eqnarray*}
\left\Vert g_{n}\right\Vert _{L^{p}(U)}^{p} & = & \intop_{[n,n+1]\times[0,1]}\left|f\right|^{p}\cdot\frac{\sqrt{-1}}{2}dz\wedge d\bar{z}\\
 & \leq & e^{-\delta n}\intop_{[n,n+1]\times[0,1]}\left|f\right|^{p}e^{\delta s}\cdot\frac{\sqrt{-1}}{2}dz\wedge d\bar{z}\\
 & \leq & e^{-\delta n}\left\Vert f\right\Vert _{L^{p;\delta}(S^{+})}^{p}.
\end{eqnarray*}
 Thus $\left\Vert g_{n}\right\Vert _{L^{p}(U)}\leq e^{-\frac{\delta n}{p}}\left\Vert f\right\Vert _{L^{p;\delta}(S^{+})}$
and (\ref{eq:g C0}) and (\ref{eq:g C0'}) (with $g$ replaced by $g_n$) give either
\begin{displaymath}
\left\Vert g_{n}\right\Vert _{C^{0}}^{1+\frac{2}{p}}\leq 32 (K')^{2/p}e^{-\frac{\delta n}{p}}\left\Vert f\right\Vert _{L^{p;\delta}(S^{+})}
\end{displaymath}
 or 
\begin{displaymath}
\left\Vert g_{n}\right\Vert _{C^{0}}\leq 4e^{-\frac{\delta n}{p}}\left\Vert f\right\Vert _{L^{p;\delta}(S^{+})}.
\end{displaymath}
Thus, $\left\Vert g_{n}\right\Vert _{C^{0}}\leq C'(1+K')e^{-\frac{\delta n}{p+2}}\left\Vert f \right\Vert_{L^{p;\delta}(S^+)}$, and hence $$|f(s,t)|\leq Ce^{-\frac{\delta (s-1)}{p+2}}\left\Vert f\right\Vert_{L^{p;\delta}(S^+)}(1+\left\Vert \nabla f \right\Vert_{C^0}).$$
\end{proof}
\begin{lemma}
\label{lem: f grad}
There exists a constant $C>0$ such that the following is true:
For any $C^2$ function $f:S^+\to \rr$, $p>2$ and $\delta>0$ we have:
$$\left|\nabla f(s,t)\right|\leq Ce^{-\delta (s-1)/(p+2)}\left\Vert\nabla f\right \Vert_{L^{p;\delta}(S^+)}(1+\left \Vert \nabla^2 f\right\Vert_{C^0(S^+)}). $$
\end{lemma}
\begin{proof}
Apply Lemma \ref{lem:f pointwise} with $f$ replaced with $\nabla f.$
\end{proof}

We now formulate the main result.
Assume $\iota$ has clean self intersection (see Definition \ref{dfn:23}). Let $u:S^{+}\to M$ be a map such that $v=u\circ \varphi ^{-1}$ satisfies 
\begin{equation}
\label{eqn:bubble}
v_x+J(w,v)(v_y-\hamvf{\pdK{}}(v))=0,
\end{equation} where $w=x+\sqrt{-1}y$ is the coordinate for $D^+\backslash \{0\}$. 
Here $\hamvf{\pdK{}}$ is a domain ($D^+$) dependent Hamiltonian vector field as in Section \ref{sec:moduli-spac-holom} (where we called it  $\hamvf{\pdK{d+1}}$).
In particular, it is smooth at $0$.
We assume $u$ also comes with a boundary lift $l$, i.e. $\iota \circ l |(\mathbb{R}^+ \times \{j\})=u|(\mathbb{R}^+ \times \{j\})$,
for $ j=0,1$. 
Notice that equation (\ref{eqn:bubble}) is the equation that the main component of a holomorphic polygon satisfies near a Type II marked point.
\begin{theorem}
  \label{thm: asymptotic behavior}
  Let $u$ be as above. Then the following
  are equivalent. 
  \begin{enumerate}
  \item[a)] u has finite $L^{2}$-energy; that is, $\int_{S^{+}}|du-\hamvf{\pdK{}}|^{2}<\infty$.
    
  \item[b)] There exists constants $\epsilon>0$ and $C$
    such that $\left\Vert u\right\Vert _{C^{1}([s,\infty)\times[0,1])}\leq Ce^{-\epsilon s}.$ 
    
  \item[c)] For some $\delta>0$, $u\in W^{1,p;\delta}(S^{+},M)$ for all $p\geq 2$.
  \end{enumerate}
\end{theorem}
\begin{proof}
Assume a).
The proof of Theorem A in \cite{MR1849689} shows that $\lim_{s\to\infty} u(s,t)=m_0$ for some $m_0\in\mnfld$, uniformly in $t$, and also $\lim_{s\to\infty}|\del_s u(s,t)|=0$ uniformly in $t$.
Notice that in the proof we only need the standard Gromov's Monotonicity (see Lemma 3.4 in \cite{MR2684508} for example).
Also Gromov's graph trick can be used to get rid of the Hamiltonian term  and make the complex structure be domain independent (as in Chapter 8 of \cite{MR2954391}).

Let $v=u\circ\phi^{-1}$.
By looking at the graph of $v$ again we can get rid of the Hamiltonian term $\hamvf{\pdK{}}$ and also make $J$ domain independent.
Then by Theorem 1.4 in \cite{MR1890078}, $v\in W^{1,p}(D^+)$ for some $p>2$.
Then Lemma \ref{lem:f w1p}, Lemma \ref{lem:f pointwise}, and Lemma \ref{lem: f grad} give us b). 

b) implies c) is obvious, and c) implies a) is also clear because $\hamvf{\pdK{}}$ is smooth at $0$ and hence $\int_{S^+} |\hamvf{\pdK{}}|^2 = \int_{D^+}|\hamvf{\pdK{}}|^2<\infty$.
\end{proof}

\bibliographystyle{amsplain}
\bibliography{mybib}

\href{mailto:galston@math.ou.edu}{galston@math.ou.edu}, 
\href{mailto:bao@math.ucla.edu}{bao@math.ucla.edu}

\end{document}